\documentclass[reqno, final, a4paper]{amsart}
\usepackage{color}
\usepackage[colorlinks=true,allcolors=blue
]{hyperref}
\usepackage{amsmath, amssymb, amsthm}
\usepackage{mathrsfs}
\usepackage{mathtools}
\usepackage[noabbrev,capitalize,nameinlink]{cleveref}
\crefname{equation}{}{}
\usepackage{fullpage}
\usepackage[noadjust]{cite}
\usepackage{graphics}
\usepackage{pifont}
\usepackage{cancel}
\usepackage{tikz}
\usepackage{bbm}
\usepackage[T1]{fontenc}
\usetikzlibrary{arrows.meta}

\usepackage{environ}
\usepackage{framed}
\usepackage{paralist}
\usepackage{url}
\usepackage[linesnumbered,ruled,vlined]{algorithm2e}
\usepackage[noend]{algpseudocode}
\usepackage[labelfont=bf]{caption}
\usepackage{framed}
\usepackage[framemethod=tikz]{mdframed}
\usepackage{pgfplots}
\usepackage{appendix}
\usepackage{graphicx}
\usepackage[textsize=tiny]{todonotes}
\usepackage{tcolorbox}
\usepackage{xcolor}
\usepackage[shortlabels]{enumitem}
\crefformat{enumi}{#2#1#3}
\crefrangeformat{enumi}{#3#1#4 to~#5#2#6}
\crefmultiformat{enumi}{#2#1#3}
{ and~#2#1#3}{, #2#1#3}{ and~#2#1#3}
\allowdisplaybreaks

\let\originalleft\left
\let\originalright\right
\renewcommand{\left}{\mathopen{}\mathclose\bgroup\originalleft}
\renewcommand{\right}{\aftergroup\egroup\originalright}

\crefname{algocf}{Algorithm}{Algorithms}

\crefname{equation}{}{} 
\AtBeginEnvironment{appendices}{\crefalias{section}{appendix}} 

\usepackage[color]{showkeys} 

\colorlet{refkey}{orange!20}
\colorlet{labelkey}{blue!30}

\crefname{algocf}{Algorithm}{Algorithms}

\numberwithin{equation}{section}
\newtheorem{theorem}{Theorem}[section]
\newtheorem{proposition}[theorem]{Proposition}
\newtheorem{lemma}[theorem]{Lemma}
\newtheorem{claim}[theorem]{Claim}
\crefname{claim}{Claim}{Claims}

\newtheorem{corollary}[theorem]{Corollary}

\newtheorem*{question*}{Question}
\newtheorem{fact}[theorem]{Fact}
\crefname{fact}{Fact}{Facts}

\theoremstyle{definition}
\newtheorem{definition}[theorem]{Definition}

\newtheorem*{definition*}{Definition}

\theoremstyle{remark}
\newtheorem{remark}[theorem]{Remark}
\newtheorem*{remark*}{Remark}

\newcommand{\floor}[1]{\left\lfloor #1 \right\rfloor}

\newcommand{\mb}{\mathbb}

\newcommand{\mc}{\mathcal}

\newcommand{\mr}{\mathrm}

\newcommand{\on}{\operatorname}

\renewcommand{\Pr}{\mb P}

\newcommand{\equal}{=}
\newcommand*{\claimproofname}{Proof of claim}
\newenvironment{claimproof}[1][\claimproofname]{\begin{proof}[#1]}{\end{proof}}

\usetikzlibrary{shapes.geometric}
\usepackage{tikz,xcolor,verbatim,hyperref}

\tikzset{
blackvertexv2/.style={circle, draw=black!100,fill=black!100,thick, inner sep=0pt, minimum size= 2.5mm},
blackvertex/.style={circle, draw=black!100,fill=black!100,thick, inner sep=0pt, minimum size= 2.2mm},
blackvertexv3/.style={regular polygon,regular polygon sides=3, draw=black!100,fill=black!100,thick, inner sep=0pt, minimum size= 2.8mm},
greenvertexv2/.style={circle, draw=black!100,fill=green!75!black,thick, inner sep=0pt, minimum size= 2.5mm},
greenvertex/.style={rectangle, draw=black!100,fill=green!25!black ,thick, inner sep=0pt, minimum size= 2.3mm},
greenvertexv3/.style={regular polygon,regular polygon sides=3, draw=black!100,fill=green!100,thick, inner sep=0pt, minimum size= 2.8mm},
bluevertexv2/.style={circle, draw=black!100,fill=blue!100,thick, inner sep=0pt, minimum size= 2.5mm},
bluevertex/.style={rectangle, draw=black!100,fill=blue!100,thick, inner sep=0pt, minimum size= 2.2mm},
bluevertexv3/.style={regular polygon,regular polygon sides=3, draw=black!100,fill=blue!100,thick, inner sep=0pt, minimum size= 2.8mm},
yellowvertexv2/.style={circle, draw=black!100,fill=yellow!100,thick, inner sep=0pt, minimum size= 2.5mm},
yellowvertex/.style={rectangle, draw=black!100,fill=yellow!100,thick, inner sep=0pt, minimum size= 2.3mm},
yellowvertexv3/.style={regular polygon,regular polygon sides=3, draw=black!100,fill=yellow!100,thick, inner sep=0pt, minimum size= 2.8mm},
redvertexv2/.style={circle, draw=black!100,fill=red!80!black,thick, inner sep=0pt, minimum size= 2.5mm},
redvertex/.style={rectangle, draw=black!100,fill=red!100,thick, inner sep=0pt, minimum size= 2.3mm},
redvertexv3/.style={regular polygon,regular polygon sides=3, draw=black!100,fill=red!100,thick, inner sep=0pt, minimum size= 2.8mm},
tealvertexv2/.style={circle, draw=black!100,fill=teal!100,thick, inner sep=0pt, minimum size= 2.3mm},
tealvertex/.style={rectangle, draw=black!100,fill=teal!100,thick, inner sep=0pt, minimum size= 2.3mm},
tealvertexv3/.style={regular polygon,regular polygon sides=3, draw=black!100,fill=teal!100,thick, inner sep=0pt, minimum size= 2.8mm},
violetvertexv2/.style={circle, draw=black!100,fill=violet!100,thick, inner sep=0pt, minimum size= 2.5mm},
violetvertex/.style={rectangle, draw=black!100,fill=violet!100,thick, inner sep=0pt, minimum size= 2.3mm},
violetvertexv3/.style={regular polygon,regular polygon sides=3, draw=black!100,fill=violet!100,thick, inner sep=0pt, minimum size= 2.8mm},
limevertexv2/.style={circle, draw=black!100,fill=olive!100,thick, inner sep=0pt, minimum size= 2.3mm},
limevertex/.style={rectangle, draw=black!100,fill=olive!100,thick, inner sep=0pt, minimum size= 2.3mm},
limevertexv3/.style={regular polygon,regular polygon sides=3, draw=black!100,fill=olive!100,thick, inner sep=0pt, minimum size= 2.8mm},
magentavertexv2/.style={circle, draw=black!100,fill=magenta!100,thick, inner sep=0pt, minimum size= 2.3mm},
magentavertex/.style={rectangle, draw=black!100,fill=magenta!100,thick, inner sep=0pt, minimum size= 2.3mm},
magentavertexv3/.style={regular polygon,regular polygon sides=3, draw=black!100,fill=magenta!100,thick, inner sep=0pt, minimum size= 2.8mm},
dummywhite/.style={circle, draw=white!100,fill=white!100,thick, inner sep=0pt, minimum size= 0.5mm},
}

\allowdisplaybreaks

\title{Smoothed analysis for graph isomorphism}
\author[Anastos]{Michael Anastos}

\address{Institute of Science and Technology Austria (ISTA).}
\email{michael.anastos@ist.ac.at}

\author[Kwan]{Matthew Kwan}

\address{Institute of Science and Technology Austria (ISTA).}
\email{matthew.kwan@ist.ac.at}

\author[Moore]{Benjamin Moore}

\address{University of Manitoba.}
\email{Ben.Moore@umanitoba.ca}

\thanks{
All authors were supported by ERC Starting Grant ``RANDSTRUCT'' No.\ 101076777. Michael Anastos was also supported in part by the Austrian Science Fund (FWF) [10.55776/ESP3863424] and by the European Union’s Horizon 2020 research and innovation
programme under the Marie Sk\l{}odowska-Curie grant agreement No.\ 101034413 \includegraphics[width=4.5mm, height=3mm]{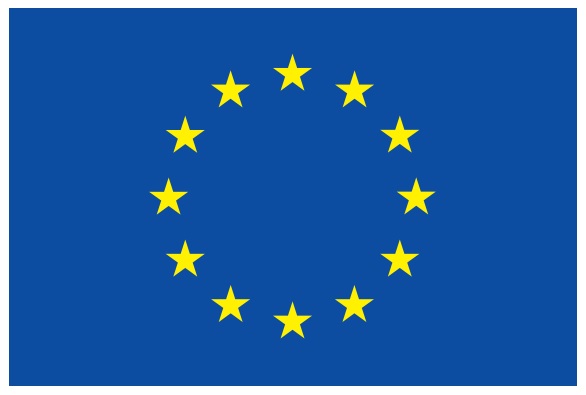}.
For Open Access purposes, the authors have applied a CC BY public copyright license to any author accepted manuscript version arising from this submission.
}

\begin{document}

\global\long\def\mk#1{\textcolor{red}{\textbf{[MK comments:} #1\textbf{]}}}
\global\long\def\ma#1{\textcolor{orange}{\textbf{[MA comments:} #1\textbf{]}}}

\begin{abstract}
There is no known polynomial-time algorithm for graph isomorphism testing, but elementary combinatorial ``refinement'' algorithms seem to be very efficient in practice. Some philosophical justification for this phenomenon is provided by a classical theorem of Babai, Erd\H os and Selkow: an extremely simple polynomial-time combinatorial algorithm (variously known as ``na\"ive refinement'', ``na\"ive vertex classification'', ``colour refinement'' or the ``1-dimensional Weisfeiler--Leman algorithm'')  yields a so-called canonical labelling scheme for ``almost all graphs''. More precisely, for a typical outcome of a random graph $\mb G(n,1/2)$, this simple combinatorial algorithm assigns labels to vertices in a way that easily permits isomorphism-testing against any other graph. 

We improve the Babai--Erd\H os--Selkow theorem in two directions. First, we consider \emph{randomly perturbed} graphs, in accordance with the \emph{smoothed analysis} philosophy of Spielman and Teng: for any graph $G$, na\"ive refinement becomes effective after a tiny random perturbation to $G$ (specifically, the addition and removal of $O(n\log n)$ random edges). Actually, with a twist on na\"ive refinement, we show that $O(n)$ random additions and removals suffice. These results significantly improve on previous work of Gaudio, R\'acz and Sridhar (resolving one of their conjectures), and are in certain senses best-possible.

Second, we complete a long line of research on canonical labelling and automorphisms for random graphs: for any $p$ (possibly depending on $n$), we prove that a random graph $\mb G(n,p)$ can typically be canonically labelled in polynomial time. This is most interesting in the extremely sparse regime where $p$ has order of magnitude $c/n$; denser regimes were previously handled by Bollob\'as, Czajka--Pandurangan, and Linial--Mosheiff. Our proof also provides a description of the automorphism group of a typical outcome of $\mb G(n,p)$ (slightly correcting a prediction of Linial--Mosheiff).
\end{abstract}
\maketitle

\section{Introduction}

Given a pair of graphs $G_{1}$ and $G_{2}$ (on the same vertex set
$\{1,\dots,n\}$, say), how can we test whether they are isomorphic?
Perhaps the most obvious first thought is to try to identify some easily-computable
isomorphism-invariant information that distinguishes the
two graphs. For example, one can easily compute the degrees of the
vertices of $G_{1}$ and the degrees of the vertices in $G_{2}$,
sort these lists and check if they are the same. If they are different,
we have successfully determined that the graphs are not isomorphic (and if they are the same our test was inconclusive).

Perhaps the most influential approach along these lines is called
\emph{colour refinement}, also known as \emph{na\"ive refinement}
or the \emph{1-dimensional Weisfeiler--Leman algorithm}\footnote{The origin of this idea is difficult to pin down, but it seems to
have been first proposed by Morgan~\cite{Mor65} in 1965 in the context of computational
chemistry!}. This is an algorithm that produces a ``colour'' for each vertex; the colour of a vertex describes the degree of that vertex, together with all other data that one can obtain by allowing degree information to ``percolate through the graph''. We will define the algorithm formally in \cref{def:refinement} (and give an example in \cref{fig:CR}), but briefly: at the start of the
algorithm, every vertex ``looks the same''. Then, in the first step
we distinguish vertices by their degrees (i.e., the colour of each vertex is its degree). In the second step, we distinguish
vertices by their \emph{number of neighbours with each degree} (i.e.,
each vertex now has a colour consisting of its own degree, together
with a multiset of the degrees of its neighbours). In general, at
each step we update the colour of each vertex by appending the
multiset of colours of its neighbours. After some number of iterations
of this procedure, it will ``stabilise'' in the sense that no further
vertices can be distinguished from each other.

It turns out that the colour refinement algorithm can be executed very efficiently: in a graph with $n$ vertices, the stable colouring can be computed\footnote{We are brushing over some subtleties here, which we will discuss further in \cref{sec:basic-notions}.} in time $O(n^2\log n)$. It is not hard to see that \emph{if} the algorithm manages to assign each vertex of $G$ a \emph{distinct label}, then the sequence of colours in the stable colouring uniquely determines the isomorphism class of $G$. In fact, more is true: in this case one can use the colours to define a \emph{canonical labelling} of $G$. This notion will be defined formally in \cref{def:canonical-labelling}, but roughly speaking it means that one can label the vertices of $G$ with the integers $\{1,\dots,n\}$ (assuming $G$ has $n$ vertices), in such a way that isomorphic graphs are always labelled the same way. 

Unfortunately, colour refinement is not always so effective. For example, in a regular graph, where every vertex has the same degree, colour refinement is
useless on its own, as it cannot distinguish any vertices from each other. However, a landmark result of Babai, Erd\H os and Selkow~\cite{BES80} shows that this situation is ``atypical'': the proportion of $n$-vertex graphs which cannot be canonically labelled using colour refinement tends to zero as $n\to \infty$ (see also the improvements in \cite{BK79,Lip78,Kar79}). This result is most easily stated in the language of random graphs, as follows. 
\begin{theorem}\label{thm:BES}
For a random graph\footnote{In the random graph $\mathbb{G}(n,p)$ (called the \emph{binomial} or sometimes the \emph{Erd\H os--R\'enyi} random graph), we fix a set of $n$ vertices and include each of the $\binom n2$ possible edges with probability $p$ independently.
} $G\sim \mb G(n,1/2)$, whp\footnote{We say a property holds \emph{with high probability}, or ``whp''
for short, if it holds with probability tending to 1. Here and for
the rest of the paper, all asymptotics are as $n\to\infty$.} the colour refinement algorithm distinguishes all vertices from each other. In particular, whp $G$ can be tested for isomorphism with any other graph in polynomial time.
\end{theorem}
\cref{thm:BES} is widely touted as philosophical justification for why algorithms based on colour refinement seem to be so effective in practice. Indeed, all graph isomorphism algorithms in common usage employ a modification of colour refinement called \emph{individualisation-refinement}\footnote{It will not be relevant for the present paper to formally describe this paradigm, but for the curious reader: the idea is that occasionally a vertex must be artificially distinguished from the other vertices of its colour in order to ``break regularity'' (one must then consider all possible ways to make this artificial choice).}. It is well-known that algorithms of this type take exponential time on worst-case inputs (see \cite{NS18} for recent general results in this direction), but one rarely seems to encounter such inputs in practice. Of course, we would be remiss not to mention that from a theoretical point of view, fully sub-exponential algorithms are now available: there is an amazing line of work due to Babai, Luks, Zemlyachenko and others (see \cite{Bab18} for a survey) applying deep ideas from group theory to the graph isomorphism problem, which culminated in Babai's recent quasipolynomial-time graph isomorphism~\cite{Bab16} and canonical labelling~\cite{Bab19} algorithms. 

\begin{remark*}
    The study of the colour refinement algorithm is of interest beyond its direct utility in graph isomorphism testing. Indeed, if graphs $G$ and $H$ are indistinguishable by colour refinement, we say that $G$ and $H$ are \emph{fractionally isomorphic}; this is an important notion of intrinsic interest in graph theory, that has surprisingly many equivalent formulations (e.g., in terms of first-order logic~\cite{CFI}, universal covers~\cite{Ang80}, tree homomorphism counts~\cite{Dvorak} and doubly-stochastic similarity of adjacency matrices~\cite{Tin86}). See for example the surveys \cite{PV11,GKMS21} and the monograph \cite{BBG17} for more.

    We also remark that the colour refinement algorithm represents the limit of so-called \emph{graph neural networks} for graph isomorphism testing~\cite{GNN1,GNN2}; these connections have recently been of significant interest in the machine learning community (see \cite{GNNsurvey,GNNsurvey2} for surveys).
\end{remark*}

\subsection{Smoothed analysis}

As our first main direction in this paper, we take the philosophy of \cref{thm:BES} much further, combining it with the celebrated \emph{smoothed analysis} framework of Spielman and Teng~\cite{ST04}. To give some context: the \emph{simplex algorithm} for linear optimisation is another fundamental example of an algorithm which seems to perform well in practice but takes exponential time in the worst case. As a very strong explanation for this, Spielman and Teng proved that if one takes \emph{any} linear optimisation problem and applies a slight (Gaussian) random perturbation to its coefficients, then for a typical outcome of the resulting perturbed linear optimisation problem, the simplex algorithm succeeds in polynomial time. This shows that poorly-performing instances are ``fragile'' or ``unstable'', and perhaps we should not expect them to appear in practice. 

Linear optimisation is a continuous problem, and for discrete problems one needs a different notion of random perturbation. In a later paper, Spielman and Teng~\cite{ST03} suggested that the most natural way
to define a discrete random perturbation is in terms of \emph{symmetric difference}: for graphs $G,G'$ on the same vertex set, write $G\triangle G'$ for the graph containing all edges which are in exactly one of $G$ and $G'$
(we can think of $G'$ as the ``perturbation graph'', specifying where we should ``flip'' edges of $G$ to non-edges, or vice versa). Significantly strengthening \cref{thm:BES}, we prove that for \emph{any} graph, randomly perturbing each edge with probability about $\log n/n$ (i.e., adding and removing about $n \log n$ random edges\footnote{Note that $\mb G(n,p)$ tends to have about $p\binom n2\approx pn^2$ edges.}) is sufficient to make colour refinement succeed whp.

\begin{theorem}\label{thm:smoothed-colour-refinement}
Fix a constant $\varepsilon>0$ and consider any $p\in [0,1/2]$ satisfying $p\ge (1+\varepsilon )\log n/n$. For any graph $G_0$, and with $G_{\mr{rand}}\sim \mb G(n,p)$, whp the colour refinement algorithm distinguishes all vertices of $G_0\triangle G_{\mr{rand}}$ from each other. In particular, whp $G_0\triangle G_{\mr{rand}}$ can be tested for isomorphism with any other graph in polynomial time.
\end{theorem}

Note that the $p=1/2$ case of \cref{thm:smoothed-colour-refinement} is precisely \cref{thm:BES}. Indeed, when $p=1/2$, the random perturbation is so extreme that all information from the original graph is lost and we end up with a purely random graph. (We may restrict our attention to $p\le 1/2$, because perturbing $G$ with a random graph $\mb G(n,p)$ is the same as perturbing the complement of $G$ with a random graph $\mb G(n,1-p)$.)

We remark that we are not the first to consider smoothed analysis for graph isomorphism: this setting was also recently considered by Gaudio, R\'acz and Sridhar~\cite{GRS}, though only under quite restrictive conditions on $G_0$, and with a stronger assumption on $p$ (both of which make the problem substantially easier). Specifically, they found an efficient canonical labelling scheme that succeeds whp as long as $G_0$ satisfies a certain ``sparse neighbourhoods'' property, and as long as $p$ has order of magnitude between $(\log n)^2/n$ and $(\log n)^{-3}$. Actually, \cref{thm:smoothed-colour-refinement} confirms a conjecture in \cite{GRS} (that randomly perturbed graphs $G_0\triangle G_{\mr{rand}}$ can be efficiently canonically labelled when $G_{\mr{rand}}\sim \mb G(n,p)$ with $p\ge (1+\varepsilon )\log n/n$).

It is not hard to see that the assumption on $p$ in \cref{thm:smoothed-colour-refinement} cannot be significantly improved. Indeed, if $G_0$ is the empty graph and $p\le (1-\varepsilon)\log n/n$, then $G_0\triangle G_{\mr{rand}}$ is likely to have many isolated vertices (see e.g.\ \cite[Theorem~3.1]{FK16}), which cannot be distinguished by colour refinement. This does not necessarily mean that colour refinement fails to provide a canonical labelling scheme (isolated vertices can be labelled arbitrarily), but to illustrate a more serious problem, suppose $G_0$ is a disconnected graph in which every component is a 3-regular graph on at most 10 vertices. If we perturb with edge probability $p\le (1/20)\log n/n$, it is not hard to see that whp many of the components of $G$ will be completely untouched by the random perturbation (and therefore colour refinement will be unable to distinguish the vertices in these components, despite the components potentially having very different structure).

Of course, while tiny regular components may foil colour refinement, they are not really a fundamental problem for canonical labelling (as we can afford to use very inefficient canonical labelling algorithms on these tiny components). We are able to go beyond \cref{thm:smoothed-colour-refinement} via a modification of the colour refinement algorithm in this spirit. Indeed, with a variation on the colour refinement algorithm called the \emph{2-dimensional Weisfeiler--Leman algorithm}, and a notion of a \emph{disparity graph} 
(which allows us to define an appropriate generalisation of the notion of a ``tiny component''), together with an exponential-time canonical labelling algorithm of Corneil and Goldberg~\cite{CG84} (only applied to the analogues of ``tiny components''), we are able to define a combinatorial canonical labelling scheme which becomes effective after extremely mild random perturbation (perturbation probability about $1/n$).

\begin{theorem}\label{thm:smoothed-disparity}
There is a set of graphs $\mc H$, and an explicit polynomial-time canonical labelling algorithm for graphs in $\mc H$ (which can also detect \emph{whether} a graph lies in $\mc H$), such that the following holds.

Consider any $p\in [0,1/2]$ satisfying $p\ge 100/n$. For any graph $G_0$, and $G_{\mr{rand}}\sim \mb G(n,p)$, whp $G_0\triangle G_{\mr{rand}}\in \mc H$. In particular, whp $G_0\triangle G_{\mr{rand}}$ can be tested for isomorphism with any other graph in polynomial time.
\end{theorem}

We highlight that the algorithm for \cref{thm:smoothed-disparity} is a ``combinatorial'' algorithm, using completely elementary refinement/recursion techniques (in particular, the algorithm can be interpreted as falling into the \emph{individualisation-refinement} paradigm, and it does not use any group theory).
Actually, up to constant factors, the restriction $p\ge 100/n$ in \cref{thm:smoothed-disparity} is essentially best possible for combinatorial algorithms, given state-of-the-art worst-case guarantees: if we were to take $G_{\mr{rand}}\sim \mb G(n,p)$ for $p=o(1/n)$ (i.e., if we were to take a milder random perturbation than in \cref{thm:smoothed-disparity}), then any polynomial-time canonical labelling algorithm that succeeds whp for graphs of the form $G_0\triangle G_{\mr{rand}}$ would immediately give rise\footnote{For any graph $G^*$, we can consider a larger graph $G_0$ with many copies of $G^*$ as connected components, with parameters chosen such that a very mild random perturbation typically leaves some copies of $G^*$ completely unaffected (and therefore a canonical labelling of the randomly perturbed graph can be translated into a canonical labelling of $G^*$). We omit the details.} to a sub-exponential-time canonical labelling algorithm for \emph{all} graphs (i.e., an algorithm that runs in time $e^{o(n)}$). Although such algorithms are known to exist (most obviously, we already mentioned Babai's quasipolynomial time algorithm), they all fundamentally use group theory.

The proof of \cref{thm:smoothed-disparity} involves a wide range of different ideas, which we outline at some length in \cref{subsec:ideas}. As an extremely brief summary: we first use expansion and anticoncentration estimates, together with a characterisation of colour refinement in terms of \emph{universal covers}, to prove that the colour refinement algorithm (applied to $G_0\triangle G_{\mr{rand}}$) typically assigns distinct colours to the vertices in the so-called \emph{3-core} of $G_{\mr{rand}}$. This is already enough to prove \cref{thm:smoothed-colour-refinement}, but to prove \cref{thm:smoothed-disparity} (with its weaker assumption on $p$), we need to combine this with a delicate 
\emph{sprinkling} argument, in which we slowly reveal the edges of $G_{\mr{rand}}$, and study how the 3-core changes during this process. Every time a new vertex joins the 3-core (thereby receiving a unique colour by the colour refinement algorithm), this new colour information cascades through the 2-dimensional Weisfeiler--Leman algorithm, breaking up many of the colour classes into smaller parts. This has the effect of partitioning the graph into smaller and smaller parts, which can be treated separately at the end.

\subsection{Sparse random graphs}
After the Babai--Erd\H os--Selkow theorem (\cref{thm:BES}) on canonical labelling for $\mb G(n,1/2)$, one of the most obvious directions for further study was to consider sparser random graphs $\mb G(n,p_n)$ (note that it suffices to consider $p_n\le 1/2$, since canonical labelling is not really affected by complementation).

The first work in this direction was by Bollob\'as (see \cite[Theorem~3.17]{Bol01}), who showed that the proof approach for \cref{thm:BES} works as long as $p_n$ does not decay too rapidly with $n$. Combining this with a later result of Bollob\'as~\cite{Bol82} (which considered a very different type of canonical labelling scheme, under different assumptions on $p_n$), and a result of Czajka and Pandurangan~\cite{CP08} (which considered colour refinement for an intermediate range of $p_n$), one obtains polynomial-time canonical labelling schemes as long as $n p_n-\log n\to \infty$ as $n\to \infty$. This condition on $p_n$ is significant because it is the same range where random graphs are typically \emph{rigid}: as was famously proved by Erd\H os and R\'enyi~\cite{ER63} (see also \cite{Wri70}), such random graphs typically have only trivial automorphisms, whereas if $n p_n-\log n\to -\infty$ then there typically exist many automorphisms.

Even sparser graphs were recently considered by Linial and Mosheiff~\cite{LM17}; they handled the regime where $np_n\to \infty$ using another different type of canonical labelling scheme. The regime $n p_n\to 0$ is easy, as in this regime (see \cite[Section~2.1]{FK16}) the components of $\mb G(n,p_n)$ are all trees with size at most $o(\log n)$ (so various different types of trivial canonical labelling schemes suffice; see for example \cite{AKRV17}). That is to say, the only regime left unaddressed is the regime where $p_n$ is about $c/n$ for some constant $c$.

As our next main result, we close this gap, finding a canonical labelling scheme for all $p_n$.
\begin{theorem}\label{thm:sparse-random-informal}
There is a set of graphs $\mc H$, and an explicit polynomial-time canonical labelling algorithm for graphs in $\mc H$ (which can also detect \emph{whether} a graph lies in $\mc H$), such that the following holds.

For any sequence $(p_n)_{n\in \mb N}\in [0,1]^{\mb N}$, and $G_n\sim \mb G(n,p_n)$, whp $G_n\in \mc H$. In particular, whp $G_n$ can be tested for isomorphism with any other graph in polynomial time.
\end{theorem}
In fact, we prove that colour refinement on its own yields a canonical labelling scheme whp, unless $p_n$ has order of magnitude $c/n$. If $p_n$ has order of magnitude $c/n$, then colour refinement still \emph{almost} works; the only obstruction is connected components which have a single cycle, which can be easily handled separately in a number of different ways.
\begin{remark}
    There are important connections between the colour refinement algorithm and mathematical logic (see for example the monograph \cite{Gro17}). Our proof shows that random graphs of any density typically have \emph{Weisfeiler--Leman dimension} at most 2 (and if we delete components with a single cycle, the Weisfeiler--Leman dimension is exactly equal to 1). That is to say, random graphs of any density can typically be uniquely distinguished by the $2$-variable fragment of first-order logic with counting quantifiers.
\end{remark}

\begin{remark}\label{rem:independent-work}An equivalent version of \cref{thm:sparse-random-informal}, in the regime $p_n=O(1/n)$, has been independently proved in concurrent work by Oleg Verbitsky and Maksim Zhukovskii~\cite{VZ} (they also study the colour refinement algorithm, but with quite different methods: while we take advantage of our general machinery already used to prove \cref{thm:smoothed-disparity}, Verbitsky and Zhukovskii take advantage of structural descriptions of the ``anatomy'' of a sparse random graph~\cite{DLP14,DKLP11}). We believe both proof approaches to be of independent interest.
\end{remark}

Given the machinery developed to prove \cref{thm:smoothed-disparity}, the proof of \cref{thm:sparse-random-informal} is rather simple. Special care needs to be taken in the regime where $p$ is very close to $1/n$ (this is the critical regime for the \emph{phase transition} of Erd\H os--R\'enyi random graphs).

A related problem is to characterise the automorphism group of a random graph. Linial and Mosheiff achieved this when $np_n\to \infty$, and asked about the case where $p_n$ has order of magnitude $1/n$. Specifically,  they wrote that they ``suspect'' that the largest $2$-connected component of the $2$-core has trivial automorphism group. We prove that Linial and Mosheiff's suspicion is mostly (but not exactly) correct. To state our theorem in this direction, we need a further definition.
\begin{definition}\label{def:2core}
    For a graph $G$, let $\mathrm{core}_k(G)$ be its $k$-core (its largest subgraph with minimum degree at least $k$). A \emph{bare path} in a graph is a path whose internal vertices have degree 2. A \emph{closed} bare path is a bare path between a vertex and itself (strictly speaking, this is actually a cycle, not a path).
\end{definition}
Note that any graph $G$ can always be obtained by gluing a (possibly trivial) rooted tree to each vertex of $\mathrm{core}_2(G)$, and adding some tree components. As observed by Linial and Mosheiff, in order to characterise the automorphisms of $G$, it suffices to characterise how such automorphisms act on $\mathrm{core}_2(G)$. Indeed, having specified how an automorphism acts on $\mathrm{core}_2(G)$, all that remains is to specify automorphisms of the rooted trees attached with each vertex, and the tree components. (Note that automorphisms of rooted trees are easy to characterise; we can only permute vertices at the same depth, and only if their corresponding subtrees are isomorphic).
\begin{theorem}\label{thm:automorphisms}
    Consider any sequence $(p_n)_{n\in \mb N}$ and let $G\sim \mb G(n,p_n)$. Then, $G$ satisfies the following property whp. Every automorphism of $G$ fixes the vertices of the $\on{core}_{2}(G)$, with the following (possible) exceptions.
\begin{itemize}
\item For every cycle component of $\on{core}_{2}(G)$, automorphisms of
this cycle may give rise to automorphisms of $G$.
\item For every pair of cycle components in $\on{core}_{2}(G)$ which have
the same length, there may be an automorphism of $G$ which exchanges
these two cycles.
\item For a pair of vertices $u,v$ with degree at least 3 in $\on{core}_{2}(G)$,
such that (in $\on{core}_{2}(G)$) there are multiple bare paths between
$u$ and $v$, there may be an automorphism of $G$ which exchanges
these bare paths.
\item For a vertex $u$ with degree at least 3 in $\on{core}_{2}(G)$, such
that (in $\on{core}_{2}(G)$) there is a (closed) bare path from $u$ to itself, there may
be an automorphism of $G$ which ``flips'' this bare path (reversing
the order of the internal vertices).
\end{itemize}
\end{theorem}

If $np_n\to \infty$, it is not hard to see that whp none of the above exceptions actually occur (and therefore \cref{thm:automorphisms} generalises the main result in \cite{LM17}, which states that the 2-core has no nontrivial automorphisms in this regime). Indeed, in the 2-core, whp: there is no pair of bare paths of the same length between any pair of degree-3 vertices, there are no bare paths from a vertex to itself, and there are no cycle components. However, when say $p_n=2/n$, there is a non-negligible probability that each of the aforementioned configurations exist (and that the rooted trees attached to these configurations permit automorphisms of $G$ which permute vertices of $\on{core}_2(G)$): the asymptotic distribution of the numbers of each of these configurations can be described by a sequence of independent Poisson random variables with nonzero means. This can be shown by the method of moments (see for example \cite[Section~6.1]{JLR00}); we omit the details.

\begin{remark}\label{rem:independent-work-2}
    In the independent work of Verbitsky and Zhukovskii mentioned in \cref{rem:independent-work}, they also deduced \cref{thm:automorphisms}. They also went on to characterise all the automorphisms of $\on{core}_2(G)$ (not just those induced by automorphisms of $G$); this requires some additional work.
\end{remark}

\subsection{Key proof ideas}\label{subsec:ideas}
Before presenting the full proofs of \cref{thm:smoothed-colour-refinement,thm:smoothed-disparity,thm:sparse-random-informal,thm:automorphisms}, we take a moment to describe some of the key ideas in the proofs.
\subsubsection{Exploring universal covers}

The starting point for all the proofs in this paper is an observation essentially due to Angluin~\cite{Ang80}, that the stable colouring obtained by the colour refinement algorithm assigns two vertices the same colour if and only if the \emph{universal covers} rooted at those vertices are isomorphic. We will use a minor modification of the notion of a universal cover called a \emph{view}, which encodes the same information but which is slightly more convenient for our purposes. Roughly speaking, the view $\mc T_G(v)$ rooted at a vertex $v$ in a graph $G$ is a (potentially infinite) tree encoding all possible walks in $G$ starting from $v$ (see \cref{sec:universal-covers} for a precise definition).

Without random perturbation, we are helpless to deal with the fact that many vertices may have isomorphic views (e.g., if $G$ is $d$-regular, then $\mc T_G(v)$ is always isomorphic to the infinite rooted tree where each vertex has $d$ children). The power of random perturbation is that, if two vertices $u$ and $v$ ``see different vertices in their walks'' (for example, if the set of neighbours of $u$ is very different from the set of neighbours of $v$), then the random perturbation is likely to affect $\mc T_G(u)$ and $\mc T_G(v)$ differently, distinguishing them from each other. One of the key technical results in this paper is a general lemma (\cref{prop:expansion-colour-refinement}) which makes this precise.

\subsubsection{Distinguishing vertices via random perturbation: expansion and anticoncentration}It is difficult to give a quick summary of \cref{prop:expansion-colour-refinement} without the necessary definitions, but to give a flavour: we define sets $\mc S^i(\{u,v\})$ which describe the vertices which ``feature differently'' in length-$i$ walks starting from $u$ and length-$i$ walks starting from $v$. \cref{prop:expansion-colour-refinement} says that even extremely mildly randomly perturbed graphs typically have the property that if for some vertices $u,v$, and some $i$, the set $\mc S^i(\{u,v\})$ contains more than about $\log n$ vertices, then $u$ and $v$ are assigned different colours by colour refinement.

To prove this, we first use the expansion properties of random graphs (via a coupling of the type often appearing in analysis of branching processes) to ``boost'' the condition on $\mc S^i(\{u,v\})$, showing that if $\mc S^i(\{u,v\})$ has more than about $\log n$ vertices then there is typically some $j\ge i$ such that $\mc S^j(\{u,v\})$ has $n^{1-o(1)}$ vertices.

There is then a huge amount of ``space'' in $\mc S^j(\{u,v\})$ to take advantage of fluctuations due to the random perturbation. In particular, we show how to algorithmically divide $\mc S^j(\{u,v\})$ into ``buckets'' with different degree statistics, use the edges within these buckets to describe certain fluctuations in $\mc T_G(u)$ and $\mc T_G(v)$ via certain essentially independent inhomogeneous random walks on $\mb Z$, and then apply an anticoncentration inequality (a variant of the classical Erd\H os--Littlewood--Offord inequality) to show that it is extremely unlikely that these random walks behave in the same way for $u$ and $v$ (so unlikely that we can union bound over all $u,v$).

Of course, in order to actually apply \cref{prop:expansion-colour-refinement}, we need to study the pairs $u,v$ for which $\mc S^i(\{u,v\})$ has more than about $\log n$ vertices (for some $i$). If the perturbation probability is greater than about $\log n/n$ (as in \cref{thm:smoothed-colour-refinement}), it is easy to show that whp \emph{all} pairs of vertices have this property (in this case we can even take $i=1$). For milder random perturbation (or sparser random graphs), we will need to restrict our attention to certain pairs $u,v$ lying in certain special subgraphs, as we discuss next.

\subsubsection{The 2-core and the kernel}\label{subsec:cores-outline}
Recall from \cref{def:2core} that the \emph{$k$-core} of a graph $G$ is its maximal subgraph of minimum degree at least $k$. The \emph{kernel} of $G$ is the smallest multigraph homeomorphic to the 2-core of $G$ (i.e., with the same topological structure as the 2-core). It can be obtained from the 2-core by iteratively replacing bare paths (also defined in \cref{def:2core}) by single edges.

Especially in the setting of random graphs, cores and kernels are objects of fundamental interest, typically possessing extremely strong expansion properties. In particular, the kernel of a random graph in some sense describes its fundamental underlying expander structure: a celebrated theorem of Ding, Lubetzky and Peres~\cite{DLP14} shows that one can in some sense ``build a random graph from its kernel'' by (very informally) starting from a random expander (the kernel), randomly replacing some edges with bare paths, (to obtain the 2-core), and randomly affixing some trees.

If $k\ge 3$, the expansion properties of the $k$-core make it quite convenient to apply \cref{prop:expansion-colour-refinement}: it is fairly simple to show that, for $k\ge 3$, a mildly randomly perturbed graph typically has the property that colour refinement assigns a unique colour to all vertices of the $k$-core of $G_{\mr{rand}}$. This immediately implies \cref{thm:smoothed-colour-refinement} (since when $p\ge (1+\varepsilon)\log n/n$ the 3-core of $G_{\mr{rand}}$ typically comprises the entire vertex set).

However, in very sparse regimes (in particular, when $p< c/n$ for a certain constant $c$, famously computed by Pittel, Spencer and Wormald~\cite{PSW96}), the 3-core is typically \emph{empty}, and we are forced to turn to the 2-core and the kernel. Unfortunately, this makes everything much more delicate. We prove a somewhat technical lemma (\cref{prop:2core}) showing that a mildly randomly perturbed graph typically has the property that colour refinement can distinguish the vertices\footnote{This is not strictly speaking true; for the purposes of this outline we are ignoring a technical caveat concerning very rare configurations of edges in the 2-core.} of degree at least 3 in the 2-core.

\subsubsection{Sparse random graphs}\label{subsec:sparse-random-outline}
The above considerations on the 2-core apply when $p\ge (1+\varepsilon)/n$ (for any constant $\varepsilon>0$). Considering the case where the initial graph $G_0$ is empty, it is straightforward to deduce that simple canonical labelling schemes are typically effective for sparse random graphs $\mb G(n,p)$, when $p\ge (1+\varepsilon)/n$ (thus proving \cref{thm:sparse-random-informal} in this regime).

The significance of this assumption on $p$ is that it corresponds to the celebrated \emph{phase transition} for Erd\H os--R\'enyi random graphs: when the edge probability is somewhat above $1/n$, there is typically a \emph{giant component} with good expansion properties, but when the edge probability is somewhat less than $1/n$, there are typically only tiny components with very poor expansion properties (see for example the monographs \cite{FK16,JLR00,Bol01} for more).

In the critical regime $(1-\varepsilon)/n<p< (1+\varepsilon)/n$, we proceed differently. In this regime, every component of $\mb G(n,p)$ with more than one cycle has quite large \emph{diameter}: exploration processes can run for quite a long time without exhausting all the vertices in the graph, and we can accumulate quite a lot of independent randomness over this time. The anticoncentration from this randomness gives us another way to prove that different vertices are assigned different colours by colour refinement (completing the proof of \cref{thm:sparse-random-informal}). The details of this argument appear in \cref{sec:random-graphs}.

\subsubsection{Sprinkling via the 3-core}\label{subsec:sprinkling-outline}
For \cref{thm:smoothed-disparity} (on canonical labelling of very mildly randomly perturbed graphs), the role of the $k$-core (and \cref{prop:2core}, discussed in \cref{subsec:cores-outline}) is that it provides a kind of ``monotonicity'' that allows us to use a technique called \emph{sprinkling}.

For the unfamiliar reader: sprinkling, in its most basic form, is the observation that a random graph $G_{\mr{rand}}\sim \mb G(n,p)$ can be interpreted as the union of two independent random graphs $G_{\mr{rand}}^1\cup G_{\mr{rand}}^2$, where $G_{\mr{rand}}^1,G_{\mr{rand}}^2\sim \mb G(n,p')$ with $1-p=(1-p')^2$. This observation allows one to first show that certain properties hold whp for $G_{\mr{rand}}^1$, then reveal an outcome of $G_{\mr{rand}}^1$ satisfying these properties, and use the independent randomness of $G_{\mr{rand}}^2$ to ``boost'' these properties.

Sprinkling only really makes sense when we are dealing with properties of $G_{\mr{rand}}$ that are \emph{monotone}, in the sense that once we have established the property for some subgraph of $G_{\mr{rand}}$, adding the remaining edges of $G_{\mr{rand}}$ cannot destroy the property. Unfortunately, the colour refinement algorithm is highly non-monotone: in general, adding additional edges can allow the algorithm to distinguish more vertices, but can also prevent the algorithm from distinguishing some vertices. In the proof of \cref{thm:smoothed-disparity}, the critical role played by \cref{prop:2core} is that it makes a connection between the colour refinement algorithm and the $k$-core, which is a fundamentally monotone object. For example, if a vertex is in the $k$-core of $G_{\mr{rand}}^1$, then it is guaranteed to be in the $k$-core of $G_{\mr{rand}}^1\cup G_{\mr{rand}}^2$.

We will actually split our random perturbation $G_{\mr{rand}}$ into many independent random perturbations $G_{\mr{rand}}^1,\dots,G_{\mr{rand}}^T\in \mb G(n,p')$ (in the proof of \cref{thm:smoothed-disparity} we will take $T=8$). If $p'\ge 10/n$, then one can show that whp the first random perturbation $G_{\mr{rand}}^1$ already has a giant 3-core\footnote{With a little more work, it would suffice to consider the vertices of degree at least 3 in the 2-core, in which case we only need $p'$ to be slightly larger than $1/n$. However, the 3-core is much more convenient to work with.}, and by \cref{prop:2core} we can safely assume that each of the vertices in this 3-core will be assigned unique colours by the colour refinement algorithm (applied to the randomly perturbed graph $G=G_0\triangle (G_{\mr{rand}}^1\cup\dots\cup G_{\mr{rand}}^T)$, which we have not yet fully revealed). Then, in each subsequent random perturbation $G_{\mr{rand}}^{i}$, for $i\ge 2$, additional vertices randomly join the 3-core, and we can assume that they will also receive unique colours.

The upshot is that, whatever properties we are able to prove about the stable colouring produced by the colour refinement algorithm, we can ``boost'' these properties by randomly assigning unique colours to some vertices (and studying how this new information propagates through the colour refinement algorithm). To describe the types of properties we are interested in, we need to define the notion of a \emph{disparity graph}\footnote{This notion has previously been introduced under the name ``flip graph''~\cite{KN24,amenablegraphs2}; we thank the anonymous referees of \emph{STOC 2025} for bringing this to our attention.}, as follows.

\subsubsection{Small components in disparity graphs}\label{subsec:disparity-outline}
Recall that \cref{thm:smoothed-colour-refinement} cannot hold for $p=o(\log n/n)$. The key obstruction to keep in mind is that if $G_0$ has many tiny regular connected components, then very mild random perturbation will leave some of these components untouched, and the colour refinement algorithm will not be able to distinguish the vertices in these components.

For the proof of \cref{thm:smoothed-disparity}, we therefore need an appropriate generalisation of the notion of ``tiny component'' (taking into account the fact that tiny components in the \emph{complement} of $G$ play the same role as tiny components of $G$). To this end, we introduce the notion of a \emph{disparity graph}.

\begin{definition}\label{def:disparity}
    For a graph $G$, a set of colours $\Omega$ and a colouring $c:V(G)\to \Omega$, define the \emph{majority graph} $M(G,c)$ (on the same vertex set as $G$) as follows. For any (possibly non-distinct) pair of colours $\omega,\omega'\in \Omega$:
    \begin{itemize}
        \item If at least half of the possible edges between vertices of colours $\omega$ and $\omega'$ are in fact present as edges of $G$, then $M(G,c)$ contains every possible edge between vertices of colours $\omega$ and $\omega'$.
        \item Otherwise (if fewer than half of the possible edges between vertices of colours $\omega$ and $\omega'$ are present), $M(G,c)$ contains no edges between vertices of colours $\omega$ and $\omega'$.
    \end{itemize}
    Then, define the \emph{disparity graph} $D(G,c)=M(G,c)\triangle G$.
\end{definition}

Informally speaking, the majority graph $M(G,c)$ is the best possible approximation to $G$ among all graphs which are ``homogeneous'' between colour classes (for every pair of colour classes, to decide whether to put all edges between them or no edges between them, we look at the majority behaviour in $G$ among vertices of those colours). Then, the disparity graph identifies the places where the majority graph differs from $G$. Equivalently, we can define the disparity graph to be the graph obtained by considering every pair of colour classes and deciding whether to complement the edges between those colour classes or not, depending on which of the two choices would make the graph sparser.

If $c$ is the stable colouring obtained from the colour refinement algorithm, it is not hard to show that a canonical labelling of $D(G,c)$ yields a canonical labelling of $G$ (this is stated formally in \cref{prop:D-to-G} later in the paper). So, we can canonically label $G$ in polynomial time whenever $D(G,c)$ has sufficiently small components (small enough that we can afford to use known inefficient canonical labelling schemes on each component).

\subsubsection{Percolation for a weaker result}\label{subsec:percolation-outline}

Let $c$ be the stable colouring obtained by the colour refinement algorithm. To prove \cref{thm:smoothed-disparity}, it would suffice to prove that whp the disparity graph $D(G_0\triangle G_{\mr{rand}},c)$ has small components (for a polynomial-time combinatorial algorithm, we need every component to have $O(\log n)$ vertices, so that we can afford to use an exponential-time canonical labelling algorithm of Corneil and Goldberg~\cite{CG84} on each component\footnote{We remark that it does not actually seem to make the problem much easier if we weaken this requirement on the size of each component (which we may, if we are willing to use a more sophisticated group-theoretic algorithm, with better worst-case guarantees, on each connected component).}).

Unfortunately, we were not quite able to manage this when the random perturbation probability is $O(1/n)$ (as demanded by \cref{thm:smoothed-disparity}). Indeed, our proof of \cref{thm:smoothed-disparity} requires a more sophisticated variant of the colour refinement algorithm, as we discuss later in this outline. However, the above goal is achievable if the random perturbation probability is at least (say) $100\log \log n/n$. For expository purposes we next sketch how to prove this, before moving on to the more sophisticated ideas in the full proof of \cref{thm:smoothed-disparity} (the details of this simpler argument can be found in \cref{app:loglog}).

As discussed in \cref{subsec:sprinkling-outline}, we can interpret our random perturbation $G_{\mr{rand}}$ as a composition of three random perturbations $G_{\mr{rand}}^{1}\cup G_{\mr{rand}}^{2}\cup G_{\mr{rand}}^{3}$. The first random perturbation $G_{\mr{rand}}^{1}$ already whp establishes a giant 3-core (of size at least $n/2$, say); we can assume that the vertices in this 3-core get unique colours (and are hence isolated vertices in the disparity graph), so we only need to worry about the vertices outside the 3-core. Our two additional random perturbations $G_{\mr{rand}}^{2}$ and $G_{\mr{rand}}^{3}$ each cause an independent random subset of vertices to receive unique colours (as they join the 3-core): if $p\ge 100\log \log n/n$, we calculate that each vertex receives a unique colour with probability at least $1-o(1/\log n)$.

With one round of sprinkling (i.e., with the random assignment of unique colours provided by $G_{\mr{rand}}^2$, followed by the colour refinement algorithm) we can show that the disparity graph has maximum degree $O(\log n)$ whp\footnote{For this bound we only need that (say) $p\ge 100/n$ (this implies that in each round of sprinkling, every vertex joins the 3-core with probability at least 0.9). Using that $p\ge 100\log \log n/n$ (so in every round of sprinkling, every vertex joins the 3-core with probability at least $1-o(1/\log n)$), we can actually show that the degrees are at most $O(\log n/\log \log n)$ whp. But, this stronger bound would not really affect the final result.}. Indeed, the disparity graph describes how the neighbourhood of a vertex differs from the ``majority behaviour'' among vertices of its colour class, so if there is a vertex with high degree in the disparity graph, then there is a pair of vertices in the same colour class with very different neighbourhoods. With a union bound over pairs of vertices, it is easy to show that no such pairs persist after a round of sprinkling (the random assignment of unique colours typically allows the colour refinement algorithm to distinguish all such pairs).

For our second round of sprinkling (provided by $G_{\mr{rand}}^3$), we view the random assignment of unique colours as \emph{percolation}: if a vertex gets a unique colour, then it becomes isolated in the disparity graph, and we can imagine that that vertex is deleted. We are interested in the connected components that remain after these deletions\footnote{Here we are sweeping under the rug some technical issues related to the ``consistency'' of the disparity graph as the underlying graph changes. This turns out to be quite delicate, and is handled in \cref{lem:disparity-components}.}. But if the disparity graph has degree $O(\log n)$ before sprinkling, and each vertex is deleted with probability $1-o(1/\log n)$, one can show that these deletions usually shatter the disparity graph into small components (the necessary analysis is similar to analysis of subcritical branching processes). This percolation step is the part where we are fundamentally using that the random perturbation probability is at least about $\log \log n/n$: if the perturbation probability were smaller than this, then the deletions would not be severe enough to break the disparity graph into small components.

\subsubsection{Splitting colour classes and components}\label{subsec:breakup-outline}
In order to go beyond the ideas in the previous subsection, to work with random perturbation probabilities as small as $O(1/n)$, we need to get a much stronger conclusion from our sprinkled random perturbation. Every time a vertex is added to the 3-core (and gets a unique colour), this vertex is not simply removed from the disparity graph: the new colour information has a cascading effect (via the colour refinement algorithm) that affects the colours of many other vertices, indirectly affecting the connected components of the disparity graph.

First, it is instructive to think about the effect of sprinkling on the connected components (of the disparity graph) which intersect a given colour class $C$. If there is a vertex $v$ which has a neighbour in $C$ (with respect to the disparity graph), then if we were to assign $v$ a unique colour this would cause $C$ to break into multiple colour classes: namely, the colour refinement algorithm would be able to distinguish the vertices in $C$ adjacent to $v$ from the vertices in $C$ which are not adjacent to $v$ (it is a simple consequence of the definition of the disparity graph that $v$ is adjacent to at most half the vertices in $C$).

So, we can consider an exploration process that starts from the vertices of $C$, and explores\footnote{For the purposes of this outline, the reader can imagine that at each step we choose the unexplored vertex which is closest to $C$ in the current disparity graph, though our actual exploration process is a bit more complicated.} all the vertices which share a connected component (in the disparity graph) with a vertex in $C$. Each time we consider a new vertex, we reveal whether it is assigned a unique colour, and propagate this information via the colour refinement algorithm\footnote{The resulting changes to the colouring also change the disparity graph. In particular, vertices which used to be in the same connected component may later be spread over multiple connected components, but this is not really a problem.}. As we continue this exploration/refinement process, the colour classes will start to break up into smaller pieces. There is a limit to how long this process can continue, since colour classes of size 1 cannot be broken up further. Indeed, by considering an auxiliary submartingale (which measures how the number of colours we have discovered so far compares to the number of vertices we have explored so far, at each point in the process), we can prove that our exploration process is likely to terminate after $O(|C|)$ steps, having explored all the vertices that share a component with a vertex in $C$. That is to say, the sizes of the connected components of the disparity graph after sprinkling are bounded in terms of the sizes of the colour classes before sprinkling. (We emphasise that in this discussion $C$ always refers to our colour class \emph{before} sprinkling, i.e., we are not varying $C$ during the exploration process.)

Unfortunately, we cannot hope to show that the colour classes are small, in general (indeed, if $G$ has many isolated vertices, then all these vertices will be assigned the same colour by any canonical vertex-colouring scheme). However, the above idea can be ``localised'' to a connected component:
We prove a crucial lemma (\cref{lem:break}), which tells us that in order to show that the disparity graph has connected components with $O(\log n)$ vertices (after sprinkling), it suffices to show that, before sprinkling, for all colour classes $C$ and connected components $X$ of the disparity graph, we have $|C\cap X|=O(\log n)$.

In order to show that these intersection sizes $|C\cap X|$ are small, we need another round of sprinkling which ``shatters large components into small colour classes''. In order for this sprinkling to have a strong enough effect, we need to consider a more powerful variant of the colour refinement algorithm, as follows.

\subsubsection{Distinguishing vertices via the 2-dimensional Weisfeiler--Leman algorithm}\label{subsec:2WL-outline}
The \emph{2-dimensional Weisfeiler--Leman algorithm} refines colourings of \emph{pairs} of vertices: starting with a certain ``trivial'' colouring $\phi_G:V(G)^2\to \Omega$, we repeatedly refine $\phi_G$ based on statistics of 3-vertex configurations, until a stable colouring $f:V(G)^2\to \{0,1\}$ is reached. This colouring of pairs of vertices then gives rise to a colouring of individual vertices $v\mapsto f(v,v)$, which contains a lot more information than the result of ordinary colour refinement.

In particular, this more sophisticated refinement operation allows us to distinguish vertices based on \emph{distances}: in the final vertex-colouring, every two vertices of the same colour see the same number of vertices of every given colour at any given distance (in the disparity graph). As discussed in \cref{subsec:percolation-outline}, after a single round of sprinkling whp the disparity graph has maximum degree $O(\log n)$, so if a connected component $X$ has more than logarithmically many vertices then it is quite sparse, meaning that there is a very rich variety of different pairs of vertices at different distances. So, when sprinkling causes new vertices to receive unique colours, hopefully the 2-dimensional Weisfeiler--Leman algorithm will be able to significantly break up the colour classes in $X$ (recall that our goal is now to prove that the intersections between colour classes and connected components have size $O(\log n)$).

Na\"ively, one might hope to prove this via a simple union bound over subsets $Z\subseteq X$ of about $\log n$ vertices: for any such set $Z$, one might try to prove that after a round of sprinkling, and the 2-dimensional Weisfeiler--Leman algorithm, it is overwhelmingly unlikely that all vertices of $Z$ have the same colour. To prove this, it would suffice to show that for any such $Z$ there are many vertices which see some vertices of $Z$ at different distances (so if any of these many vertices receive a new unique colour, this could be used by the 2-dimensional Weisfeiler--Leman algorithm to give different colours to some of the vertices of $Z$).

Unfortunately, this direct approach does not seem to yield any nontrivial bounds, without making structural assumptions about $G$. Instead, we use a ``fingerprint'' technique reminiscent of the method of \emph{hypergraph containers} (see \cite{containers-1,containers-2}) in extremal combinatorics, which reduces the scope of our union bound. Specifically, if there were a large colour class $C'\subseteq X$ after sprinkling, we show that this would imply the existence of a much smaller ``fingerprint'' set $S\subseteq X$ which sees many vertices (specifically, many vertices of $C'$) at a variety of different distances. We can then take a cheaper union bound over the smaller fingerprint sets $S$.

We remark that the above sketch was very simplified, and serves only to illustrate the rough ideas. The full proof of \cref{thm:smoothed-disparity} (which appears in \cref{sec:disparity}) confronts a number of delicate technical issues and requires seven rounds of sprinkling, each of which gradually reduce the degrees, colour classes and connected components of the disparity graph.

\subsection{Further directions}\label{subsec:further-directions}
There are a very large number of natural directions for further research. First is the question of optimising the probability implicit in the ``whp'' in each of \cref{thm:smoothed-colour-refinement,thm:smoothed-disparity,thm:sparse-random-informal}, and improving the runtimes of the relevant algorithms. In the setting of \cref{thm:BES}, both these issues were comprehensively settled by Babai and Kucera~\cite{BK79}: they showed that, except with exponentially small probability, two steps of colour refinement suffice (yielding a linear-time algorithm).

The algorithm in \cref{thm:smoothed-colour-refinement} (colour refinement) runs in time $O((n+m)\log n)$, where $n$ and $m$ are the numbers of vertices and edges of our graph of interest, while the algorithm in \cref{thm:sparse-random-informal} na\"ively runs in time $O(n(n+m)\log n)$, due to repeated iteration of the colour refinement algorithm. With some careful analysis, it may be possible to bound the necessary number of colour refinement steps, to obtain an optimal linear-time algorithm in the setting of \cref{thm:smoothed-colour-refinement}. In fact, the work of Gaudio, R\'acz and Sridhar~\cite{GRS}, mentioned earlier in the introduction, already makes some initial steps in this direction: they essentially show that three steps of colour refinement suffice, as long as $p$ has order of magnitude between $(\log n)^2/n$ and $(\log n)^{-3}$. It seems plausible that in fact three steps of colour refinement suffice as long as $p\ge (1+\varepsilon)\log n/n$, and recent work on ``shotgun reassembly''~\cite{GM22,JKRS} indicates that there should be a phase transition between two steps of colour refinement being necessary, and three steps being necessary, at around $p=(\log n)^2(\log \log n)^{-3}/n$.

With similar considerations on the necessary number of colour refinement steps, and the necessary number of iterations of the colour refinement algorithm, it may also be possible to improve the runtime in the setting of \cref{thm:sparse-random-informal} to near-linear-time. Regarding \cref{thm:smoothed-disparity}: this algorithm also na\"ively runs in about quadratic-time, due to the use of the 2-dimensional Weisfeiler--Leman algorithm. However, we do not need the full power of the general algorithms that we cite, and it seems plausible that special-purpose variants could be designed that might also yield a near-linear-time algorithm.

Also, there is the possibility of improving on the amount of random perturbation in \cref{thm:smoothed-disparity}. In particular, taking advantage of group-theoretic techniques (e.g., using the quasipolynomial-time algorithm of Babai~\cite{Bab19}), it might be possible to design a canonical labelling scheme that becomes effective with tiny amounts of random perturbation (e.g., perturbation probability $o(1/n)$). In the case where $G_0$ is a regular graph, it may also be worth considering alternative models of random perturbation that do not destroy the regularity of $G_0$, as colour refinement is completely ineffective for regular graphs. Perhaps surprisingly, this may actually make the problem \emph{easier}, as sparse random regular graphs are much better expanders than sparse Erd\H os--R\'enyi random graphs  (cf.\ the work of Bollob\'as~\cite{Bol82} showing that distance profiles provide an efficient canonical labelling scheme even for very sparse random regular graphs).

Next, there is the possibility of stronger connections between our results on smoothed analysis for graph isomorphism, and the work of Spielman and Teng on smoothed analysis for linear optimisation. Indeed, (a slight variant of) the colour refinement algorithm is used for \emph{dimension reduction} in linear optimisation (see \cite{GKMS14}) and it would be interesting to consider smoothed analysis in this setting. 

There is also the possibility of considering smoothed analysis for many different types of graph algorithms other than isomorphism-testing. Some early work in this direction was undertaken by Spielman and Teng~\cite{ST03} (see also \cite{AM09,MR22,BRU17,BRS19}), but somewhat surprisingly, despite the algorithmic origin of the smoothed analysis framework, there is now a much larger body of work on randomly perturbed graphs in extremal and probabilistic graph theory (see for example \cite{RP1,RP2,RP3,RP4,RP5,RP6,RP7,RP8,RP9,RP10,RP11,RP12,RP13,RP14,RP15,RP16}) than on algorithmic questions.

In particular, it is worth remarking that canonical labelling is closely related to the so-called ``network alignment problem'' (also known as the ``graph matching problem''), in which the objective is to find a mapping between the vertex sets of two graphs such that the number of adjacency disagreements between the two graphs is minimised. See for example \cite{DMWZ21} (and the references therein) for recent work on alignment of random graphs; it would be interesting to explore to what extent this can be extended to the smoothed analysis stetting.

\subsection{Notation}
Our graph-theoretic notation is for the most part standard. For a graph $G$, we write $V(G)$ for its set of vertices. We write $G[S]$ for the subgraph of $G$ induced by the vertex subset $S$, write $G[S,T]$ for the bipartite subgraph of $G$ induced between $S$ and $T$ (assuming $S,T$ are disjoint), and write $G\triangle G'$ for the \emph{symmetric difference} of two graphs $G,G'$ on the same vertex set (i.e., $e$ is an edge of $G\triangle G'$ if it is an edge of exactly one of $G,G'$). We write $\on{dist}_G(u,v)$ for the distance between two vertices $u,v$, and we write $\mb G(n,p)$ for the Erd\H{o}s--R\'enyi random graph on $n$ vertices with edge probability $p$. 

We (slightly abusively) conflate vertex-colourings with vertex-partitions. In particular, for vertex-colourings $c:V(G)\to \Omega$ and $c':V(G)\to \Omega'$, we say that $c$ is a \emph{refinement} of $c'$ (or $c'$ is a \emph{coarsening} of $c$) if each colour class of $c$ is a subset of a colour class of $c'$.

Our use of asymptotic notation is standard as well. For functions
$f=f\left(n\right)$ and $g=g\left(n\right)$, we write $f=O\left(g\right)$
to mean that there is a constant $C$ such that $\left|f\right|\le C\left|g\right|$,
$f=\Omega\left(g\right)$ to mean that there is a constant $c>0$
such that $|f|\ge c\left|g\right|$ for sufficiently large $n$, and $f=o\left(g\right)$ to mean that $f/g\to0$ as $n\to\infty$. We say that an event occurs with high probability (``whp'') if it holds with probability $1 - o(1)$.

For a real number $x$, the floor and ceiling functions are denoted $\lfloor x\rfloor=\max\{i\in \mb Z:i\le x\}$ and $\lceil x\rceil =\min\{i\in\mb Z:i\ge x\}$. We will however sometimes omit floor and ceiling symbols and assume large numbers are integers, wherever divisibility considerations are not important. All logarithms in this paper are in base $e$, unless specified otherwise.

\subsection{Organisation of the paper} The structure of the paper is as follows. We start with some preliminaries in \cref{sec:basic-notions}; in particular, this section features the formal definitions of many of the notions informally discussed in this introduction. Some of the more routine proofs in this section are deferred to \cref{app:labelling}.

Then, in \cref{sec:universal-covers} we introduce some machinery for working with \emph{views}, and in \cref{sec:expansion} we use this machinery to prove a general lemma (\cref{prop:expansion-colour-refinement}) giving an expansion-type condition under which two vertices are typically assigned different colours by colour refinement.

In \cref{sec:3core} we apply \cref{prop:expansion-colour-refinement} to prove that (under suitable random perturbation) colour refinement typically assigns unique colours to all of the vertices of degree 3 in the 2-core of $G_\mr{rand}$ which satisfy a certain technical condition. We deduce (in \cref{prop:3core}) that colour refinement typically assigns unique colours to all the vertices in the 3-core of $G_\mr{rand}$. This immediately implies \cref{thm:smoothed-colour-refinement}.

In \cref{sec:disparity} we prove \cref{thm:smoothed-disparity}, combining the results in \cref{sec:3core} with a large number of additional ideas. This section can be read independently of most of the rest of the paper (treating \cref{prop:3core} as a black box). A subset of the ideas in \cref{sec:disparity} can also be used to give a much simpler proof of a slightly weaker result (assuming that the random perturbation probability $p$ is at least about $\log \log n/n$); we include the details in \cref{app:loglog}.

Finally, in \cref{sec:random-graphs} we prove \cref{thm:sparse-random-informal,thm:automorphisms}, via the results in \cref{sec:3core} and some additional considerations for near-critical random graphs. 

\subsection{Acknowledgements} We would like to thank Oleg Verbitsky and Maksim Zhukovskii for several insightful comments, and for alerting us to their independent work (see \cref{rem:independent-work,rem:independent-work-2}). The first two authors would also like to thank Marc Lelarge for insightful discussions which inspired some of the directions in this paper, and the second author would like to thank Oliver Riordan for interesting comments on the necessary number of colour refinement steps (see \cref{subsec:further-directions}). Finally, we would like to thank the anonymous reviewers of previous versions of this paper, for their careful reading and a large number of comments and suggestions that have significantly improved the paper.

\section{Basic notions and preliminary lemmas}\label{sec:basic-notions}
In this section we formally define some of the notions informally discussed in the introduction, and we state and prove some general-purpose lemmas that will be used throughout the paper. In \cref{subsec:CL} we present some general definitions and results about canonical labelling, in \cref{subsec:CR} we collect some basic facts about colour refinement, in \cref{subsec:WL} we define the 2-dimensional Weisfeiler--Leman algorithm and collect some basic results about it, and in \cref{subsec:prob} we collect some probabilistic estimates.
\subsection{Canonical labelling} \label{subsec:CL}
We start by formally defining the notion of canonical labelling.
\begin{definition}\label{def:canonical-labelling}
Let $\mc G_n$ be the set of all graphs on the vertex set $\{1,\dots,n\}$, and let $\mc S_n$ be the set of all permutations of $\{1,\dots,n\}$. A \emph{labelling scheme} is a map $\Phi:\mc G_n\to \mc S_n$. For a graph $G\in \mc G_n$, we write $\Phi((G))\in \mc G_n$ for the graph obtained by applying the permutation $\Phi(G)$ to the vertices and edges of $G$. 

We say that a family of graphs $\mathcal F\subseteq \mc G_n$ is \emph{isomorphism-closed} if it is a disjoint union of isomorphism classes in $\mc G_n$. For an isomorphism-closed family of graphs $\mathcal F\subseteq \mc G_n$, we say that $\Phi:\mc G_n\to \mc S_n$ is a \emph{canonical} labelling scheme for $\mc F$ if $\Phi((G))=\Phi((G'))$ whenever $G,G'\in \mc F$ are isomorphic.
\end{definition}

Intuitively speaking, one should think of a canonical labelling scheme as a way to assign labels to an \emph{unlabelled} graph $G$, where we are not allowed to make ``arbitrary choices'' (e.g., arbitrarily breaking ties), unless different outcomes of those choices correspond to automorphisms of $G$.

Note that a canonical labelling scheme $\Phi$ for an isomorphism-closed family $\mc F\subseteq \mc G_n$ gives rise to a simple way to test for isomorphism between a graph $G\in \mc F$ and any other graph $G'\in \mc G_n$: simply compute $\Phi((G))$ and $\Phi((G'))$, and compare the two labelled graphs (e.g., by their adjacency matrices).

The primary engine underlying all canonical labelling schemes in this paper is the so-called \emph{colour refinement algorithm}. It was informally introduced in the introduction; here we define it more formally.
\begin{definition}\label{def:refinement}
For a graph $G$ and a colouring $c:V(G)\to \Omega$, the \emph{refinement} $\mathcal R_Gc:V(G)\to \Omega\times\mb N^\Omega$ of $c$ is defined by $\mc R_G c(v)=(c(v),(d_\omega(v))_{\omega\in \Omega})$, where $d_\omega(v)$ denotes the number of neighbours of $v$ which have colour $\omega$ (with respect to $c$). Note that every colouring $c:V(G)\to \Omega$ defines a partition of $V(G)$ into colour classes. If $c,c'$ define the same vertex partition, then $\mc R_Gc,\mc R_Gc'$ also define the same vertex partition. So, we can actually view $\mc R_G$ as an operator on vertex partitions.

Also, denote the $t$-fold iteration of $\mc R_G$ by
\[\mathcal R_G^tc=\underbrace{\mathcal R_G\dots \mathcal R_G}_{t\text{ times}} c.\]

 Note that for any colouring $c$, the partition defined by $\mc R_G c$ is always a refinement of the partition defined by $c$. For a vertex set of size $n$, the longest possible chain of proper refinements has length $n$, so if we start with any colouring $c$ in an $n$-vertex graph, and repeatedly apply the refinement operation $\mc R_G$ (at most $n$ times), we always end up with a ``stable colouring''. That is to say, there is always some $t\le n$ such that $\mc R_G^{t}c$ and $\mc R_G^{t+1} c$ define the same vertex partition.
We denote the corresponding stable colouring by $\mc R_G^*c=\mc R_G^tc$. 

We will write $\sigma_G:V(G)\to \{0\}$ for the trivial colouring in which every vertex is given the same ``0'' colour.
\end{definition}

So, $\mc R_G\sigma_G$ distinguishes vertices by their degree, and $\mc R_G^* \sigma_G$ is the stable colouring obtained by the colour refinement algorithm as described in the introduction. See \cref{fig:CR} for an example. For all the notation introduced in \cref{def:refinement}, we will often omit the subscript ``$G$'' when it is clear from context.

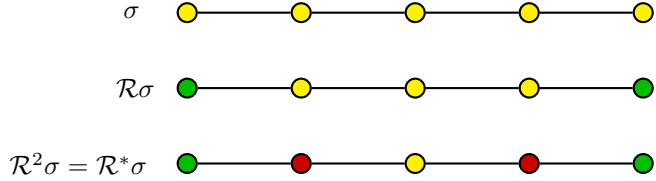
\begin{figure}
\begin{center}
    \begin{tikzpicture}
    \node[yellowvertexv2] at (0,0) (v1) {};
    \node[yellowvertexv2] at (1.5,0) (v2) {};
    \node[yellowvertexv2] at (3,0) (v3) {};
    \node[yellowvertexv2] at (4.5,0) (v4) {};
    \node[yellowvertexv2] at (6,0) (v5) {};
    \node[dummywhite] at (-1,0) (dummy1) [label = right:$\sigma$]{};
    \draw[thick,black] (v1)--(v2)--(v3)--(v4)--(v5);

    \begin{scope}[yshift =-1cm]
    \node[greenvertexv2] at (0,0) (v1) {};
    \node[yellowvertexv2] at (1.5,0) (v2) {};
    \node[yellowvertexv2] at (3,0) (v3) {};
    \node[yellowvertexv2] at (4.5,0) (v4) {};
    \node[greenvertexv2] at (6,0) (v5) {};
    \node[dummywhite] at (-1.1,0) (dummy1) [label = right:$\mathcal{R} \sigma$] {};
    \draw[thick,black] (v1)--(v2)--(v3)--(v4)--(v5);
    \end{scope}

    \begin{scope}[yshift = -2cm]
    \node[greenvertexv2] at (0,0) (v1) {};
    \node[redvertexv2] at (1.5,0) (v2) {};
    \node[yellowvertexv2] at (3,0) (v3) {};
    \node[redvertexv2] at (4.5,0) (v4) {};
    \node[greenvertexv2] at (6,0) (v5) {};
    \node[dummywhite] at (-2.5,0) (dummy1) [label = right:$\mathcal{R}^{2} \sigma\equal \mc R^*\sigma$] {};
    \draw[thick,black] (v1)--(v2)--(v3)--(v4)--(v5);
    \end{scope}

    \end{tikzpicture}
    \caption{\label{fig:CR}An illustration of the colour refinement algorithm on the 4-edge path.
    }
    \end{center}
\end{figure}

If $\mc R^*\sigma$ assigns every vertex in $G$ a distinct colour, then it is easy to obtain a canonical labelling for $G$ from $\mc R^*\sigma$: we simply fix some total order of all the potential colours of vertices (e.g. lexicographical order), and assign the labels $\{1,\dots,n\}$ in increasing order. However, it is not completely obvious that this gives rise to an efficient canonical labelling scheme: na\"ively, after many steps of refinement the sizes of the labels become extremely large. In order to address this, after each step of colour refinement we can re-encode all colours with some subset of the integers $\{1,\dots,n\}$ (i.e., we do not actually care about the identities of the colours, only the induced partition of the vertices), but we must be careful to do this ``in a canonical way'' (i.e., all choices can only depend on the isomorphism class of our graph). This issue is comprehensively handled by Berkholz, Bonsma and Grohe \cite{BBG17} (building on ideas in \cite{CC82,Hop71,PT87}), who present an algorithm that computes a canonical stable colouring of $G$ in time $O((n+m)\log n)$ (where $m$ is the number of edges in $G$).

\begin{theorem}\label{thm:CR-discrete-canonical}
    Let $\mc F\subseteq \mc G_n$ be the set of all graphs on the vertex set $\{1,\dots,n\}$ for which  $\mc R^* \sigma$ assigns each vertex a distinct colour. Then, the colour refinement algorithm defined in \cite{BBG17} is a canonical labelling scheme for $\mc F$. It computes a canonical labelling for a graph $G\in \mc F$ in time $O((n+m) \log n)$ (where $m$ is the number of edges in $G$).
\end{theorem}
We will use \cref{thm:CR-discrete-canonical} to prove \cref{thm:smoothed-colour-refinement}, but for \cref{thm:smoothed-disparity,thm:sparse-random-informal} we will need some more robust lemmas on canonical labelling. First, recall the definition of the \emph{disparity graph} from \cref{def:disparity}. As mentioned in \cref{subsec:disparity-outline}, we obtain an efficient canonical labelling scheme when the disparity graph can itself be efficiently labelled; we formalise this in \cref{prop:D-to-G} below, after introducing another definition.

\begin{definition}A \emph{canonical colouring scheme} $(c_G)_{G\in \mc G_n}$ is an assignment of a colouring $c_G:V(G)\to \Omega$ to each graph $G\in \mc G_n$, in such a way that for any graphs $G,H\in \mc G_n$, and any isomorphism $\psi$ from $G$ to $H$, we have $c_H(\psi(v))=c_G(v)$ for all $v\in V(G)$.

\end{definition}
Note that the colour refinement algorithm, starting from the trivial colouring, defines a canonical colouring scheme (i.e., let $c_G=\mc R_G^*\sigma$). However, in our proof of \cref{thm:smoothed-disparity} we will end up needing a slightly more sophisticated canonical colouring scheme (to be defined in \cref{subsec:WL}).

\begin{proposition}\label{prop:D-to-G}
    Let $\Phi_{\mc{H}}$ be a canonical labelling scheme for a graph family $\mathcal H\subseteq \mc G_n$, and let $(c_G)_{G\in \mc G_n}$ be a canonical colouring scheme, such that $\Phi_{\mc{H}}(G)$ and $c_G$ can both be computed in time $T$.
    
    Let $\mathcal F$ be the family of all graphs $G\in \mc G_n$ such that $D(G,c_G)\in \mc H$. Then there is a canonical labelling scheme $\Phi$ for $\mc F$, such that for every $G\in \mc F$, we can compute $\Phi(G)$ in time $O(n^2 +T)$.
\end{proposition}

\begin{proof}
    For a graph $G\in \mc F$, our canonical labelling $\Phi(G)\in \mc S_n$ is computed as follows. First compute the colouring $c_G$ (in time $T$).
    Then, compute the disparity graph $D(G,c_G)$ (this can be done in time $O(n^2)$). To each vertex $v$ associate the pair 
    \[\Big(c_G(v),\,\Phi_{\mc{H}}(D(G,c_G))(v)\Big).\]
    Then, obtain $\Phi(G)$ by ordering the vertices lexicographically by their associated pairs.
To see that this indeed describes a canonical labelling scheme, note that every automorphism of $D(G,c_G)$ which fixes the colouring $c_G$ is an automorphism of $G$.
\end{proof}

We will need to combine \cref{prop:D-to-G} with an efficient canonical labelling scheme for graphs with small connected components (it is easy to deal with each connected component separately, so this really comes down to worst-case guarantees for graphs on a fixed number of vertices). There are algorithms based on group-theoretic ideas with very strong theoretical guarantees (due to recent work of Babai~\cite{Bab19}), but as the purpose of this paper is to study the typical performance of \emph{combinatorial} algorithms, we will instead use the following theorem of Corneil and Goldberg~\cite{CG84}: with an ingenious recursion scheme, it is possible to canonically label graphs in exponential time (improving on na\"ive factorial-time algorithms).

\begin{theorem}\label{thm:CG}
There is a canonical labelling scheme $\Phi$ for $\mc G_n$ such that for every $G\in \mc G_n$, we can compute $\Phi(G)$ in time $\exp(O(n))$.
\end{theorem}
We record the following corollary of \cref{prop:D-to-G,thm:CG}.
\begin{corollary}\label{cor:small-components}
Fix a polynomial-time computable canonical colouring scheme $(c_G)_{G\in \mc G_n}$. Fix a constant $C$, and let $\mathcal F_C$ be the set of all graphs $G$ on the vertex set $\{1,\dots,n\}$, for which $D(G,c_G)$, has connected components with at most $C\log n$ vertices. Then there is a polynomial-time-computable canonical labelling scheme for $\mc F_C$.
\end{corollary}
For the reader unfamiliar with canonical labelling, we provide some details in \cref{app:labelling} for how exactly to deduce \cref{cor:small-components} from \cref{prop:D-to-G,thm:CG} (i.e., how to handle connected components separately).

Next, for \cref{thm:sparse-random-informal}, we need a generalisation of \cref{thm:CR-discrete-canonical} to a much larger class of graphs, as follows.
\begin{definition}
A graph $G$ is \textit{CR-determined} if the multiset of colours in $\mc R^{*} \sigma$ distinguishes $G$ from all non-isomorphic graphs (i.e., if there is no other graph $H$, not isomorphic to $G$, for which colour refinement produces the same multiset of colours).
\end{definition}
We remark that there is a combinatorial characterisation of CR-determined graphs (see \cite{AKRV17,amenablegraphs2}), though we will not need this. 
\begin{theorem}[see \cite{AKRV17}, Corollary 18]\label{thm:CR-determined-canonical}
    Let $\mc F\subseteq \mc G_n$ be the set of all CR-determined graphs on the vertex set $\{1,\dots,n\}$. Then there is a canonical labelling algorithm that computes a canonical labelling for a graph $G\in \mc F$ in polynomial time.
\end{theorem}
The canonical labelling algorithm for \cref{thm:CR-determined-canonical} is a slight variation on the colour refinement algorithm (due to Tinhofer~\cite{tinhofer} and Immerman and Lander~\cite{IL90}, and made canonical by Arvind, K\"{o}bler, Rattan and Verbitsky~\cite{AKRV17}; see also \cite{amenablegraphs2} for similar independent results). The idea is that if the colour refinement algorithm reaches a stable colouring in which two vertices have the same colour, we arbitrarily give one of those two vertices a unique colour, and continue.

To use \cref{thm:CR-determined-canonical} in our proof of \cref{thm:sparse-random-informal}, we will need a way to show that graphs are CR-determined. To this end, it \emph{almost} suffices to study the vertices in the kernel of $G$ (i.e., the vertices with degree at least 3 in the 2-core), as follows.
\begin{definition}\label{def:V23}
For a graph $G$, let $V_{2,3}(G)$ denote the vertices in the $2$-core of $G$ which have degree at least $3$.
\end{definition}
\begin{proposition}
\label{prop:kernel-canonical}
    Let $G$ be a connected graph with $V_{2,3}(G)\ne \emptyset$. If $\mc R^{*} \sigma$ assigns all vertices of $V_{2,3}(G)$ distinct colours, then $G$ is CR-determined.
\end{proposition}
It is fairly straightforward to prove \cref{prop:kernel-canonical} by reasoning carefully about how to reconstruct $G$ given $\mc R^{*} \sigma$ (first put bare paths between vertices of $V_{2,3}(G)$ to obtain $\on{core}_2(G)$, then attach trees to $\on{core}_2(G)$ to obtain $G$). For completeness, we provide some details in \cref{app:labelling}.

\begin{remark}\label{rem:automorphisms}
    With the same considerations, it is straightforward to characterise the automorphisms of a graph $G$ for which all vertices of $V_{2,3}(G)$ are assigned distinct colours. Indeed, the vertices of $V_{2,3}(G)$ are fixed by any automorphism, so the only flexibility comes from components with at most one cycle (i.e., components with no vertices of $V_{2,3}(G)$), bare paths between vertices of $V_{2,3}(G)$, and the trees attached to $\on{core}_2(G)$. This is relevant for \cref{thm:automorphisms}.
\end{remark}

\cref{prop:kernel-canonical} does not tell us anything about graphs with $V_{2,3}(G)= \emptyset$, but these graphs are easy to handle separately. Indeed, every connected graph with $V_{2,3}(G)= \emptyset$ either has no cycles (i.e., is a tree) or has exactly one cycle. Such graphs have very simple structure, and there are many ways to handle them; for example, all such graphs are \emph{outerplanar} (i.e., they have a planar embedding for which all vertices belong to the outer face of the embedding), so we can appeal to the following classical theorem of Sys\l o~\cite{Sys77}.
\begin{theorem}\label{thm:outerplanarcl}
    There is an $O(n)$-time algorithm to canonically label the family of $n$-vertex outerplanar graphs. 
\end{theorem}
Combining \cref{thm:outerplanarcl,prop:kernel-canonical,thm:CR-determined-canonical} yields the following corollary (the details of the deduction appear in \cref{app:labelling}).
\begin{corollary}\label{cor:kernel-canonical}
    Let $\mc F \subseteq \mc G_{n}$ be the set of graphs $G$ on the vertex set $\{1,\ldots,n\}$ for which $\mc R^{*} \sigma$ assigns vertices  of $V_{2,3}(G)$ distinct colours. Then there is a polynomial-time computable canonical labelling scheme for $\mathcal F$.
\end{corollary}

\subsection{Basic facts about colour refinement}\label{subsec:CR}
We collect some simple observations about colour refinement. First, we observe that stable colourings of $G$ are always \emph{equitable}\footnote{This should not be confused with a different meaning of the term ``equitable colouring'' more common in extremal graph theory (namely, that the colouring has a roughly equal number of vertices in each colour).}: $G$ is regular between each pair of colour classes.
\begin{definition}\label{def:equitable}
    For a graph $G$, a colouring $c:V(G)\to \Omega$ is \emph{equitable} if for any two distinct colour classes $A,B$ of $G$, the induced subgraph $G[A]$ is regular, and the induced bipartite subgraph $G[A,B]$ is a biregular bipartite graph. 
\end{definition}
\begin{fact}\label{fact:equitable}
    Consider any graph $G$ and any colouring $c:V(G)\to \Omega$ then $\mc R^*c$ is equitable. 

    In fact, $\mc R^*c$ is the \emph{coarsest} equitable partition refining $c$, in the sense that any equitable colouring $c'$ refining $c$ is also a refinement of $\mc R^* c$.
\end{fact}

The uniqueness of $\mc R^* c$ as the coarsest equitable partition refining $c$ (in \cref{fact:equitable}) has the following consequence for the ``consistency'' of the colour refinement algorithm: if we start the colour refinement algorithm from any coarsening of $\mc R^* \sigma$, it will still end up at $\mc R^* \sigma$.

\begin{fact}\label{fact:consistent}
    Consider any graph $G$, and let $c:V(G)\to \Omega$ be a colouring which is a coarsening of $\mc R^*\sigma$. 
    Then, $\mc R^*\sigma$ and $\mc R^*c$ define the same partition of $V(G)$. 
\end{fact}

Finally, we observe that the colour refinement algorithm can ``see'' which pairs of colours in a graph $G$ are connected by a path. This is also true for the disparity graph, and for a more general class of graphs in which we can prescribe for each pair of colour classes whether to complement the edges between those colour classes.

\begin{definition}\label{def:DL}
For a graph $G$, a vertex-colouring $V(G)\to \Omega$, and a matrix $L\in \{0,1\}^{\Omega\times \Omega}$, let $M_L(G,c)$ be the graph defined as follows. For any (possibly non-distinct) pair of colours $\omega,\omega'\in \Omega$:
\begin{itemize}
        \item If $L_{\omega,\omega'}=1$, then $M_L(G,c)$ contains every possible edge between vertices of colours $\omega$ and $\omega'$.
        \item If $L_{\omega,\omega'}=0$, then $M_L(G,c)$ contains no edges between vertices of colours $\omega$ and $\omega'$.
    \end{itemize}
    Then, define the \emph{generalised disparity graph} $D_L(G,c)=M_L(G,c)\triangle G$.
\end{definition}

\begin{fact}\label{fact:detect-paths}
    Fix a graph $G$, let $c=\mc R^*\sigma$, and let $\Omega$ be the set of colours used by $c$. Consider any $L\in \{0,1\}^{\Omega\times \Omega}$, any two vertices $u,v$ for which $c(u)=c(v)$, and any colour $\omega\in \Omega$. Then, with respect to the graph $D_L(G,c)$, there is a path from $u$ to an $\omega$-coloured vertex if and only if there is a path from $v$ to an $\omega$-coloured vertex.
\end{fact}
\begin{proof}[Proof sketch]
    The key point is that $c$ induces an equitable partition of $D_L(G,c)$ (indeed, $c$ is an equitable partition of $G$ by \cref{fact:equitable}, and complementing edges between pairs of colour classes does not change this property). So, if there is a path from $u$ to a colour class $C_\omega$ via some sequence of colour classes $C_1,\dots,C_\ell$, then by biregularity, through the same sequence of colour classes it is also possible to find a path from $v$ to $C_\omega$.
\end{proof}

We remark that the main reason we introduced the general notion of $L$-disparity graphs is that if we have two colourings $c,c'$ of a graph $G$, such that $c'$ is a refinement of $c$, then the disparity graph $D(G,c)$ with respect to $c$ can be interpreted as a generalised disparity graph $D_L(G,c')$ with respect to $c'$.

\subsection{The 2-dimensional Weisfeiler--Leman algorithm}\label{subsec:WL}
In this section we introduce a variant of colour refinement called the \emph{2-dimensional Weisfeiler--Leman algorithm}, which is a refinement algorithm for colourings of pairs of vertices, instead of colourings of single vertices\footnote{As the reader may have guessed, the colour refinement algorithm (also known as the 1-dimensional Weisfeiler--Leman algorithm) and the 2-dimensional Weisfeiler--Leman algorithm are both special cases of a more general \emph{$k$-dimensional Weisfeiler--Leman algorithm}, first defined in an influential paper of Cai, F\"urer and Immerman~\cite{CFI}. This will not be relevant for the present paper.}.

\begin{definition}\label{def:2WL}
For a graph $G$ and a colouring $f:V(G)^2\to \Omega$ of the ordered pairs of vertices of $G$, the \emph{2-dimensional refinement} $(2 \mc R)f:V(G)^2\to \Omega\times\mb N^\Omega$ of $f$ is defined by $(2\mc R) f(u,v)=(f(u,v),(d_{\omega_1,\omega_2}(u,v))_{(\omega_1,\omega_2)\in \Omega^2})$,
where $d_{\omega_1,\omega_2}(u,v)$ denotes the number of vertices $w\in V(G)$ such that $(f(u,w),f(w,v))=(\omega_1,\omega_2)$. As in \cref{def:refinement}, let $(2\mc R)^*f$ be the ``stable colouring'', obtained by iterating the refinement operation $2\mc R$ until we see the same vertex-pair partition twice in a row.

For a vertex colouring $c:V(G)\to \Omega$, we write $\phi_{G,c}:V(G)^2\to \Omega\sqcup \{0,1\}$  for the colouring defined by $\phi_{G,c}(u,v)=1$ if $uv$ is an edge, $\phi_{G,c}(u,v)=c(v)$ if $u=v$, and $\phi_G(u,v)=0$ otherwise (here we write $\sqcup$ for disjoint union). We write $\phi_G$ as shorthand for $\phi_{G,\sigma_G}$, recalling that $\sigma_G$ is the trivial vertex-colouring of $V(G)$ assigning every vertex the same colour.

Also, for a colouring $f:V(G)^2\to \Omega$, define the ``vertex-projection'' $\Pi f:V(G)\to \Omega$ by $\Pi f(v)=f(v,v)$. Let $\mc V^*=\Pi(2\mc R)^*$ be the vertex-projection of the 2-dimensional Weisfeiler--Leman algorithm.
\end{definition}

\begin{remark}
    There are (at least) two slightly different definitions of the 2-dimensional Weisfeiler--Leman algorithm, with different properties. The definition above was the original one introduced by Cai, F\"urer and Immerman~\cite{CFI}, and is sometimes called the ``2-folklore-Weisfeiler--Leman'' algorithm (see \cite{GNN1}).
\end{remark}

The first crucial fact we need about the 2-dimensional Weisfeiler--Leman algorithm is that it is efficiently computable. This was first observed by Cai, F\"urer and Immerman~\cite{CFI}.

\begin{fact}\label{fact:2WL-efficient}
For any graph $G$, the stable vertex-pair-colouring $(2\mc R)^*\phi_G$ can be computed in polynomial time. Moreover, this can be done ``canonically'' in the sense that different relabellings of the same graph always yield the same colouring $(2\mc R)^*\phi_G$.
\end{fact}

Specifically, Cai, F\"urer and Immerman~\cite{CFI} proved that the stable vertex-pair-colouring $(2\mc R)^*\phi_G$ can be computed in time $O(n^3\log n)$, but we will need an additional $O(\log n)$ ``overhead'' in order to make sure that all choices are made canonically (as discussed before \cref{thm:CR-discrete-canonical}, to prevent the sizes of the labels becoming too large we need to re-label the colours at each step; if we do not wish to make ``arbitrary choices'' then for this re-labelling we need a lexicographic sorting step). In the setting of the colour refinement algorithm, Bonsma, Berkholz and Grohe showed how to remove this ``$O(\log n)$ overhead''; this may also be possible for the 2-dimensional Weisfeiler--Leman algorithm, but we were not able to find this in the literature.

The next crucial fact we need about the 2-dimensional Weisfeiler--Leman algorithm (really, the entire reason we need to consider this algorithm, instead of the simpler colour refinement algorithm) is that it can ``detect distances between vertices'', generalising \cref{fact:detect-paths}. Specifically, $\mc V^*\phi_G(u)=(2\mc R)^*\phi_G(u,u)$ tells us the number of vertices of each colour at every given distance from $u$. This can be easily proved by induction (see also \cite[Theorem~2.6.7]{dist} for a proof in a somewhat more general framework). We will need a version of this fact for the generalised disparity graphs defined in \cref{def:DL}, which can be proved in the same way.
\begin{fact}\label{fact:detect-distances}
Fix a vertex-colouring $c$ of a graph $G$, and let $\Omega$ be the set of colours used by $\mc V^*c$. Consider any $L\in \{0,1\}^{\Omega\times \Omega}$, any two vertices $u,v$ for which $\mc V^*c(u)=\mc V^*c(v)$, any colour $\omega\in \Omega$ and any $i\in \mb N\cup \{\infty\}$. Then, with respect to the graph $D_L(G,\mc V^*c)$, the number of $\omega$-coloured vertices at distance $i$ from $u$ is exactly the same as the number of $\omega$-coloured vertices at distance $i$ from $v$.
\end{fact}

A special case of \cref{fact:detect-distances} (where $L$ is the all-0 matrix, and $i=1$) is that $\mc V^*\phi_{G,c}(u)$ is always an equitable partition. Since $\mc V^*\phi_{G,c}(u)$ is a refinement of $c$, it is a refinement of the stable colouring $\mc R^*c$ produced by the colour refinement algorithm, by \cref{fact:equitable}. (Informally, the 2-dimensional Weisfeiler--Leman algorithm is ``at least as powerful'' as the colour refinement algorithm).
\begin{fact}\label{fact:2WL-refines-CR}
    For any vertex-colouring $c$ of a graph $G$, the vertex-colouring $\mc V^*\phi_{G,c}$ is a refinement of $\mc R^* c$.
\end{fact}

Finally, we need an analogue of \cref{fact:consistent}, on the ``consistency'' of the 2-dimensional Weisfeiler--Leman algorithm: if we start the 2-dimensional Weisfeiler--Leman algorithm from any coarsening of $(2\mc R)^* f$, we will end up with a coarsening of $(2\mc R)^* f$. 
This is a direct consequence of, e.g., the discussion in \cite[Section~2.6.1]{dist} (see also \cite[Proposition~2.1]{FKV21}), which characterises the vertex-pair-partition defined by $(2\mc R)^* f$ as the unique coarsest coherent configuration refining $f$ (here ``coherent configurations'' are the 2-dimensional analogues of the equitable partitions defined in \cref{def:equitable}).
\begin{fact}\label{fact:2WL-consistent}
Fix any graph $G$ and any vertex-pair colouring $f:V(G)^2\to \Omega$. If $g:V(G)^2\to \Omega$ is a coarsening of $(2\mc R)^*f$, then $(2\mc R)^*g$ is a coarsening of $(2\mc R)^*f$. 
\end{fact}

\subsection{Probabilistic estimates}\label{subsec:prob}
At several points in the paper we will need some general-purpose probabilistic estimates. First, we need a Chernoff bound for sums of independent random variables; the following can be deduced from \cite[Theorem~2.1]{JLR00}.
\begin{lemma}[Chernoff bound]\label{lem:chernoff}
Let $X$ be a sum of independent random variables, each of which take values in $\{0,1\}$.
Then for any $\delta>0$ we have
\[\Pr[X\le (1-\delta)\mb{E}X]\le\exp(-\delta^2\mb{E}X/2),\qquad\Pr[X\ge (1+\delta)\mb{E}X]\le\exp(-\delta^2\mb{E}X/(2+\delta)).\]
\end{lemma}
We also need a Littlewood--Offord type \emph{anticoncentration estimate}, showing that certain sums of independent random variables are unlikely to take any particular value. Specifically, we need a generalisation of the classical  Erd\H os--Littlewood--Offord theorem to the ``$p$-biased'' setting, essentially due to Ju\v{s}kevi\v{c}ius and Kurauskas~\cite{JK21} and Singhal~\cite{SINGHAL}.

\begin{definition}
\label{def:modalprobability}
    Let $M(n,p)$ be the modal probability of a $\on{Binomial}(n,p)$ random variable, i.e., 
    \[M(n,p)=\binom{n}{x}p^{x}(1-p)^{n-x},\]
    for $x=\lfloor (n+1)p \rfloor$. (It is well-known---see for example \cite{KB80}---that $\lfloor (n+1)p \rfloor$ is always a mode for $\on{Binomial}(n,p)$).
\end{definition}

\begin{theorem}
\label{thm:LO}
Fix a positive integer $n$ and a probability $p \in (0,1/2)$. Let $a_{1},\ldots,a_{n}\in \mb R\setminus \{0\}$ be non-zero real numbers, and let $\xi_{1},\ldots,\xi_{n}\in \{0,1\}^n$ be independent $\on{Bernoulli}(p)$ random variables (meaning $\Pr[\xi_i=1]=p$ and $\Pr[\xi_i=0]=1-p$). Then
\[\sup_{x\in \mb R} \Pr[a_1\xi_1+\dots+a_n\xi_n=x]\le M(\lfloor n/2\rfloor,p)=O\left(\frac{1}{\sqrt{np}}\right).\]
\end{theorem}

The estimate $O(1/\sqrt{np})$ was first proved by Costello and Vu~\cite[Lemma~8.2]{CV08}. The sharper bound $M(\lfloor n/2\rfloor,p)$ can be deduced from a recent result due to Ju\v{s}kevi\v{c}ius and Kurauskas~\cite{JK21} and independently Singhal~\cite{SINGHAL}, as follows.

\begin{proof}[Proof of \cref{thm:LO}]
    First note that we may assume $n$ is even (otherwise, condition on any outcome of $\xi_{n-1}$ to reduce to the same question about $a_1\xi_1+\dots+a_{n-1}\xi_{n-1}$, noting in this case that $\lfloor (n-1)/2\rfloor=\lfloor n/2\rfloor$).

    The main result of \cite{SINGHAL,JK21} is that (given our assumption that $n$ is even), the quantity $\sup_{x\in \mb R} \Pr[a_1\xi_1+\dots+a_n\xi_n=x]$ is maximised when all of the $a_i$ are equal to $1$ or $-1$. By symmetry, we can assume $a_1,\dots,a_{\lfloor n/2\rfloor}=1$. Conditioning on arbitrary outcomes of $a_{\lfloor n/2\rfloor+1},\dots,a_n$, we see that the probability in question is bounded by 
    \[
        \sup_{x\in \mb R}\Pr\big[a_1\xi_1+\dots+a_{\lfloor n/2\rfloor}\xi_{\floor{n/2}}=x\big]=M(\lfloor n/2\rfloor,p).
    \]
    (here we used that $a_1,\dots,a_{\lfloor n/2\rfloor}=1$ to see that $a_1\xi_1+\dots+a_{\lfloor n/2\rfloor}\xi_{\floor{n/2}}\sim \on{Binomial}(\lfloor n/2\rfloor ,p)$).
\end{proof}

\section{Encoding colour refinement via views}\label{sec:universal-covers}
There are many (quite different) lenses from which to view colour refinement. 
In particular, an important way to understand the performance of the colour refinement algorithm (around a vertex $v$ in a graph $G$) is via a tree $\mc T_G(v)$ called the \emph{view} (first defined by Yamashita and Kameda~\cite{views}).
\begin{definition}\label{def:universal-cover}
Let $G$ be a connected graph and $v \in V(G)$. Let $\mathcal{T}_G(v)$ be the (infinite, unless $v$ is an isolated vertex) rooted tree defined as follows.
\begin{compactitem}
    \item The vertex set of $\mathcal{T}_G(v)$ is the set of all \emph{walks} $W = (x_{0},\ldots,x_{n})$ starting at $x_0=v$ (i.e., we require that $x_{i}x_{i+1}$ is an edge of $G$ for all $i \in \{0,\ldots,n-1\}$).
    \item The root vertex of $\mathcal{T}_G(v)$ is the trivial walk $W=(v)$.
    \item We say that a walk $W$ is a child of another walk $W'$ (and put an edge between $W$ and $W'$) if $W'$ is a one-step extension of $W$ (i.e., if we can write $W=(x_{0},\ldots,x_{n})$ and $W' = (x_{0},\ldots,x_{n},x_{n+1})$).
\end{compactitem}
Note that the set of vertices at depth $\ell$ in $\mathcal{T}_G(v)$ correspond precisely to the set of walks of length $\ell$ starting from $v$ in $G$. For $i \geq 0$, we write $\mathcal{T}^i_G(v)$ for the subtree of $\mathcal{T}_G(v)$ restricted to walks of length at most $i$. See \cref{fig:L-example} for an example.

We note that it will sometimes be convenient to view each vertex $u=(x_{0},\ldots,x_{n})$ of $\mathcal{T}_G(v)$ as a ``copy'' of the vertex $x_n$.
\end{definition}

The following lemma provides a connection between views and colour refinement. It was first observed by Angluin~\cite{Ang80} in a slightly different context (the particular statement here follows directly from \cite[Lemma~2.6]{Universalcover}).

\begin{lemma}
\label{Universialtreecoverstep}
Let $G$ be a connected graph, and $u,v$ two vertices. Then we have the rooted tree isomorphism $\mathcal{T}^i_G(u) \cong \mathcal{T}^i_G(v)$ if and only if $\mc R^i \sigma(u)=\mc R^i \sigma(v)$ (i.e., if $u$ and $v$ have the same colour after $i$ steps of colour refinement starting from the trivial colouring $\sigma$). In particular, $\mathcal{T}_G(u) \cong \mathcal{T}_G(v)$ if and only if $u$ and $v$ have the same colour in the stable colouring $\mc R^*\sigma$.
\end{lemma}
\begin{remark}\label{rem:view-vs-UC}
    \cite[Lemma~2.6]{Universalcover} is stated for \emph{universal covers}, which have the same definition as views (as in \cref{def:universal-cover}) except that one only considers \emph{non-backtracking} walks $W = (x_{0},\ldots,x_{n})$ with $x_{i+1}\ne x_{i-1}$ for all $i\in \{1,\dots,n-1\}$. As discussed in \cite{Universalcover}, it is easy to see that views and universal covers encode the same information.
\end{remark}
For our purposes, it will not be very convenient to work with views directly. Instead, we consider sequences of multisets $\mc L_G^i(u,v)$ which describe the difference between two views $\mc T_G^i(u)$ and $\mc T_G^i(v)$ at depth $i$.

\begin{definition}\label{def:multisets}
Let $G$ be a graph. 
For distinct $u,v \in V(G)$ and $i \geq 0$, we recursively define multisets $\mc L^i_G(u,v)$, as follows. 
\begin{compactitem}
    \item For all $u,v$, let $\mc L_G^{0}(u,v): = \{u\}$.
    \item For $u,v,w \in V(G)$ and $i \geq 1$,
    let $\ell_G^i(w,u,v)$ be the number of neighbours of $w$ in $\mc L_G^{i-1}(u,v)$.
    \item For $u,v\in V(G)$ and $i \geq 1$, 
    we define $\mc L_G^{i}(u,v)$ as follows. For each $w\in V(G)$,
    if $\ell_G^{i}(w,u,v) > \ell_G^{i}(w,v,u)$, then put $\ell_G^{i}(w,u,v) -\ell_G^{i}(w,v,u)$ copies of $w$ in $\mc L_G^{i}(u,v)$.
\end{compactitem}
\begin{figure}
\begin{center}
\begin{tikzpicture}
    \node[dummywhite] at (0,0) (dummy1) {$G$};
    \node[blackvertexv2] at (90:1.5cm) (u0) [label = left:$u_{0}$]{};
    \node[blackvertexv2] at (90+90:1.5cm) (u1) [label =above:$u_{1}$] {};
    \node[blackvertexv2] at (90+2*90:1.5cm) [label = above:$u_{2}$] (u2) {};
    \node[blackvertexv2] at (90+3*90:1.5cm) [label = above:$u_{3}$] (u3) {};
    \draw[thick,black] (u0)--(u1)--(u2)--(u3)--(u0);
    \node[blackvertexv2] at (0,2.5) (u4) [label = left:$u_{4}$] {};
    \node[blackvertexv2] at (-1,3.5) (u5) [label = left:$u_{5}$]{};
    \node[blackvertexv2] at (1,3.5) (u6) [label = left:$u_{6}$] {};
    \draw[thick,black] (u0)--(u4)--(u5);
    \draw[thick,black] (u4)--(u6);
    \begin{scope}[xshift = 3.5cm,yshift=-0.5cm]
    \node[blackvertex] at (0,0) (v0) [label = above:$u_{2}$] {};
    \node[blackvertex] at (1,1) (v1) [label = above:$u_{1}$] {};
    \node[blackvertex] at (1,-1.2) (v2) [label = above:$u_{3}$] {};
    \node[blackvertex] at (2,1.5) (v3) [label = above:$u_{2}$] {};
    \node[blackvertex] at (2,0.5) (v4) [label = above:$u_{0}$] {};
    \node[blackvertex] at (3,2.2) (v5) [label = right:$u_{1}$] {};
    \node[blackvertex] at (3,1.75) (v6) [label = right:$u_{3}$] {};
    \node[blackvertex] at (3,1) (v7) [label = right:$u_{1}$] {};
    \node[blackvertex] at (3,0.5) (v8) [label = right:$u_{3}$] {};
    \node[blackvertex] at (3,0) (v9) [label = right:$u_{4}$] {};
    \node[blackvertex] at (2,-.8) (v10) [label = above:$u_{2}$] {};
    \node[blackvertex] at (2,-1.7) (v11) [label = above:$u_{0}$] {};
    \node[blackvertex] at (3,-.5) (v12) [label = right:$u_{1}$] {};
    \node[blackvertex] at (3,-.9) (v14) [label =right:$u_{3}$] {};
    \node[blackvertex] at (3,-1.6) (v13) [label = right:$u_{1}$] {};
    \node[blackvertex] at (3,-2) (v15) [label = right:$u_{3}$] {};
    \node[blackvertex] at (3,-2.4) (v16) [label = right:$u_{4}$] {};
    \node[dummywhite] at (1.5,-2) (dummy1) [label = left:$\mathcal{T}^{3}_{G}(u_{2})$] {};
    \draw[thick,black] (v0)--(v1)--(v3)--(v5);
    \draw[thick,black] (v3)--(v6);
    \draw[thick,black] (v1)--(v4)--(v7);
    \draw[thick,black] (v4)--(v8);
    \draw[thick,black] (v4)--(v9);
    \draw[thick,black] (v0)--(v2)--(v10)--(v12);
    \draw[thick,black] (v10)--(v14);
    \draw[thick,black] (v2)--(v11)--(v13);
    \draw[thick,black] (v11)--(v15);
    \draw[thick,black] (v11)--(v16);
    \end{scope}
    \begin{scope}[xshift =3.5cm, yshift = 3cm]
    \node[blackvertex] at (0,0) (v0) [label = above:$u_{5}$] {};
    \node[blackvertex] at (1,0) (v1) [label = above:$u_{4}\;\;$] {};
    \node[blackvertex] at (2,-0.5) (v2) [label = above:$u_{5}$] {};
    \node[blackvertex] at (2,0.5) (v6) [label = above:$u_{6}$] {};
    \node[blackvertex] at (3,-0.5) (v3) [label = right:$u_{4}$] {};
    \node[blackvertex] at (3,0.5) (v4) [label = right:$u_{4}$] {};
    \node[blackvertex] at (2,1.5) (v5) [label = above:$u_{0}$] {};
     \node[blackvertex] at (3,1.2) (v7) [label = right:$u_{4}$] {};
    \node[blackvertex] at (3,1.7) (v8) [label = right:$u_{3}$] {};
    \node[blackvertex] at (3,2.2) (v9) [label = right:$u_{1}$] {};
    \draw[thick,black] (v0)--(v1)--(v2)--(v3);
    \draw[thick,black] (v1)--(v6);
    \draw[thick,black] (v6)--(v4);
    \draw[thick,black] (v1)--(v5);
    \draw[thick,black] (v5)--(v7);
    \draw[thick,black] (v5)--(v8);
    \draw[thick,black] (v5)--(v9);
    \node[dummywhite] at (1.5,-0.75) (dummy1) [label = left:$\mathcal{T}^{3}_{G}(u_{5})$] {};
    \end{scope}
    \begin{scope}[xshift = 10.5cm]
    \node[dummywhite] at (-0.6,1.5) (dummy1) [label = {$\mathcal{L}^{0} (u_{2},u_{5}) \equal \{u_{2}\}$}] {};
    \node[dummywhite] at (-0.32,1) (dummy2) [label = {$\mathcal{L}^{1}(u_{2},u_{5}) \equal \{u_{1},u_{3}\}$}] {};
    \node[dummywhite] at (-.04, 0.5) (dummy3) [label = {$\mathcal{L}^{2}(u_{2},u_{5}) \equal \{u_{0},u_{2},u_{2}\}$}]{};
    \node[dummywhite] at (0.8,0) (dummy4) [label = {$\mathcal{L}^{3}(u_{2},u_{5}) \equal \{u_{1},u_{1},u_{1},u_{3},u_3,u_3\}$}] {};
    
    \node[dummywhite] at (-0.6, 4.5) (dummy5) [label = {$\mathcal{L}^{0}(u_{5},u_{2}) \equal \{u_{5}\}$}] {};
    \node[dummywhite] at (-0.6,4) (dummy6) [label = {$\mathcal{L}^{1}(u_{5},u_{2}) \equal \{u_{4}\}$}] {};
    \node[dummywhite] at (-0.32,3.5) (dummy7) [label = {$\mathcal{L}^{2}(u_{5},u_{2}) \equal \{u_5,u_{6}\}$}] {};
    \node[dummywhite] at (-0.6,3) (dummy8) [label = {$\mathcal{L}^{3}(u_{5},u_{2}) \equal \{u_{4}\}$}] {};

    \node[dummywhite] at (-0.51,-1.5) (dummy9) [label = {$\mathcal{S}^{0}(\{u_{2},u_{5}\}) \equal \{u_{2},u_{5}\}$}] {};
    \node[dummywhite] at (-0.23,-2) (dummy10) [label = {$\mathcal{S}^{1}(\{u_{2},u_{5}\}) \equal \{u_{1},u_{3},u_{4}\}$}] {};
        \node[dummywhite] at (-0.51,-2.5) (dummy9) [label = {$\mathcal{S}^{2}(\{u_{2},u_{5}\}) \equal \{u_{0},u_{6}\}$}] {};
            \node[dummywhite] at (-1.08,-3) (dummy9) [label = {$\mathcal{S}^{3}(\{u_{2},u_{5}\}) \equal \emptyset$}] {};

\end{scope}
\end{tikzpicture}
\caption{\label{fig:L-example}An example graph $G$, and the depth-$3$ views for $u_{5}$ and $u_{2}$. As these trees are non-isomorphic, one can deduce that $u_{2}$ and $u_{5}$ are assigned different colours after three steps of the colour refinement algorithm (in fact, they already receive different colours after the first step). In the figure we show the multisets $\mc L^i(u_2,u_5)$, $\mc L^i(u_5,u_2)$ and the sets $\mathcal{S}^{i}(\{u_{2},u_{5}\})$ for $i\in \{0,1,2,3\}$.}
\end{center}
\end{figure}
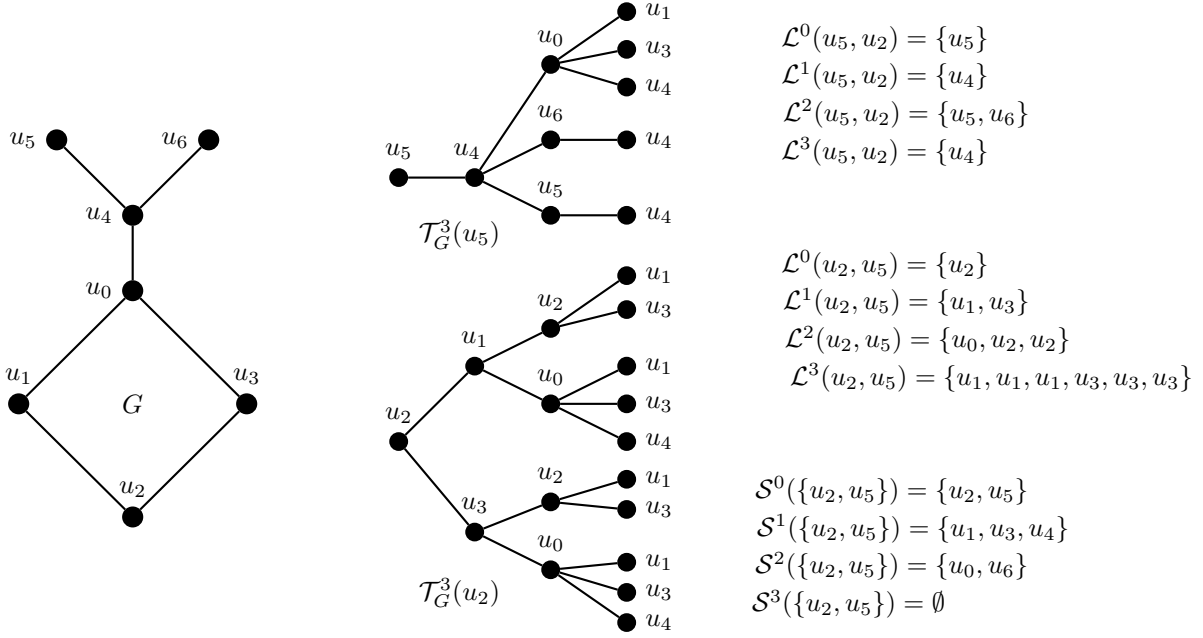

Then, write $\on{supp}\mc L_G^{i}(u,v)$ for the support of $\mc L_G^{i}(u,v)$, that is, the set of elements appearing in $\mc L_G^{i}(u,v)$, and write
\[\mc S_G^{\leq i}(\{u,v\}) := \bigcup_{j\leq i}\left(\on{supp}\mc L_G^{j}(u,v)\cup \on{supp}\mc L_G^{j}(v,u)\right),\quad \mc S_G^{i} (\{u,v\}):=\mc S_G^{\le i}(\{u,v\})\setminus \mc S_G^{\le i-1}(\{u,v\}).
\]
Note that $\on{supp}\mc L_G^i(u,v)\cup \on{supp}\mc L_G^i(v,u)$ is a \emph{disjoint} union: by definition, there are no vertices with positive multiplicity in both $\mc L_G^i(u,v)$ and $\mc L_G^i(v,u)$. We think of $\mc S_G^i(\{u,v\})$ as being
the set of vertices which ``appear differently'' in the views $\mc T_G(u)$ and $\mc T_G(v)$ for the first time at depth $i$.
\end{definition}

See \cref{fig:L-example} for an example. For the notation introduced in \cref{def:universal-cover,def:multisets}, we will often omit the subscript ``$G$'' when it is clear from context.

\cref{Universialtreecoverstep} has the following consequence in terms of our multisets $\mc L^i(u,v)$.

\begin{lemma}\label{lem:receive_same_colour}
    Let $G$ be a graph and consider distinct $u,v\in V(G)$. If there exist positive integers $i$ and $d$ such that the number of copies of vertices of degree $d$ in  $\mc L^i(u,v)$ and $\mc L^i\mc (v,u)$ are not equal then $\mc R^* \sigma(u) \neq \mc R^* \sigma(v)$.
\end{lemma}
\cref{lem:receive_same_colour} is in some sense self-evident (the multisets $\mc L^i(u,v)$ describe the discrepancy between the $i$th layers of $\mc T(u)$ and $\mc T(v)$), but for completeness we provide a formal proof.
\begin{proof}[Proof of \cref{lem:receive_same_colour}]
Recalling that the vertices of each $\mc T(v)$ can be viewed as copies of vertices of $G$, for each $d,i,u,v$ let $\mc M_d^i(v)$ be the multiset of vertices of degree $d$ appearing on the $i$-th layer
of $\mc T(v)$, and let  $\mc L_d^i(u,v)$ be the multiset of vertices of degree $d$ in $\mc L^i(u,v)$.

We will prove the contrapositive of the lemma statement. Fix $u,v$ and suppose that $\mc R^* \sigma(u)=\mc R^* \sigma(v)$; our objective is to prove that $|\mc L_d^i(u,v)|=|\mc L_d^i(v,u)|$ for all $i,d$. To this end, we will prove (by induction on $i$) that for all $d$ there is a multiset $W_d^i(\{u,v\})$ (of degree-$d$ vertices common to the $i$-th layers of $\mc T(u),\mc T(u)$) such that $\mc M_d^i(u)=\mc L_d^i(u,v)\cup W_d^i(\{u,v\})$ and $\mc M_d^i(v)=\mc L_d^i(v,u)\cup W_d^i(\{u,v\})$ (these are multiset unions, respecting multiplicity). By \cref{Universialtreecoverstep}, we have $|\mc M_d^i(u)|=|\mc M_d^i(v)|$, so this suffices.

As the ($i=0$) base case for the induction, note that $\mc R^* \sigma(u)=\mc R^* \sigma(v)$ implies that $u,v$ have the same degree, thus the statement holds with $ W_d^i(\{u,v\})=\emptyset$. So, consider $i \geq 1$, and assume that for all $d$ there exists a set $W_d^{i-1}(\{u,v\})$ such that 
$$ \mc M_d^{i-1}(u)= W_d^{i-1}(\{u,v\})\cup \mc L_d^{i-1}(u,v) \text{ and }  \mc M_d^{i-1}(v)= W_d^{i-1}(\{u,v\})\cup \mc L_d^{i-1}(v,u) \text{ for } d\geq 1.$$  
Now, define multisets of vertices $X^i_d(\{u,v\})$ and $Y^i_d(\{u,v\})$ as follows. For $w\in V(G)$, of degree $d(w)=d$, the number of copies of $w$ in $X^i_d(\{u,v\})$ is precisely $\min\{\ell_G^{i}(w,u,v),\ell_G^{i}(w,v,u)\}$, and the number of copies of $w$ in $Y^i_d(\{u,v\})$ is the number of neighbours of $w$ (counted with multiplicity) in the multiset union $\bigcup_{d\geq 1} W_d^{i-1}(\{u,v\})$. Recalling the definitions of $\mc L^i(u,v)$ and $\mc L^i(u,v)$ (as the ``surplus'' of vertices in the $i$-th layer of $\mc T(u)$ compared to $\mc T(v)$, and vice versa), and recalling the inductive assumption, we have $\mc M_d^i(v)=\mc L_d^i(v,u)\cup X_d^i(\{u,v\})\cup Y_d^i(\{u,v\})$ and $\mc M_d^i(u)=\mc L_d^i(u,v)\cup X_d^i(\{u,v\})\cup Y_d^i(\{u,v\})$. So we can take $W_d^i(\{u,v\})=X_d^i(\{u,v\})\cup Y_d^i(\{u,v\})$, for the desired conclusion. 
\end{proof}

\section{Expansion for views}\label{sec:expansion}

Recall the notation $\mc S^{i}(\{u,v\})$ defined in \cref{def:multisets}, describing the set of vertices which appear differently in the views $\mc T(u)$ and $\mc T(v)$ for the first time at depth $i$.
\\

In this section we prove that, for two vertices $u,v$ in an appropriately randomly perturbed graph $G$, it is very likely that whenever $\mc S^{\leq i}(\{u,v\})$ grows reasonably large, $u$ and $v$ end up with different colours in the stable colouring $\mc R^*\sigma$. This is formalised in the following proposition.

\begin{proposition}\label{prop:expansion-colour-refinement}
    Let $G_0$ be a graph on the vertex set $\{1,\dots,n\}$ and let $G_{\mr{rand}}\sim \mb G(n,p)$, for some $p\in [0,1/2]$ satisfying
    \[\frac{1+(\log n)^{-40}}{n}\le p \le \frac{100\log n}n.\]
    Then whp $G=G_0\triangle G_{\mr{rand}}$ satisfies the following property. For every pair of vertices $u,v$ for which there is some $i$ satisfying
    \begin{equation}\label{eq:condition:expansion}
  |\mc S^{\le i}_G(\{u,v\})|\geq\frac{100\log n}{\min\{(np-1)^2,1\} },
    \end{equation}
we have $\mc R^*_{G} \sigma(u) \neq \mc R^*_{G} \sigma (v)$.
\end{proposition}
We remark that, while we did not necessarily make an effort to optimise the assumption on $p$, there is some significance to the ``$1/n$''
essentially appearing in \cref{prop:expansion-colour-refinement}. Indeed, this corresponds to the celebrated \emph{phase transition} for Erd\H os--R\'enyi random graphs (as briefly discussed in \cref{subsec:sparse-random-outline}).

Morally speaking, the main reason why \cref{prop:expansion-colour-refinement} is true is that random graphs are extremely good \emph{expanders}: starting from any reasonably large set of vertices $S$, the set of vertices within distance $i$ of $S$ grows geometrically in $i$ until it spans nearly the whole graph. The following lemma (a key ingredient in our proof of \cref{prop:expansion-colour-refinement}) is a counterpart to this fact for the  sets $\mc S^{i}(\{u,v\})$ in a randomly perturbed graph: whenever $\mc S^{\le i}(\{u,v\})$ is reasonably large, there is some $j$ for which $\mc S^{j}(\{u,v\})$ is huge, spanning a significant proportion of the whole graph.
\begin{lemma}\label{lem:expansion_SLi}
Let $G_0,p,G_{\mr{rand}}$ be as in \cref{prop:expansion-colour-refinement}. Then whp $G=G_0\triangle G_{\mr{rand}}$ satisfies the following property. For every pair of vertices $u,v$ for which there is some $i$ satisfying \cref{eq:condition:expansion}, there is some $j$ satisfying
\begin{equation}\label{eq:condition:expansion:final_set}
 \big|\mc S^{j}_G(\{u,v\}) \big|\geq \frac{n}{4 (\log n)^{140}}.
    \end{equation}
\end{lemma}

We defer the proof of \cref{lem:expansion_SLi} to the end of the section. Given \cref{lem:expansion_SLi}, the strategy for the proof of \cref{prop:expansion-colour-refinement} is as follows. For each $u,v$ satisfying \cref{eq:condition:expansion}, and $j$ as guaranteed by \cref{lem:expansion_SLi}, we first observe that it is possible to reveal the set $\mc S^{j}_{G}(\{u,v\})$ without revealing too much information about the random perturbation $G_{\mr{rand}}$ inside this set. We wish to use the remaining randomness, together with a Littlewood--Offord-type anticoncentration inequality (\cref{thm:LO}), to show that it is very unlikely that the multisets $\mc L^j_{G}(u,v)$ and $\mc L^j_{G}(v,u)$ have the same degree statistics.

This final step is actually rather delicate. Indeed, if we let $Q_d^i(u,v)$ be the number of vertices of degree $d$ in $\mc L^i(u,v)$, minus the corresponding number in $\mc L^i(v,u)$, it is not too hard to see that $\sum_d d Q_d^i(u,v)$ (which is the number of edges incident to $\mc L^i(u,v)$ plus the number of edges spanned by $\mc L^i(u,v)$) can be expressed as a weighted sum of Bernoulli random variables (as required for \cref{thm:LO}). So, we can use \cref{thm:LO} to show that usually we have $\sum_d d Q_d^i(u,v)\ne 0$ (in which case some $Q_d^i(u,v)\ne 0$, so the degree statistics of $\mc L^i(u,v)$ and $\mc L^i(v,u)$ must be different). However, this estimate is not good enough for a union bound over choices of $u,v$, and it is necessary to understand in much more detail the \emph{joint} anticoncentration of $dQ_d^i(u,v)$, between different $d$. This is accomplished by first revealing almost all information about $G_{\mr{rand}}$ except the edges inside a tiny subset of $\mc S^{j}_{G}(\{u,v\})$ (which we call a ``hole''). In the resulting conditional probability space, we obtain fairly weak anticoncentration bounds on the $dQ_d^i(u,v)$, but we can show that $dQ_d^i(u,v)$ is essentially independent from $d'Q_{d'}^i(u,v)$ whenever $d$ and $d'$ are reasonably far apart (so we can multiply weak anticoncentration bounds together for a much stronger final bound).

The details of the proof of \cref{prop:expansion-colour-refinement} are as follows.
 \begin{proof}[Proof of \cref{prop:expansion-colour-refinement}]
 
Fix vertices $u,v$. Our objective is to prove that with probability $1-o(n^{-2})$, there is no $i$ satisfying \cref{eq:condition:expansion}, or the property in \cref{lem:expansion_SLi} fails to hold, or there are some $j,d$ such that the number of copies of vertices of degree $d$ in $\mc L^j(u,v)$ and $\mc L^j\mc (v,u)$ are not equal. The desired result will then immediately follow by a union bound over choices of $u,v$, and \cref{lem:expansion_SLi,lem:receive_same_colour}.

\medskip
\noindent\textbf{Step 1: Exploration.} 
We reveal edges of $G_{\mr{rand}}$ according to the following procedure. For each $j$ (starting at $j=0$ and increasing $j$ by one at each step), if \cref{eq:condition:expansion:final_set} does not yet hold, then reveal all edges of $G_{\mr{rand}}$ incident to $\mc S^{j}(\{u,v\})$. The procedure terminates at the first step $j$ for which either \cref{eq:condition:expansion:final_set} holds or $\mc S^{j}(\{u,v\})=\emptyset$.
We fix this $j$ for the rest of the proof; from now on, we will treat $j,  \mc S^{j}(\{u,v\}), \mc S^{\leq j-1}(\{u,v\})$ as non-random objects (they are determined by the aforementioned revelation process). Note that we have not yet revealed anything about the edges of $G_{\mr{rand}}$ inside 
$\mc S^{j}(\{u,v\})$.

If there is $i$ satisfying \cref{eq:condition:expansion}, and if the property in \cref{lem:expansion_SLi} holds, then the latter of the two termination conditions cannot happen: we can assume that \cref{eq:condition:expansion:final_set} holds.

\medskip
\noindent\textbf{Step 2: Setting up a small ``hole''.} 
Recall from \cref{def:multisets} that $\mc S^{j}(\{u,v\})$ is defined in terms of the disjoint supports $\on{supp} \mc L^{j}(u,v)$ and $\on{supp} \mc L^{j}(v,u)$. 
Let $A=\mc S^{j}(\{u,v\}) \cap \on{supp} \mc L^{j}(u,v) $,
and assume without loss of generality $A$ is larger than 
$\mc S^{j}(\{u,v\}) \cap \on{supp} \mc L^{j}(v,u)$, 
so by \cref{eq:condition:expansion:final_set} we have
\begin{equation}
    |A|\geq \frac{n}{ 8 (\log n)^{140}}.\label{eq:Asize}
\end{equation}
Let $S$ be a subset of $n^{0.75}$ vertices of $A$. Let $G_1=G_0\triangle (G_{\mr{rand}}-G_{\mr{rand}}[S])$ be the graph obtained by randomly perturbing $G_0$ via all the edges of $G_{\mr{rand}}$ except those inside $S$.

The idea is that, since $S$ is so small, the degrees of vertices in $S$ with respect to $G_1$ are a very good prediction of the degrees with respect to $G$. So, if we reveal $G_1$ (leaving only the randomness inside $S$), we might hope to be able to find disjoint ``buckets'' of vertices in $S$, whose degrees with respect to $G_1$ are so far apart from each other that it is highly unlikely that two distinct buckets will have vertices with the same degree in $G$. This means we can handle each bucket essentially independently, which will yield strong joint anticoncentration bounds.

\medskip
\noindent\textbf{Step 3: Specifying degree buckets.}
The following claim will be used to construct these buckets (essentially, it shows that the degree distribution of the vertices of $S$ with respect to $G_1$ is ``spread out''). Let $V(d)$ be the set of vertices in $S$ which have degree $d$ with respect to $G_1$, and for a set $D\subseteq \mb N$, let $V(D)=\bigcup_{d\in D}V(d)$. Also let $n(d)=|V(d)|$ and $n(D)=|V(D)|$.
\begin{claim}\label{claim:degrees_G^1}
With probability $1-o(n^{-2})$ (over the remaining randomness of $G_1$; i.e., conditioned on the edges that have already been revealed), for every set $D\subseteq \{0,\dots,n\}$ of size $100^2$ we have $|S|-n(D)> n^{0.7}$ (i.e., there are many vertices whose degrees do not lie in $D$).
\end{claim}
\begin{claimproof}
Fix a set $D\subseteq \{0,\dots,n\}$ of size $100^2$. We will show that the desired property holds for this $D$ with probability at least $1-o(n^{-100^2-2})$, so the desired result will follow from the union bound.

Let $T$ be a subset of $n/(10(\log n)^{140})$ vertices of $A$ disjoint from $S$ (such a $T$ exists by \cref{eq:Asize}).

First, reveal all edges of $G_1$ except the edges between $S$ and $T$. In the resulting conditional probability space, the degree of each vertex $x\in S$ (with respect to $G_1$) only depends on the edges between $x$ and $T$ in $G_{\mr{rand}}$ (and thus these degrees are independent): we can write
\[\deg_{G_1}(x)=b_x+c_x+E_x^+-E_x^-,\]
where
\begin{compactitem}
    \item $b_x$ is the number of neighbours of $x$ outside $T$ with respect to $G_1$,
    \item $c_x$ is the number of neighbours of $x$ inside $T$ with respect to $G_0$,
    \item $E_x^+$ is the number of neighbours of $x$ inside $T$ with respect to $G_{\mr{rand}}$, among non-neighbours of $x$ with respect to $G_0$.
    \item $E_x^-$ is the number of neighbours of $x$ inside $T$ with respect to $G_{\mr{rand}}$, among neighbours of $x$ with respect to $G_0$.
\end{compactitem}
Note that $E_x^+\sim \on{Binomial}(e_x^+,p)$ and $E_x^-\sim \on{Binomial}(e_x^-,p)$, where $e_x^+$ (respectively $e_x^-$) is the number of non-neighbours (respectively, neighbours) of $x$ inside $T$ with respect to $G_0$.

Now, for each $x\in S$, we want to show that it is reasonably likely that $\deg_{G_1}(x)\notin D$. Fix some $x$, and assume without loss of generality that $e_x^+\ge e_x^-$ (so $e_x^+=n^{1-o(1)}$). Since $|D|= 100^2$, there is some $d_x\notin D$ such that $0\le d_x-(b_x+c_x)\le 100^2$, meaning that one way to have $\deg_{G_1}(x)\notin D$ is to have $E_x^-=0$ and $E_x^+=d_x-(b_x+c_x)$. We can compute the probability that this occurs (in our conditional probability space, given information revealed so far).
\begin{align*}
    \Pr[\deg_{G_1}(x)\notin D]\ge \Pr[\deg_{G_1}(x)=d_x]&\ge \Pr[E_x^-=0\text{ and }E_x^+=d_x-(b_x+c_x)]\\
    &\ge \min_{0\le k\le 100^2}(1-p)^{e_x^-}\binom{e_x^+}k(1-p)^{e_x^+-k}p^k=n^{-o(1)}\ge n^{-0.01},
    \end{align*}
recalling that $p=n^{o(1)}/n$ and $e_x^+=n^{1-o(1)}$.

Now, $|S|-n(D)$ is the number of $x\in S$ for which $\deg_{G_1}(x)\notin D$. In our conditional probability space, recall that the events $\deg_{G_1}(x)\notin D$ are independent. We have $\mb E[|S|-n(D)] \ge |S|n^{-0.01}\ge n^{0.73}$, so by a Chernoff bound (\cref{lem:chernoff}) we have
\[\Pr[|S|-n(D) <  n^{0.7}]\le \exp(-\Omega(n^{0.73}))=o(n^{-100^2-2}),\]
as desired.
\end{claimproof}

Now, we use \cref{claim:degrees_G^1} to construct many large sets of vertices with very distinct degrees with respect to $G_1$, as follows.

\begin{claim}\label{claim:intervals}
If $n$ is sufficiently large and $G_1$ satisfies the conclusion of \cref{claim:degrees_G^1}, then there exists a sequence of pairwise disjoint sets $D_1,\dots,D_{33}\subseteq \mb N$ such that $n(D_t)\ge n^{0.6}$ for all $t$, and such that for each $s\ne t$ and each $d\in D_s,d'\in D_t$ we have $|d-d'|>200$.
\end{claim}
\begin{claimproof}
Let $D_{\mr{big}}$ be the set of integers $d$ such that $n(d)\geq n^{0.6}$. We split into cases depending on the size of $D_{\mr{big}}$.

\medskip
\noindent\textbf{Case 1: $|D_{\mr{big}}|\geq 100^2$.} In this case, we can simply order the elements of $D_{\mr{big}}$ as $d_1\leq d_2\leq \cdots \leq d_{\ell}$, and take $D_t=\{d_{201t}\}$ for each $t\in \{1,\dots,33\}$.

\medskip
\noindent\textbf{Case 2: $|D_{\mr{big}}|<100^2$.}
In this case we greedily construct a sequence of nonnegative integers $a_1\le \dots\le a_{100}$, such that $n(\{a_r,a_r+1,a_r+2,\dots,a_{r+1}-1\})\ge n^{0.6}$ and $a_{r+1}-a_r\geq 101$
for each $r\in \{1,\dots,99\}$; this suffices, as we can then take $D_t=\{a_{3t},a_{3t}+1,\dots,a_{3t+1}-1\}$ for each $t\in \{1,\dots,33\}$.

Specifically, we let $a_1=0$, and for $r\ge 1$ we recursively define $a_{r+1}$ to be the minimal integer larger than $a_r+100$ such that $n(\{a_r,a_r+1,a_r+2,\dots,a_{r+1}-1\})\ge n^{0.6}$. To see that such an $a_{r+1}$ always exists, note that
\[n(\{0,1,2,\dots,a_{r}-1\})\le2rn^{0.6} +100(r-1)n^{0.6}
+n(D_\mr{big}) \le 102rn^{0.6} +n(D_\mr{big}) ,\]
so for $r\le 100$ we have $n(\{a_r,a_{r+1}\dots,n\})\ge n-n(D_\mr{big})-102rn^{0.6}\ge n^{0.7}-102\cdot 100\cdot n^{0.6} \ge n^{0.6}$ by the property in \cref{claim:degrees_G^1}, and an appropriate choice for $a_{r+1}$ exists.
\end{claimproof}

Now, reveal all the edges in $G_{\mr{rand}}$ except those inside $S$ (i.e., we have now revealed $G_1$, so from now on we treat quantities of the form $n(D)$ as being non-random). Assume that the conclusion of \cref{claim:intervals} holds, giving us a sequence of disjoint intervals $D_1,\dots,D_{33}$.

\medskip
\noindent\textbf{Step 4: Independence between the buckets.} 
We want to be able to study each $V(D_t)$ separately, so we need to make sure that the degrees of vertices in the different $V(D_t)$ cannot coincide. The following claim will be used for this purpose. (Recall that $S$ is our small ``hole'' that each of the $D_t$ lie inside).

\begin{claim}\label{claim:max-degree}
    With probability $1-o(n^{-2})$, the induced subgraph $G_{\mr{rand}}[V(S)]$ has maximum degree at most 100.
\end{claim}
\begin{claimproof}
The probability there is a vertex in $G_{\mr{rand}}[V(S)]$ with at least 100 neighbours is at most
\[n^{0.75} \binom{n^{0.75}}{100}p^{100}=o(n^{-2}),\]recalling that $p=n^{o(1)}/n$.
\end{claimproof}

For each $t\in \{1,\dots,33\}$, let $B(D_t,100)$ be the set of all $b\in \mb N$ which differ from some element of $D_t$ by at most 100. So, the property in \cref{claim:max-degree} implies that for each $t$ and each vertex $w\in V(D_t)$, the degree of $w$ with respect to $G$ lies in $B(D_t,100)$. By construction (of the $D_t$), the sets $B(D_t,100)$ are disjoint from each other.

Now, we also reveal all the edges in $G_{\mr{rand}}$ except those inside the sets $V(D_1),\dots,V(D_{33})$. So, the only remaining randomness is inside these sets.

\medskip
\noindent\textbf{Step 5: Anticoncentration of degree statistics.}
Let $Q_d$ be the number of vertices in $\mathcal L^j(u,v)$ which have degree $d$ with respect to $G$, minus the corresponding number in $\mathcal L^j(v,u)$ (here, the vertices are counted with multiplicity). Given the information revealed so far, our objective is to prove that with probability $1-o(n^{-2})$, either the property in \cref{claim:max-degree} fails to hold, or there is some $d$ such that $Q_d$ is nonzero (or both).

For each $t\in \{1,\dots,33\}$, let
\[Z_t=\sum_{d\in B(D_t,100)}dQ_d. \]
Note that if $Z_t\ne 0$, then there is some $d$ such that $Q_d$ is nonzero. These random variables $Z_t$ have been carefully defined in such a way that they can be expressed in the form required for \cref{thm:LO}.

Indeed, note that the information we have revealed so far already determines the degree (with respect to $G$) of every vertex outside the sets $V(D_1),\dots,V(D_{33})$, and recall that $\mc L^j(u,v)$ and $\mc L^j(v,u)$ are disjoint. For all $t$:
\begin{compactitem}
    \item Let $b_t$ be the contribution to $d Q_d$ from all vertices outside $V(D_1)\cup \dots\cup V(D_{33})$ (i.e., $b_t$ is a weighted sum of degrees, with respect to $G_1$, of vertices which do not lie in $V(D_1)\cup \dots\cup V(D_{33})$).
    \item Let $c_t$ be the sum of the degrees \emph{with respect to $G_1$} of all vertices in $V(D_t)$, multiplied by the number of times that vertex appears in $\mc L^j(u,v)$.
    \item For each pair of vertices $\{x,y\}\subseteq V(D_t)$, let $\xi_{\{x,y\}}$ be the indicator random variable for the event that $xy$ is an edge in $G_{\mr{rand}}$.
    \item For each $\{x,y\}\subseteq V(D_t)$, let $\chi(\{x,y\})=-1$ if $xy$ is an edge of $G_0$, and $\chi(\{x,y\})=1$ if $xy$ is not an edge of $G_0$.
    \item For each $x\in V(D_t)$, let  $a_{x}$ be the number of times that $x$ appears in $\mc L^j(u,v)$.
\end{compactitem}
Define
\[Z_t^*=b_t+c_t+\sum_{\{x,y\}\subseteq V(D_t)} \chi(\{x,y\})(a_{x}+a_{y}) \xi_{\{x,y\}},\]
and note that 
whenever the property in \cref{claim:max-degree} holds, we have $Z_t=Z_t^*$.

Note that $\chi(\{x,y\})(a_{x}+a_{y})$ is nonzero for any distinct $x,y\in V(D_t)$, and note that the different $Z_t^*$ are independent from each other. Recalling \cref{claim:intervals}, we have $\binom{|V(D_t)|}{2}\ge \binom{n^{0.6}}2=\Omega(n^{1.2})$. Also recall that $p=n^{o(1)}/n$. So, by \cref{thm:LO}, we have
\[\Pr[Z_t^*=0\text{ for all }t]\le O\left(\left(\frac{1}{\sqrt{pn^{1.2}}}\right)^{\!\!33}\right)=o(n^{-2}).\]
It follows that with probability $1-o(n^{-2})$ there is some $t$ such that $Z_t^*=Z_t\ne 0$, meaning that there is some $d$ such that $Q_d$ is nonzero, meaning that $u$ and $v$ receive different colours in the stable colouring $\mc R^*\sigma$.
\end{proof}
We have now completed the proof of \cref{prop:expansion-colour-refinement} modulo  \cref{lem:expansion_SLi}, which we prove in the next subsection.

\subsection{Exploring views: proof of \texorpdfstring{\cref{lem:expansion_SLi}}{Lemma~\ref{lem:expansion_SLi}}
}\label{subsec:Exploration}

As mentioned earlier, roughly speaking \cref{lem:expansion_SLi} is proved using the expansion properties of $G_{\mr{rand}}$. As a na\"ive approach to prove \cref{lem:expansion_SLi}, it is not hard to prove that whenever $\mc S^{i}(\{u,v\})$ is reasonably large, then $\mc S^{i+1}(\{u,v\})$ tends to be even larger (by a factor of about $np$), and one can iterate this to show that $\mc S^{j}(\{u,v\})$ eventually reaches the desired size. However, assumption \cref{eq:condition:expansion} in \cref{lem:expansion_SLi} does not guarantee that any \emph{individual} $\mc S^{i}(\{u,v\})$ is large, so we must be more careful. We proceed via a delicate coupling argument comparing sets of the form $\mc S^{\le q}(\{u,v\})$ and $\mc S^{\le q+1}(\{u,v\})$.

 \begin{proof}[Proof of \cref{lem:expansion_SLi}]Recall that we are trying to prove that for all pairs of distinct vertices $u,v$, if \cref{eq:condition:expansion} holds for some $i$ then \cref{eq:condition:expansion:final_set} holds for some $j$.

 \medskip
\noindent\hypertarget{step:one-step}{\textbf{Step 1: Reducing to single-step expansion.}}
 Our main objective will be to prove that whp $G$ satisfies the following property: for every $u,v,q$ for which
 \begin{equation}
 \frac{100\log n}{\min\{(np-1)^2,1\} }\leq |\mc S^{\leq q}(\{u,v\})| \leq \frac{n}{(\log n)^{100}},\label{eq:Sq-start}\end{equation}
 we have
\begin{equation}|\mc S^{\le q+1}(\{u,v\})|\geq 
\left(1+\frac{np-1}{2}\right)|\mc S^{\leq q}(\{u,v\})|,\label{eq:Sq-end} \end{equation}
or equivalently
 \[|\mc S^{q+1}(\{u,v\})|\geq 
\frac{np-1}{2}|\mc S^{\leq q}(\{u,v\})|.\]
To see that this suffices to prove the lemma statement, consider $u,v,i$ satisfying \cref{eq:condition:expansion}, and suppose that the above property holds. Let $j\ge i$ be minimal such that $|\mc S^{\leq j}(\{u,v\})|> n/(\log n)^{100}$. Taking $q=j-1$, we obtain
\begin{align*}
\frac2{np-1}|\mc S^{j}(\{u,v\})|&\ge 
    |\mc S^{\leq j-1}(\{u,v\})|,\\
    \left(1+\frac2{np-1}\right)|\mc S^{j}(\{u,v\})|&\ge 
    |\mc S^{\leq j}(\{u,v\})|>\frac{n}{(\log n)^{100}},
\end{align*}
so recalling that $np-1\ge (\log n)^{-40}$, we have $1/(1+2/(np-1))\ge 1/(4(\log n)^{40})$ and
\[\big|\mc S^{j}_G(\{u,v\})\big|\ge \frac n{4(\log n)^{140}},\]
meaning that $j$ satisfies \cref{eq:condition:expansion:final_set}.

So, from now on we focus on showing that the above property holds whp. In fact, fix distinct vertices $u,v$; we will prove that with probability $1-o(n^{-2})$, each $q$ satisfying \cref{eq:Sq-start} also satisfies \cref{eq:Sq-end}. This suffices; the desired result will then follow from a union bound over $u,v$.

In order to prove this, we need to study how $\mc S^{\leq q+1}(\{u,v\})$ is defined in terms of random edges of $G_{\mr{rand}}$ incident to $\mc S^{\leq q}(\{u,v\})$. The idea, roughly speaking, is that we expect $\mc S^{\leq q+1}(\{u,v\})$ to have size at least about $(1+np)|\mc S^{\leq q}(\{u,v\})|$, so having size at most $(1+(np-1)/2)|\mc S^{\leq q}(\{u,v\})|$ is very unlikely.

\medskip
\noindent\hypertarget{step:exploration-defs}{\textbf{Step 2: Exploration.}}
As in the proof of \cref{prop:expansion-colour-refinement} we consider the following procedure to iteratively reveal edges of $G_{\mr{rand}}$. For each $t$ (starting at $t=0$ and increasing $t$ by one at each step), reveal all edges of $G_{\mr{rand}}$ incident to $\mc S^{t}(\{u,v\})$; this determines the set $\mc S^{t+1}(\{u,v\})$, but all the edges incident to $\mc S^{t+1}(\{u,v\})$ remain unrevealed.

We also define a sub-procedure to break down further the process of revealing edges of $G_{\mr{rand}}$ incident to $\mc S^{t}(\{u,v\})$. To define this sub-procedure we need some more notation.
\begin{compactitem}
    \item Let $G^t_{\mr{rand}}$ be the graph of edges of $G_{\mr{rand}}$ which have been revealed so far (i.e., $G^t_{\mr{rand}}$ is the subgraph of $G_{\mr{rand}}$ consisting of edges incident to $\mc S^{\le t-1}(\{u,v\})$).
    \item Let $G^t=G_0\triangle G^t_{\mr{rand}}$.
    \item Let $W^t=\{1,\dots,n\}\setminus \mc S^{\le t}(\{u,v\})$ be the set of vertices outside $\mc S^{\le t}(\{u,v\})$ (i.e., the set of candidate vertices for $\mc S^{t+1}(\{u,v\})$).
    \item Recall the notation in \cref{def:multisets}, and let $Z^t$ be the set of vertices $w\in W^t$ such that $\ell^{t+1}_{G^t}(w,u,v)\ne \ell^{t+1}_{G^t}(w,v,u)$ (i.e., such that $w$ has a different number of neighbours in $\mc L^t(u,v)$ than $\mc L^t(v,u)$, with respect to $G^t$).
\end{compactitem}
Note that the vertices in $Z^t$ are the vertices which (intuitively speaking) seem to be most likely to end up in $\mc S^{t+1}(\{u,v\})$. Specifically, we have the following observation.
\begin{fact}\label{fact:exactly-zero}
    For every $z\in Z^t$, if $G_{\mr{rand}}$ has no edges between $\mc S^{t}(\{u,v\})$ and $z$, then $z\in \mc S^{t+1}(\{u,v\})$.
\end{fact}
The vertices in $W^t\setminus Z^t$ have an equal number of neighbours in $\mc L^t(u,v)$ and $\mc L^t(v,u)$ with respect to $G^t$, so they will not end up in $\mc S^{t+1}(\{u,v\})$ unless $G_{\mr{rand}}$ has certain edges between $\mc S^{t}(\{u,v\})$ and $z$. The simplest way this can happen is as follows.
\begin{fact}\label{fact:exactly-one}
    For every $w\in W^t\setminus Z^t$, if $G_{\mr{rand}}$ has exactly one edge between $\mc S^{t}(\{u,v\})$ and $w$, then $w\in \mc S^{t+1}(\{u,v\})$.
\end{fact}
With the above two facts in mind, our revelation sub-procedure is defined as follows.
\begin{compactitem}
    \item First, fix an arbitrary ordering of the vertices $z\in Z^t$. For each such $z$:
    \begin{compactitem}
        \item reveal all edges between $z$ and $\mc S^{t}(\{u,v\})$, determining whether $z$ is in $\mc S^{t+1}(\{u,v\})$. We say this is a \emph{type-$Z$ critical moment}, and if $z\in \mc S^{t+1}(\{u,v\})$ we say the critical moment \emph{succeeds}.
    \end{compactitem}
    \item Then, fix an arbitrary ordering of the vertex pairs $s\in \mc S^{t}(\{u,v\})$ and $w\in W^t\setminus Z^t$. For each such $s,w$:
    \begin{compactitem}
        \item if we have not yet revealed whether $sw$ is an edge of $G_{\mr{rand}}$, then we say this is a \emph{type-$W$ critical moment} for $sw$, and reveal the status of $sw$. If $sw$ is an edge of $G_{\mr{rand}}$ we say the critical moment \emph{advances}; in this case, also reveal the rest of the edges between $w$ and $\mc S^{t}(\{u,v\})$, determining whether $w$ is in $\mc S^{t+1}(\{u,v\})$. If $w\in \mc S^{t+1}(\{u,v\})$, then we say the critical moment \emph{succeeds}.
    \end{compactitem}
\end{compactitem}

\medskip
It is convenient to continue the entire revelation process only as long as $|\mc S^{\le t}(\{u,v\})|\le n/(\log n)^{100}$; i.e., once we have finished revealing a set $\mc S^{t+1}(\{u,v\})$, and see that it is larger than $n/(\log n)^{100}$, then we terminate the entire procedure and stop revealing further edges (these further edges will not be relevant for the property we are trying to study). This ensures that, for the rest of the proof, we can assume that $\mc S^{t}(\{u,v\})$ constitutes a relatively small proportion of the vertices.

Going forward, the broad strategy is as follows. Recall from the discussion at the end of \hyperlink{step:one-step}{Step 1} that we are trying to prove that, with probability $1-o(n^{-2})$, \cref{eq:Sq-start} always implies \cref{eq:Sq-end}, by proving that $\mc S^{  q+1}(\{u,v\})$ tends to have size at least about $(1+np)|\mc S^{\leq q}(\{u,v\})|$ for all $q$. We will prove this by considering the above exploration process for $t\le q$.

First, we will see that type-$Z$ critical moments almost always succeed. That is to say, for all $t$, typically almost all of the vertices of $Z^t$ end up in $\mc S^{t+1}(\{u,v\})$. So, if for any step $t\le q$ the set $Z^t$ occupies a sufficiently large fraction of $W^t$, it is easy to show that  $\mc S^{\le q+1}(\{u,v\})$ is large enough to satisfy \cref{eq:Sq-end}. We can therefore assume that most vertices in $W^t$ do not lie in $Z^{t}$, and it essentially suffices to restrict our attention to the vertices in $W^t\setminus Z^t$ throughout the revelation procedure.

Second, we will see that type-$W$ critical moments succeed with probability (at least) about $p$. For a vertex \emph{not} to end up in $\mc S^{\le q+1}$, we must have (roughly speaking, ignoring the effect of the sets $Z^t$) failed at least $|\mc S^{\le q}(\{u,v\})|$ type-$W$ critical moments, meaning that we expect $\mc S^{\leq q+1}(\{u,v\})= \mc S^{\leq q}(\{u,v\}) \cup \mc S^{q+1}(\{u,v\})$ to have size at least about $(1+np)|\mc S^{\leq q}(\{u,v\})|$.

\medskip
\noindent\textbf{Step 3: Coupling with a sequence of coin flips.}
We first make some observations about the success probabilities in critical moments.
\begin{claim}\label{claim:critical-moment-p}
Let $p'=1-400/(\log n)^{99}$.
\begin{compactitem}
    \item Every type-$Z$ critical moment succeeds with probability at least $p'$ (conditional on all information revealed until that moment).
    \item Every type-$W$ critical moment advances with probability $p$. Given that we advance, we then succeed with probability at least $p'$.
\end{compactitem}
\end{claim}
\begin{claimproof}
For the first bullet point: at a type-$Z$ critical moment, we are revealing the status of $|\mc S^{t}(\{u,v\})|<n/(\log n)^{100}$ edges of $G_{\mr{rand}}$ (recalling that we continue the revelation procedure only as long as $|\mc S^{\le t}(\{u,v\})|\le n/(\log n)^{100}$), and by \cref{fact:exactly-zero}, the critical moment succeeds if none of these edges are present. So, the probability that the critical moment succeeds is at least
\[(1-p)^{n/(\log n)^{100}}\geq e^{-2np/(\log n)^{100}}\geq e^{-{200}/{(\log n)^{99}}}\geq \left(1- \frac{400}{(\log n)^{99}} \right)=p'
,\]
    recalling that $p\le 100\log n/n$.

The second bullet point is very similar: at a type-$W$ critical moment, we advance with probability $p$ (as that is the probability any particular edge appears in $G_{\mr{rand}}$), and then we succeed if no edges are present among a certain set of at most $|\mc S^{\le t}(\{u,v\})|-1\le n/(\log n)^{100}$ potential edges; by the same calculation as above this happens with probability at least $p'$.
\end{claimproof}

Now, let $(\alpha_i)_{i\in \mb N}\in \{0,1\}^{\mb N}$ be a sequence of independent $\on{Bernoulli}(p')$ random variables, and let $(\beta_i)_{i\in \mb N}\in \{0,1\}^{\mb N}$ be a sequence of independent $\on{Bernoulli}(p)$ random variables (such that $(\alpha_i)_{i\in \mb N}$ and $(\beta_i)_{i\in \mb N}$ are independent). We can couple these random variables with our revelation procedure in such a way that:
\begin{compactitem}
    \item whenever we encounter a $Z$-critical moment, we check the next unrevealed random variable $\alpha_i$, and if $\alpha_i=1$ then the critical moment succeeds;
    \item whenever we encounter a $W$-critical moment, we check the next unrevealed random variable $\beta_j$, and if $\beta_j=1$ then the critical moment advances. We then check the next unrevealed random variable $\alpha_i$, and if $\alpha_i=1$ then the critical moment succeeds.
\end{compactitem}
Let $R_\alpha^{\le q}$ be the total number of $\alpha_i$ that we reveal in the first $q$ steps of the revelation procedure (i.e., up to the point where we have determined $\mc S^{\le q+1}(\{u,v\})$). Let $N_\alpha^{\le q}$ be the number of times we see $\alpha_i=1$, among these $R_\alpha^{\le q}$
different $\alpha_i$. Similarly, let $R^{\le q}_\beta$ be the total number of $\beta_j$ that we reveal in the first $q$ steps of the revelation procedure, and let $N_\beta^{\le q}$ be the number of these $\beta_j$ that are equal to 1.

\medskip
\noindent\hypertarget{step:interpreting-coins}{\textbf{Step 4: Interpreting the coin flips.}}
Every time we see $\alpha_i=1$ (at step $t$, say), we add a new vertex to $\mc S^{t+1}(\{u,v\})$. So,
\begin{equation}|\mc S^{\le q+1}(\{u,v\})|\ge N_\alpha^{\le q}.\label{eq:success-interpretation}\end{equation}
Every time we see $\beta_j=1$, we reveal a new $\alpha_i$, so
\begin{equation}R_\alpha^{\le q}\ge N_\beta^{\le q}.\label{eq:R-N}\end{equation}
Also, note that at step $t$, exactly one $\alpha_i$ is revealed for each $z\in Z^{t}$, so $|Z^t|\le R_\alpha^{\le q}$ for each $t\le q$.
We can also relate $R_\beta^{\le q}$ to $|\mc S^{\le q}(\{u,v\})|$, as follows.
\begin{claim}\label{claim:reveal-count}
If $|\mc S^{\le q}(\{u,v\})|\le n/(\log n)^{100}$ then
\[R_\beta^{\le q}\ge |\mc S^{\le q}(\{u,v\})|\cdot\left(n-2R_\alpha^{\le q}-\frac n{(\log n)^{100}}\right).\]
\end{claim}
\begin{claimproof}
For every $t\le q$ and every pair of vertices $s\in \mc S^{t}(\{u,v\})$ and $w\in W^t\setminus Z^t$, the only reason there might not be a type-$W$ critical moment for $sw$ at step $t$ (i.e., the only reason $sw$ might not contribute to $R_\beta^{\le q}$) is if $sw$ has already been revealed in a type-$W$ critical moment for some other pair $s'w$ (which advanced, and may or may not have succeeded; in particular it contributed to $R^{\le q}_\alpha$). So,
\begin{align*}
R_\beta^{\le q}&\ge \sum_{t\le q}|\mc S^{t}(\{u,v\})|\cdot|W^t\setminus Z^t| - \max_{t\le q} |\mc S^{t}(\{u,v\})|\cdot R_\alpha^{\le q}\\
&\ge \sum_{t\le q}|\mc S^{t}(\{u,v\})|\cdot\left(n-|\mc S^{\le t}(\{u,v\})|-|Z^t|\right)-|\mc S^{\le q}(\{u,v\})|\cdot R_\alpha^{\le q}\\
&\ge |\mc S^{\le q}(\{u,v\})|\cdot\left(n-|\mc S^{\le q}(\{u,v\})|-R_\alpha^{\le q}\right) -|\mc S^{\le q}(\{u,v\})|\cdot R_\alpha^{\le q}.
\end{align*}
The desired result follows, recalling our assumption $|\mc S^{\le q}(\{u,v\})|\le n/(\log n)^{100}$.
\end{claimproof}

\noindent\hypertarget{step:chernoff}{\textbf{Step 5: Concluding with Chernoff bounds.}}
Let $\mc E_\alpha$ be the event that for every $r\ge n/(\log n)^{100}$ there are at most $800r/(\log n)^{99}$ different ``$0$''s among the first $r$ of the $\alpha_i$. If $\mc E_{\alpha}$ holds, then for all $q$ satisfying $R_\alpha^{\le q}\ge n/(\log n)^{100}$, we have
\begin{equation}
    N_\alpha^{\le q}\ge (1-800/(\log n)^{99})R_\alpha^{\le q}.\label{eq:chernoff-alpha}
\end{equation}Recall that $p'=1-400/(\log n)^{99}$, so by a Chernoff bound (\cref{lem:chernoff}) and a union bound over $r$, we see that
\[1-\Pr[\mc E_\alpha]\le \sum_{r=n/(\log n)^{100}}^\infty \exp\left(-\Omega\left(\frac r{((\log n)^{99})^{2}}\right)\right)= o(n^{-2}),\]
i.e., $\mc E_\alpha$ holds with probability $1-o(n^{-2})$.

Similarly, let $\mc E_\beta$ be the event that for every $s\ge 100 \log n/\min\{(np-1)^2,1\}$, there are at least $(1+2(np-1)/3)s$ different ``$1$''s among the first $sn$ of the $\beta_i$. 
By a Chernoff bound and a union bound over $s$, we see that
\[1-\Pr[\mc E_\beta]\le \sum_{s=100 \log n/\min\{(np-1)^2,1\}}^\infty \exp\left(-\left(\frac{np-1}{3}\right)^2 spn/2\right)=o(n^{-2}).\]
i.e., $\mc E_\beta$ holds with probability $1-o(n^{-2})$.

Now, recall that we are trying to prove that, with probability $1-o(n^{-2})$, whenever $q$ satisfies \cref{eq:Sq-start} it also satisfies \cref{eq:Sq-end}. It suffices to show that this property follows from $\mc E_\alpha \cap \mc E_\beta$.
So, for the purpose of contradiction, suppose that $\mc E_\alpha \cap \mc E_\beta $ 
holds, and suppose that there is some $q$ for which
 \[\frac{100\log n}{\min\{(np-1)^2,1\} }\leq |\mc S^{\leq q}(\{u,v\})| \leq \frac{n}{(\log n)^{100}},\]
but
\begin{equation}|\mc S^{\le q+1}(\{u,v\}) |<
\left(1+\frac{np-1}{2}\right)|\mc S^{\leq q}(\{u,v\})| \leq \frac{50n}{(\log n)^{99}},\label{eq:chernoff-conclusion}\end{equation}
recalling that $p\le 100\log n/n$. First, note that \cref{eq:chernoff-conclusion,eq:success-interpretation,eq:chernoff-alpha} imply that
\[R_\alpha^{\le q}\le \max\left\{\frac n{(\log n)^{100}},\,2N_\alpha^{\le q}\right\}\le \max\left\{\frac n{(\log n)^{100}},\,2|\mc S^{\le q+1}(\{u,v\})|\right\}\le 100n/(\log n)^{99}.\] This and \cref{claim:reveal-count} imply that
\[
R^{\le q}_\beta\ge |\mc S^{\leq q}(\{u,v\})|\cdot \left(n-\frac{200n}{(\log n)^{99}}-\frac{n}{(\log n)^{100}}\right)\ge |\mc S^{\leq q}(\{u,v\})|\cdot \left(n-\frac{201n}{(\log n)^{99}}\right),
\]
while \cref{eq:R-N,eq:success-interpretation,eq:chernoff-alpha,eq:chernoff-conclusion} imply that
\[N^{\le q}_\beta\le R_\alpha^{\le q}\le \frac{N_\alpha^{\le q}}{1-800/(\log n)^{99}}\le \frac{|\mc S^{\le q+1}(\{u,v\})|}{1-800/(\log n)^{99}}\le \frac{(1+(np-1)/2)|\mc S^{\leq q}(\{u,v\})|}{1-800/(\log n)^{99}}.\]
Recalling that $np-1\ge (\log n)^{-40}$, these inequalities contradict $\mc E_\beta$ (taking $s=|\mc S^{\le q}(\{u,v\})|$ in the definition of $\mc E_\beta$).
\end{proof}

\section{Distinguishing vertices in the 2-core}\label{sec:3core}

In the previous section we showed that given two vertices $u,v$ in an appropriately randomly perturbed graph $G$, it is very likely that whenever the sets $\mc S^{i}(\{u,v\})$ grow reasonably large, $u$ and $v$ end up with different colours in the stable colouring $\mc R^*\sigma$. In this section we prove that this condition is typically satisfied for pairs of vertices $u,v$ satisfying certain combinatorial conditions (defined in terms of the \emph{2-core} of the random perturbation graph $G_{\mr{rand}}$). In particular, the main result of this section will directly imply \cref{thm:smoothed-colour-refinement} (as we will see at the end of this section). 
\begin{definition}\label{def:2core-extra}
Recall from \cref{def:2core} that the $k$-core of a graph $G$ is its maximum subgraph with minimum degree at least $k$.
\begin{itemize}
\item  Write $V_k(G)$ for the vertex set of this $k$-core (so the $k$-core itself is $\mr{core}_k(G)=G[V_k(G)]$).
    \item As in \cref{def:V23}, let $V_{2,3}(G)\subseteq V_3(G)$ be the set of vertices which have degree at least $3$ in the 2-core $G[V_2(G)]$.
    \item Let $V_{2,3}^{\mr{safe}}(G)\subseteq V_{2,3}(G)$ be the set of vertices $v\in V_{2,3}(G)$ satisfying the following property. If we delete any two vertices (other than $v$), the connected component of $v$ still contains at least three vertices of $V_{2,3}(G)$ (including $v$).
\end{itemize}
\end{definition}
In words, note that the 2-core can be viewed as a network of ``hubs'' (vertices with degree at least 3), joined by bare paths in which every interior vertex has degree 2. Each hub has at least three bare paths emanating from it, but it could happen that several of these bare paths lead to the same hub. We say a hub $v$ is ``safe'' if it is not possible to delete two vertices to create a tiny connected component containing $v$ and at most one other hub.

\begin{proposition}\label{prop:2core}
Let $G_0$ be a graph on the vertex set $\{1,\dots,n\}$ and let $G_{\mr{rand}}\sim G(n,p)$ for some $p\in [0,1/2]$ satisfying 
 \[p\ge \frac{1+(\log n)^{-40}}{n}.\]
Let $G=G_0\triangle G_{\mr{rand}}$. Then whp $G$ has the property that $\mc R^*_{G} \sigma(u) \neq \mc R^*_{G} \sigma (v)$ for every pair of distinct vertices $u,v$ satisfying one of the following assumptions.
\begin{enumerate}[{\bfseries{A\arabic{enumi}}}]
    \item\label{item:2core-for-perturbed}$u\in V_{2,3}^{\mr{safe}}(G_{\mr{rand}})$, or
    \item\label{item:2core-for-random}$u,v\in V_{2,3}(G_{\mr{rand}})$, and $G_0$ is the empty graph.
\end{enumerate}
\end{proposition}

The idea of the proof of \cref{prop:2core} is as follows. We fix two distinct vertices $u,v$, and consider the same exploration procedure that featured in the proofs of \cref{prop:expansion-colour-refinement,lem:expansion_SLi}: at step $t$ we reveal the edges of $G_{\mr{rand}}$ incident to $\mc S^{t}(\{u,v\})$ to determine the next set $\mc S^{t+1}(\{u,v\})$. 
Vertices of $\mc S^{t}(\{u,v\})$ which lie in $V_{2,3}(G_{\mr{rand}})$ are typically connected, via bare paths in $V_2(G)$, to at least three vertices of $V_{2,3}(G_{\mr{rand}})$. Our exploration process will follow these paths, and typically each of these ``neighbouring'' vertices in $V_{2,3}(G_{\mr{rand}})$ will end up in some future $\mc S^{t'}(\{u,v\})$. That is to say, each time our exploration process reaches a hub, it branches out into at least two more hubs; the resulting exponential growth means it will not take long for $\mc S^{\le t}(\{u,v\})$ to grow large enough to apply \cref{prop:expansion-colour-refinement}.
Indeed, if the exploration process ``gets stuck'' (with $\mc S^{i+1}(\{u,v\})=\emptyset$) before \cref{eq:condition:expansion} is satisfied, some very unlikely events need to have happened: we need to have been in a subset of the 2-core with very poor expansion properties, or some edges must have shown up in atypical places during the iterative revelation process.

\begin{proof}[Proof of \cref{prop:2core}]
Before we do anything else, the first step is to reduce to the case where $p$ is in a suitable range to apply \cref{prop:expansion-colour-refinement}.

\medskip
\noindent\textbf{Step 1: Making assumptions on $p$.} 
We can assume $p\le 2\log n/n$, because for larger $p$ we can simply view the random perturbation as a composition of two random perturbations. Specifically, we can choose $p_1\in [0,1]$ in such a way that, if we take $G_{\mr{rand}}'\sim \mb G(n,2\log n/n)$ and $G_1\sim \mb G(n,p_1)$, then $G_{\mr{rand}}$ has the same distribution as $G_1\triangle G_{\mr{rand}}'$. (Namely, we need $p_1$ to satisfy the equation $p_1(1-2\log n/n)+(2\log n/n)(1-p_1)=p$). Then, letting $G_0'=G_0\triangle G_1$ we have $G=G_0'\triangle G_{\mr{rand}}'$, so after conditioning on an outcome of $G_0'$ we are considering a randomly perturbed graph with edge perturbation probability exactly $2\log n/n$. Moreover, whp the sets  $V_{2,3}^{\mr{safe}}(G_{\mr{rand}}')$ and $V_{2,3}^{\mr{safe}}(G_{\mr{rand}})$ are identical: in this regime whp both, $G_{\mr{rand}}'$ and $G_{\mr{rand}}$ are $3$-connected (see for example \cite[Theorem 4.3]{FK16}) hence $V_{2,3}^{\mr{safe}}(G_{\mr{rand}}')=V_3(G_{\mr{rand}}')=V(G)=V_3(G_{\mr{rand}})=V_{2,3}^{\mr{safe}}(G_{\mr{rand}})$.

\medskip
\noindent\textbf{Step 2: Expansion.} We observe a (very weak) expansion property of $G_{\mr{rand}}$: whp there are no small vertex cuts that disconnect a small dense subgraph.

\begin{claim}\label{claim:no-cut}
Whp $G_{\mr{rand}}$ has the following property. Consider disjoint vertex sets $T,U$, where $|T|\le 2$ and $|U|\le (\log n)^{100}$. Suppose that $G_{\mr{rand}}$ has no edges between $U$ and $V(G)\setminus T$ (i.e., $T$ disconnects $U$ from the rest of the graph), and suppose that $G_{\mr{rand}}[T\cup U]$ is connected. Then the number of edges in $G_{\mr{rand}}[U\cup T]$ is at most $|U|+|T|$.
\end{claim}

\begin{claimproof}
If the desired conclusion were to fail, there would be a pair of vertex sets $U,T$, with $|T|\le 2$ and $|U|\le (\log n)^{100}$, such that $G_{\mr{rand}}$ has no edges between $U$ and $V(G)\setminus T$, and such that $G_{\mr{rand}}[U\cup T]$ has at least $|U|+|T|+1$ edges.

Summing over possibilities for $t=|U|+|T|$, the probability that such a pair of vertex sets exists is at most
\begin{align*}
\sum_{t = 4 }^{(\log n)^{100} +2}
\binom{n}{t}t^2 t^{t-2} (t^2)^2 p^{t+1}(1-p)^{(t-2)(n-t)}
&\leq 2\sum_{t = 4 }^{n^{o(1)}} \bigg(\frac{en}t\bigg)^t t^{t+4}p^{t+1}e^{-pnt}e^{2np}
\\& \leq 2\sum_{t = 4 }^{n^{o(1)}} (enpe^{-pn})^t e^{2np} pt^4=o(1).
\end{align*}
In the last equality we used that if $np\leq 1000$ then $e^{2np}$, whereas if $np\geq 1000$ then $(enpe^{-pn})^t e^{2np} \geq e^{-tnp/10}$, for $t\geq 4$.

To provide a bit of explanation for the above calculation: first note that $U\cup T$ always spans at most $|U|+|T|$ edges if $|U|+|T|\le 3$, so we only need to consider $t\ge 4$. The number of possible ways to choose $U$ and $T$ is at most $\binom n t t^2$ (first choose $U\cup T$, then choose $T$). Then, note that $G_{\mr{rand}}[U\cup T]$, being connected, must contain a spanning tree (which has $t-1$ edges, and can be chosen in $t^{t-2}$ different ways, by Cayley's formula), plus two additional edges (which can be chosen in at most $(t^2)^2$ different ways). These $(t-1)+2=t+1$ edges are present in $G_{\mr{rand}}$ with probability $p^{t+1}$. Also, since $T$ disconnects $U$ from the rest of the graph, there are $(t-2)(n-t)$ specific pairs of vertices which are \emph{not} edges of $G_{\mr{rand}}$, namely the edges from $U$ to $G-(U\cup T)$) (which happens with probability $(1-p)^{(t-2)(n-t)}$).

In the last line, we used the inequality $xe^{-x}\le 1/e$ (which holds for all $x\in \mb R$), and the assumption $p=n^{o(1)-1}$.
\end{claimproof}

\noindent\textbf{Step 3: Exploration.} Define
\[f(n)=\frac{100\log n}{\min\{(np-1)^2,1\}} \leq 100\log^{81} n\le (\log n)^{100}.\]
Fix distinct vertices $u,v$. We will prove that with probability at least $1-o(n^{-2})$, either the property in \cref{claim:no-cut} fails to hold or there is $i$ such that $|\mc S^{\le i}(\{u,v\})|\ge f(n)$. The desired result will then directly follow from a union bound over $u,v$, and \cref{prop:expansion-colour-refinement}.

As in the proofs of \cref{prop:expansion-colour-refinement,lem:expansion_SLi}, we consider an exploration/revelation procedure. In each step $t$, if we do not already have $|\mc S^{\le t+1}(\{u,v\})|\ge f(n)$ (this is condition \cref{eq:condition:expansion} in \cref{prop:expansion-colour-refinement}), then we reveal the edges of $G_{\mr{rand}}$ incident to $\mc S^{t}(\{u,v\})$ to determine the next set $\mc S^{t+1}(\{u,v\})$. This process terminates at the end of the first step $i$ for which $|\mc S^{\le i+1}(\{u,v\})|\ge f(n)$ or $\mc S^{ i+1}(\{u,v\})=\emptyset$.

Our goal is to rule out (with probability $1-o(n^{-2})$) the possibility that \cref{item:2core-for-perturbed} or \cref{item:2core-for-random} holds, $\mc S^{i+1}(\{u,v\})=\emptyset$ and $|\mc S^{\le i}(\{u,v\})|<f(n)$.

\medskip
\noindent\noindent\hypertarget{step:atypical}{\textbf{Step 4: Atypical events during exploration.}} Recall the definitions of the sets $Z^t\subseteq W^t$ from \hyperlink{step:exploration-defs}{Step 2} of the proof of \cref{lem:expansion_SLi} (in \cref{subsec:Exploration}). Namely, $W^t$ is the set of ``candidate vertices'' for $\mc S^{t+1}(\{u,v\})$, and $Z^t$ is a subset of vertices that are especially likely to end up in $\mc S^{t+1}(\{u,v\})$.

Recalling \cref{fact:exactly-one}, if there is exactly one edge between $\mc S^{t}(\{u,v\})$ and $w\in W^t\setminus Z^t$ (with respect to $G_{\mr{rand}}$), then we have $w\in \mc S^{t+1}(\{u,v\})$. That is to say, an edge of $G_{\mr{rand}}$ between $\mc S^{t}(\{u,v\})$ and $w\in W^t$ always contributes to the growth of $\mc S^{t+1}(\{u,v\})$, unless $w\in Z^t$, or unless there is more than one edge between $w$ and $\mc S^{t}(\{u,v\})$. The following claim shows that these two possibilities happen very rarely (unless $|\mc S^{\le i}(\{u,v\})|$ is large, in which case we are done).

\begin{claim}\label{claim:2core-atypical}
For $t\le i$, let $Z^t_{\mr{bad}}\subseteq Z^t$ be the set of vertices in $Z^t$ which are adjacent to $\mc S^{t}(\{u,v\})$ with respect to $G_{\mr{rand}}$, and  let $W^t_{\mr{bad}}$ be the set of vertices in $W^t\setminus Z^t$ with more than one neighbour in $\mc S^{t}(\{u,v\})$, again with respect to $G_{\mr{rand}}$. Then, with probability $1-o(n^{-2})$, we have 
\[|\mc S^{\le i+1}(\{u,v\})|\ge f(n)\quad\text{or}\quad\sum_{t=0}^{i}\big(|W^t_{\mr{bad}}|+|Z^t_{\mr{bad}}|\big)\le 2.\]
\end{claim}
\begin{claimproof}

First, note that with probability $1-o(n^{-2})$, we have $|Z^t|< 2f(n)$ for all $t\le i$ such that $|\mc S^{\le t+1}(\{u,v\})|< f(n)$. Indeed, consider the \emph{first} step $t$ such that $|Z^t|\ge 2f(n)$ (if such a step exists), and suppose that $|\mc S^{\le t}(\{u,v\})|< f(n)$. Conditioning on the information revealed up to step $t$, by the same argument as given in \cref{claim:critical-moment-p}, each vertex in $Z^{t}$ independently ends up in $\mc S^{t+1}(\{u,v\})$ with probability $1-o(1)$, so by a Chernoff bound (\cref{lem:chernoff}), we have $|\mc S^{t+1}(\{u,v\})|\ge |Z^t|/2\ge f(n)$ with probability at least $1-o(n^{-2})$, in which case the desired property holds.

Now, let $X_W=\sum_{t}|W^t_{\mr{bad}}|$, where the sum is over all $t\le f(n)$ for which $|\mc S^{t}(\{u,v\})|<f(n)$, and let $X_Z=\sum_{t}|Z^t_{\mr{bad}}|$, where the sum is over all $t\le f(n)$ for which $|Z^t|<2f(n)$ and $|\mc S^{t}(\{u,v\})|<100\log n$.  It suffices to show that $X_W+X_Z\le 2$ with probability $1-o(n^{-2})$.

Note that, if we condition on the information revealed up to step $t$ (i.e., after we have revealed $\mc S^{ t}(\{u,v\})$, but before we have revealed $\mc S^{t+1}(\{u,v\})$), then each vertex $w\in W^t\setminus Z^t$ ends up in $W^t_{\mr{bad}}$ with probability at most $p^2 |\mc S^{ t}(\{u,v\})|^2=n^{-2+o(1)}$, and each vertex $z\in Z^t$ ends up in $Z^t_{\mr{bad}}$ with probability at most $p|\mc S^{ t}(\{u,v\})|=n^{-1+o(1)}$. So, $X_W+X_Z$ is stochastically dominated by a sum of at most $f(n)\cdot n =n^{1+o(1)}$ Bernoulli trials with probability $n^{-2+o(1)}$, plus a sum of at most $f(n)\cdot 2f(n)=n^{o(1)}$ Bernoulli trials with probability $n^{-1+o(1)}$ (all independent). We conclude that
\begin{align*}
\Pr[X_W+X_Z\ge 3]&\le \sum_{q=0}^3 \Pr[X_W\ge q,\,X_Z\ge 3-q]
\\&\le \sum_{q=0}^3 \binom{n^{1+o(1)}}{q}(n^{-2+o(1)})^q\cdot \binom{n^{o(1)}}{3-q}(n^{-1+o(1)})^{3-q}
=n^{-3+o(1)} = o(n^{-2}), \end{align*}
as desired.
\end{claimproof}

\noindent\textbf{Step 5: Concluding.} 
We now wish to prove that if \cref{item:2core-for-perturbed} or \cref{item:2core-for-random} are satisfied (i.e., $u\in V_{2,3}^{\mr{safe}}(G_{\mr{rand}})$, or $G_0=\emptyset$ and $u,v\in V_{2,3}(G_{\mr{rand}})$), then the properties in \cref{claim:2core-atypical,claim:no-cut} imply that $|\mc S^{\le i+1}(\{u,v\})|\ge f(n)$. So, suppose for the purpose of contradiction that \cref{item:2core-for-perturbed} or \cref{item:2core-for-random} is satisfied, and the properties in \cref{claim:2core-atypical,claim:no-cut} hold, and for some $i$, $|\mc S^{\le i}(\{u,v\})|< f(n)$, and $\mc S^{i+1}(\{u,v\})=\emptyset$.

Let $T=\bigcup_{t=0}^i (W^t_{\mr{bad}}\cup Z^t_{\mr{bad}})$, so $|T|\le 2$ by the property in \cref{claim:2core-atypical}.  As discussed in \hyperlink{step:atypical}{Step 3}, every edge of $G_{\mr{rand}}$ between $\mc S^{ t}(\{u,v\})$ and $w\in W^t$ contributes to $\mc S^{\le t+1}(\{u,v\})$, unless $w\in W^t_{\mr{bad}}\cup Z^t_{\mr{bad}}$. Since $\mc S^{i+1}(\{u,v\})=\emptyset$, this means $T$ disconnects $\mc S^{\le i}(\{u,v\})$ from the rest of $G_{\mr{rand}}$. Let $U_u$ and $U_v$ be the vertex sets of the connected components of $u$ and $v$, with respect to $G_{\mr{rand}}-T$. Note that $U_u\cup U_v$, being a subset of $\mc S^{\le i}(\{u,v\})$, has fewer than $f(n)\le (\log n)^{100}$ vertices. 

First, suppose \cref{item:2core-for-perturbed} is satisfied (i.e., $u\in V_{2,3}^{\mr{safe}}(G_{\mr{rand}})$). 
By the definition of $V_{2,3}^{\mr{safe}}(G_{\mr{rand}})$,  the connected component of $u$ in $G_{\mr{rand}}-T$ always contains at least three vertices in $ V_{2,3}(G_{\mr{rand}})$. So, $G_{\mr{rand}}\big[(T\cup U_u)\cap V_2(G_{\mr{rand}})\big]$ has at least three vertices of degree at least 3, and at most two vertices of degree 1 (these can only be vertices in $T$). It follows that $G_{\mr{rand}}\big[(T\cup U_u)\cap V_2(G_{\mr{rand}})\big]$ has more edges than vertices, and the same holds for $G_{\mr{rand}}\big[T\cup U_u\big]$ (recall that a graph is obtained from its 2-core by adding ``dangling trees'', which do not change the balance of vertices and edges), contradicting \cref{claim:no-cut}.

Second, suppose \cref{item:2core-for-random} is satisfied (i.e., $G_0=\emptyset$ and $u,v\in V_{2,3}(G_{\mr{rand}})$). Since $G_0$ is the empty graph, we have $Z^t_{\mr{bad}}=Z^t=\emptyset$ for each $t$. So, every vertex in $T$ lies in some $W^t_{\mr{bad}}$, meaning it has at least two neighbours in $\mc S^{\le i}(\{u,v\})$ (which in turn are connected to $u$ or $v$, via some path). So, all vertices in $G_{\mr{rand}}\big[(T\cup U_u\cup U_v)\cap V_2(G_{\mr{rand}})\big]$ have degree at least 2; since $u$ and $v$ have degree at least 3, it follows that $G_{\mr{rand}}\big[(T\cup U_u\cup U_v)\cap V_2(G_{\mr{rand}})\big]$ has more edges than vertices, and the same holds for $G_{\mr{rand}}\big[T\cup U_u\cup U_v\big]$. This graph may not be connected, but at least one of its connected components contradicts \cref{claim:no-cut}.
\end{proof}

Before ending this section, we record a simple consequence of \cref{prop:2core}: for any $k\ge 3$, whp the vertices in the 3-core of $G_{\mr{rand}}$ are all assigned unique colours by the colour refinement algorithm. Recall that $V_k(G)$ is the vertex set of the $k$-core of a graph $G$, and note that $V_k(G)\subseteq V_{k'}(G)$ when $k'\le k$.

\begin{proposition}\label{prop:3core}
    Let $G_0$ be a graph on the vertex set $\{1,\dots,n\}$,  let $G_{\mr{rand}}\sim G(n,p)$ for some $p\in [0,1/2]$, and let $G=G_0\triangle G_{\mr{rand}}$. Then whp $G$ has the property that $\mc R^*_{G} \sigma(u) \neq \mc R^*_{G} \sigma (v)$ for every pair of distinct vertices $u,v$ such that $u\in V_3(G_{\mr{rand}})$.
\end{proposition}
\begin{proof}
It was famously proved by Pittel, Spencer and Wormald~\cite{PSW96} that there is a constant $c\approx 3.35$ such that if $p\le c/n$  then $V_3(G_{\mr{rand}})=\emptyset$ whp. So, we can assume that $p\ge 3/n$.

Also, it is easy to prove (see for example \cite[Theorem 3.2]{FK16}) that if $p\ge (1+\varepsilon)\log n/n$ for any fixed $\varepsilon>0$, then whp every vertex of $G_{\mr{rand}}$ has degree at least 4; that is, $V_{2,3}(G)=V_4(G)=V(G)$. Note that if a vertex $v\in V_{2,3}(G)$ has at least four neighbours in $V_{2,3}(G)$, then trivially $v\in V_{2,3}^{\mr{safe}}$, so in this case the desired result follows directly from \cref{prop:2core}. We can therefore assume that $p\le 2\log n/n$.

Now, it suffices to prove that whp $V_3(G_{\mr{rand}})\subseteq V_{2,3}^{\mr{safe}}(G_{\mr{rand}})$. Consider $v\in V_3(G_{\mr{rand}})\subseteq V_{2,3}(G_{\mr{rand}})$. Note that $v$ has at least three neighbours with degree at least 3 in $V_{2,3}(G_{\mr{rand}})$. The only way that deletion of two vertices other than $v$ could possibly cause the connected component of $v$ to contain fewer than three vertices of $V_{2,3}(G_{\mr{rand}})$ is if the two deleted vertices $x,y$ are both neighbours of $v$, and there is a third neighbour of $v$ whose neighbours in $V_{2,3}(G_{\mr{rand}})$ are precisely $x,y,v$. If this situation occurs, then $G$ has a subgraph with four vertices and at least five edges. It is easy to see that whp no such subgraph exists in $G_{\mr{rand}}$, given our assumption that $p\le 2\log n/n$ (indeed, the expected number of such subgraphs is $O(p^5 n^4)=o(1)$).
\end{proof}

As observed in the proof of \cref{prop:3core}, for any fixed $\varepsilon>0$, if $G_{\mr{rand}}\sim \mb G(n,p)$ with $p\ge (1+\varepsilon)\log n/n$, then whp every vertex of $G_{\mr{rand}}$ has degree at least 3, so $V_{3}(G_{\mr{rand}})=V(G)$. So, \cref{prop:3core} implies that $\mc R^*\sigma$ gives every vertex a different colour. That is to say, \cref{thm:smoothed-colour-refinement} follows directly from \cref{prop:3core} and \cref{thm:CR-discrete-canonical}.

\section{Small components with very mild perturbation}\label{sec:disparity}
In this section we prove \cref{thm:smoothed-disparity}. The proof is fairly involved, so we start with a rough description of the overall structure of the proof, supplementing the discussion in \cref{subsec:ideas}. In fact, the reader may benefit from re-reading \cref{subsec:breakup-outline} before starting with this section. We also remark that a much simpler implementation of similar ideas can be found in \cref{app:loglog}, where we prove a weaker version of \cref{thm:smoothed-disparity}.   Recall the definitions of the colour refinement and 2-dimensional Weisfeiler--Leman algorithms, from \cref{def:refinement,def:2WL}.

Recalling the notation in \cref{def:disparity,def:2WL} (on disparity graphs and the 2-dimensional Weisfeiler--Leman algorithm), our goal is to prove that if $G=G_0\triangle G_{\mr{rand}}$ is a randomly perturbed graph with perturbation probability as small as $O(1/n)$, then the connected components of the disparity graph $D(G,\mc V^*\phi_G)$ have size $O(\log n)$. Indeed, we saw in \cref{cor:small-components,fact:2WL-efficient} that there is an efficient canonical scheme for graphs satisfying this property.

Towards this goal, we view the random perturbation $G_{\mr{rand}}$ as the union of eight slightly sparser independent random perturbations $G_{\mr{rand}}^{1},\dots,G_{\mr{rand}}^{8}\sim \mb G(n,p')$ (this is called \emph{sprinkling}). Assuming $p\ge 100/n$ (thus $p'>10/n$), standard results show that the 3-core of $G_{\mr{rand}}^{1}$ (which is a subgraph of the 3-core of $G_{\mr{rand}}$) has at least $n/2$ vertices, and by \cref{prop:2core}, whp all these vertices are assigned unique colours by $\mc R_G^*\sigma$ (therefore by $\mc V^*\phi_G$, recalling \cref{fact:2WL-refines-CR}). 

So, we first reveal the 3-core of $G_{\mr{rand}}^1$, and fix a set $V_{\mr{core}}$ of $n/2$ vertices in this 3-core. We then define the complementary set $U=V(G)\setminus V_{\mr{core}}$, and reveal all edges of $G_{\mr{rand}}$ except those between $U$ and $V_{\mr{core}}$. For the rest of the proof we work with the remaining seven rounds of random edges (in $G_\mr{rand}^2,\dots,G_{\mr{rand}}^8$) between $U$ and $V_{\mr{core}}$.

Specifically, whenever a vertex in $U$ has at least three neighbours in $V_{\mr{core}}$, that vertex is guaranteed to be in the 3-core of $G_{\mr{rand}}$, and therefore we can assume it will be assigned a unique colour. So, we can consider a vertex-colouring of $G[U]$ that is refined with each round of sprinkling: at each round, new random vertices in $U$ are assigned unique colours, and then the 2-dimensional Weisfeiler--Leman algorithm propagates this information to obtain a new stable vertex-colouring. In each round, we show that the colour classes, and the components of the disparity graph, are broken into smaller and smaller pieces.

\subsection{Preparatory lemmas}\label{subsec:preparation}
Before we get into the details of the proof of \cref{thm:smoothed-disparity}, we start with a few general graph-theoretic lemmas.

First, we need two near-trivial facts about disparity graphs with respect to equitable colourings: in a disparity graph, the density between any pair of colour classes is at most $1/2$, and refinement cannot increase degrees in the disparity graph. Recall the definition of equitability from \cref{def:equitable}, and recall the definition of the generalised disparity graph $D_L(G,c)$ from \cref{def:DL} (if $c_0$ is a coarsening of $c$, then $D(G,c_0)$ can be interpreted as a generalised disparity graph of the form $D_L(G,c)$).
\begin{fact}\label{fact:D-equitable}
    Consider any graph $G$ and any equitable colouring $c:V(G)\to \Omega$, and let $A,B$ be two distinct colour classes of $c$.
    \begin{compactitem}
\item $D(G,c)[A]$ is regular, with degree at most $(|A|-1)/2$.
    \item $D(G,c)[A,B]$ is biregular. In this bipartite subgraph, the common degree of vertices in $A$ is at most $|B|/2$, and the common degree of vertices in $B$ is at most $|A|/2$.
    \end{compactitem}
\end{fact}

\begin{fact}\label{fact:degrees-decrease}
Let $c:V(G)\to \Omega$ be an equitable vertex-colouring of a graph $G$. Then for every vertex $v$, the degree of $v$ in $D(G,c)$ is at most the degree of $v$ in any generalised disparity graph $D_L(G,c)$ (and therefore at most the degree of $v$ in $D(G,c_0)$, for any possibly non-equitable coarsening $c_0$ of $c$).
\end{fact}

Next, the following lemma shows that refining a colouring cannot significantly increase the sizes of connected components of the disparity graph.
\begin{lemma}\label{lem:disparity-components}
Let $c:V(H)\to \Omega$ be an equitable vertex-colouring of a graph $H$, and consider a matrix $L\in \{0,1\}^{\Omega\times \Omega}$. Let $H_1=D_L(H,c)$ and $H_2=D(H,c)$, and for a vertex $v$, let $X_1(v),X_2(v)$ be the (vertex sets of) the connected components containing $v$ in the graphs $H_1$ and $H_2$, respectively. Then
\[|X_2(v)|\le 2|X_1(v)|.\]
(The same inequality therefore holds when $H_1=D(H,c_0)$ for some coarsening $c_0$ of $c$).
\end{lemma}

The proof of \cref{lem:disparity-components} uses 
the equitability of $c$ together with the following simple structural fact (see \cite[Exercise 1.1.12]{bondymurty}).
The ``special'' graphs in the second bullet point of \cref{fact:graph-structure} may already give some hint for why the factor of two in \cref{lem:disparity-components} is required.

\begin{fact}\label{fact:graph-structure}
Let $H$ be a graph.
\begin{compactitem}
    \item If $H$ has at least $\binom{|V(H)|}2/2$ edges, then it is connected.
    \item If $H$ is a biregular bipartite graph with bipartition $A,B$ and at least $|A||B|/2$ edges, then it is connected, unless $H$ is a disjoint union of two isomorphic complete bipartite graphs (in which case $H$ has exactly $|A||B|/2$ edges). In this latter case we say $H$ is ``special''.
\end{compactitem}
\end{fact}

\begin{proof}[Proof of \cref{lem:disparity-components}]
Fix a vertex $v$, and let $C_1, C_2,...,C_k\subseteq V(G)$ be the colour classes which contain at least one vertex of the component $X_2(v)$. Let $X_1^*(v)$ be the (vertex set of the) component of $v$ in $H_1[C_1\cup\dots\cup C_k]$, so  $X_1^*(v)\subseteq X_1(v)$ and it suffices to show  that
\begin{equation}
 |X_1^*(v)|\geq \frac{|X_2(v)|}2.\label{eq:goal-coupling-disparity}\end{equation}
First, we show that the paths in $H_2$ between $C_1,\dots,C_k$ imply the existence of corresponding paths in $H_1$.

\begin{claim}\label{fact:connected}
    Consider any $i\le k$ and any vertex $u\in C_j$. In the graph $H_1[C_1\cup \dots\cup C_k]$ there is a path between $u$ and $C_i$.
\end{claim}
\begin{claimproof}
By the choice of $C_1,\dots,C_k$, note that for each $i,j$ there is a path in $H_2[C_1\cup \dots\cup C_k]$ between some vertex of $C_i$ and some vertex of $C_j$. But if there is any edge in $H_2=D(H,c)$ between two colour classes $C_i$ and $C_j$, then $D(H,c)[C_i,C_j]$ is neither complete nor empty\footnote{Here we are abusing notation slightly and writing $F[A,A]$ to mean $F[A]$. ``Complete'' should be understood in context to mean either a complete graph or a complete bipartite graph.} (recalling \cref{fact:D-equitable}), and the same is true for $H_1[C_i,C_j]=D_L(H,c)[C_i,C_j]$. So, for each $i,j$ there is a path in $H_1[C_1\cup \dots\cup C_k]$ between some vertex of $C_i$ and some vertex of $C_j$, and the desired result follows from \cref{fact:detect-paths}.
\end{claimproof}

Now, we break into cases depending on the structure of the graphs $H_1[C_i,C_j]$. Recall the structural description in \cref{fact:graph-structure}, and the definition of ``special''.

\medskip
\noindent\textbf{Case 1: $H_1[C_i,C_j]$ is not special for any $i\ne j$.} Note that if $H_1[C_i,C_j]=H_2[C_i,C_j]$ for all $i,j$, we trivially have $X_2(v)=X_1^*(v)$, and the desired result \cref{eq:goal-coupling-disparity} follows. So, we assume that $H_1[C_i,C_j]\ne H_2[C_i,C_j]$ for some $i,j$. Since $H_1,H_2$ can both be interpreted as generalised disparity graphs of $H$ with respect to $c$, this means that $H_1[C_i,C_j]$ is the bipartite complement of $H_2[C_i,C_j]$.

By \cref{fact:D-equitable}, $H_2[C_i,C_j]$ is (bi)regular with at most half the total possible number of edges, so $H_1[C_i,C_j]$ is biregular with at \emph{least} half the total possible number of edges. Since in this case we are assuming that $H_1[C_i,C_j]$ is not special, by \cref{fact:graph-structure} $H_1[C_i,C_j]$ is connected.

By \cref{fact:connected}, in $H_1[C_1\cup \dots\cup C_k]$ there is some path from every vertex to $C_i$. Since $H_1[C_i,C_j]$ is connected, this means that in fact $H_1[C_1\cup \dots\cup C_k]$ is connected, meaning that $X_1^*(v)=C_1\cup \dots \cup C_k\supseteq X_2(v)$, and the desired result \cref{eq:goal-coupling-disparity} follows.

\medskip
\noindent\textbf{Case 2: There is a special subgraph.} It remains to consider the case where $H_1[C_i,C_j]$ is special for some $i\ne j$. This means that $H_1[C_i,C_j]$ is a disjoint union of two complete bipartite graphs with bipartitions $(C_i^{\mr A},C_j^{\mr A})$ and $(C_i^{\mr B},C_j^{\mr B})$, for some partition of $C_i$ into equal-sized subsets $C_i^{\mr A},C_i^{\mr B}$ and some partition of $C_j$ into equal-sized subsets $C_j^{\mr A},C_j^{\mr B}$. Note that $H_1[C_i^{\mr A},C_j^{\mr A}]$ and $H_1[C_i^{\mr B},C_j^{\mr B}]$ are both connected.

Now, recall that we are assuming that $c$ is equitable for $H$, and by the definition of a generalised disparity graph it immediately follows that $c$ is equitable for $H_1$. 
By \cref{fact:connected}, in $H_1[C_1\cup \dots\cup C_k]$ there is some path from every vertex to $C_i$. So, $H_1[C_1\cup \dots\cup C_k]$ has at most two connected components (one including $C_i^{\mr A},C_j^{\mr A}$ and one including $C_i^{\mr B},C_j^{\mr B}$). If there is only one component, we are done as in the previous case. If there are two components, note that both components have the same degree statistics between each pair of parts $C_i,C_j$ (by equitability); since $C_i^{\mr A}$ and $C_i^{\mr B}$ have the same size, this means that both components have exactly half the vertices of $H_1[C_1\cup \dots\cup C_k]$. So, $|X^*_1(v)|=|C_1\cup \dots \cup C_k|/2\ge |X_2(v)|/2$; this is the desired inequality \cref{eq:goal-coupling-disparity}.
\end{proof}

Next, we need a graph-theoretic lemma about distances in sparse graphs (i.e., graphs of bounded degree). Roughly speaking, this lemma says that for any set of vertices $C$ in a sparse connected graph, we can find a vertex $y$ and a set of distances $\mc I$, such that the number of vertices in $C$ which see $y$ at those distances is not too large and not too small. This means that if $C$ is large, then there are many vertices which see $y$ at distances in $\mc I$ and many vertices which do not see $y$ at distances in $\mc I$ (i.e., it is not the case that $y$ has the same distance to almost all vertices).

\begin{lemma}\label{lem:findingavertex}
Let $F$ be a connected graph of maximum degree at most $\Delta\geq 2$, and fix a vertex subset $C\subseteq V(F)$. Then, for any $1\le a\leq |C|$, there is a vertex $y \in V(G)$ and a subset $\mc I\subseteq \mb N$, such that the number of vertices $v\in C$ with $\on{dist}_G(y,v) \in \mc I$ lies in the interval $[a,a\Delta)$. 
\end{lemma}

\begin{proof}[Proof of \cref{lem:findingavertex}]
For any vertex $y$ and set $\mc I\subseteq \mb N$, let $Q_{\mc I}(y)$ be the number of vertices in $C$ whose distance to $y$ lies in $\mc I$. Fix a vertex $z$.

\medskip\noindent\textbf{Case 1: there are at least $a$ vertices in $C$ with the same distance $d$ to $z$.} In this case, let $z_d=z$ and recursively define a sequence of vertices $z_{d-1},\dots,z_0$ by choosing $z_i$ to be a neighbour of $z$ such that $Q_{\{i\}}(z_i)\ge Q_{\{i+1\}}(z_{i+1})/\Delta$ (this is possible, because the $Q_{\{i+1\}}(z_{i+1})$ vertices in $C$ which have distance $i+1$ to $z_{i+1}$ are all at distance $i$ to at least one of the at most $\Delta$ neighbours of $z_{i+1}$).

Note that $Q_{\{d\}}(z_{d})\ge a$ and $Q_{\{0\}}(z_{0})=1$, so as we gradually decrease $i$ from $d$ to $0$, we must find some $i$ for which $Q_{\{i\}}(z_{i})\in [a,a\Delta)$ (and then we can take $y=z_{i}$ and $\mc I=\{i\}$).

\medskip\noindent\textbf{Case 2: for every $d\in \mb N$, there are at most $a$ vertices in $C$ with distance $d$ to $z$.} In this case, let $\mc I(i)=\{0,1,\dots,i-1\}$. Note that $Q_{\mc I(0)}(z)=0$ and $Q_{\mc I(n)}(z)=|C|$ for large enough $n$, and note that $Q_{\mc I(i+1)}(z)-Q_{\mc I(i)}(z)\le a$ for each $i$. So, as we gradually increase $i$ from $0$ to $n$, we must find some $i$ for which $Q_{\mc I(i)}(z)\in [a,a\Delta)$ (and then we can take $z=y$ and $\mc I=\mc I(i)$).
\end{proof}

Finally, the following very simple lemma will be used in the ``fingerprinting'' argument briefly outlined in \cref{subsec:2WL-outline}. For a set of vertices $A$ in a graph $G$, write $N(A)=\bigcup_{v\in A}N(v)$ for the set of all vertices adjacent to a vertex in $A$.

\begin{lemma}\label{lem:findset}
Let $F$ be a non-empty biregular bipartite graph with bipartition $C\cup B$, where the vertices in $B$ have degree $d_B$. Then there is a subset $S\subseteq B$ of size at most $|B|/2d_B+1$ with the property that $|N(S)|\geq |C|/4$. 
\end{lemma}
\begin{proof}
Let $S_0=\emptyset$, and for $i=1,..., \lceil |C|/(2d_B) \rceil$, recursively define $S_i=S_{i-1}\cup \{z_i\}$, where $z_i\in B$ is a vertex which has the maximum possible number of neighbours in $C\setminus N(S_{i-1})$. For each $i$, note that $N(S_{i-1})\le (i-1)d_B<(|C|/(2d_B))d_B \leq |C|/2$, so $C\setminus N(S_{i-1})\ge |C|/2$. 
The degree of every vertex in $C$ is $d_B |B|/|C|$, so the vertices in $C\setminus N(S_{i-1})$ are incident to at least $d_C\cdot |C|/2=d_B|B|/2$ edges. Thus, $z_i$ is incident to at least $d_B/2$ vertices in $C\setminus N(S_{i-1})$. It follows that $|N(S)|\geq (|C|/2d_B)(d_B/2)=|C|/4$.   
\end{proof}

\subsection{Bounding the degrees}
Now we begin to get to the heart of the proof of \cref{thm:smoothed-disparity}. Recall from the discussion at the start of this section that we are going to define a partition $V(G)=V_{\mr{core}}\cup U$, and with each round of sprinkling, each vertex of $U$ will be assigned a unique colour if it has at least three neighbours in $V_{\mr{core}}$. Provided that $p\ge 100/n$, we will calculate that in each round of sprinkling, each vertex of $U$ is assigned a unique colour with probability at least $0.7$ independently. 

Our first key lemma is the simple observation that this random assignment of unique colours typically causes the disparity graph to have maximum degree $O(\log n)$. Indeed, if a vertex $v$ in the disparity graph has large degree, it must have quite a different neighbourhood to some other vertex $u$ in its colour class. But then it is very likely we assign a unique colour to some vertex which is in the neighbourhood of $v$ but not $u$, or vice versa, so colour refinement will distinguish $u$ and $v$ from each other.

\begin{lemma}\label{lem:degrees-log}
    Consider an $n$-vertex graph $H$ and an equitable colouring $c:V(H)\to \Omega$. Let $c'$ be a random refinement of $c$, obtained by assigning each vertex a new unique colour with probability at least $0.7$, independently. Then, whp $D(H,\mc R^*c')$ has maximum degree at most $4\log n$.
\end{lemma}
We remark that in our proof of \cref{thm:smoothed-disparity}, we will apply \cref{lem:degrees-log} with $H=G[U]$.
\begin{proof}[Proof of \cref{lem:degrees-log}]
    For a pair of vertices $u,v$, let $N_H(u)\triangle N_H(v)$ be the set of vertices in $H$ which are a neighbour of exactly one of $u,v$.
    \begin{claim}\label{claim:distinguish-colour-cheap}
        Whp the following holds. For every pair of vertices $u,v$ with $|N_H(u)\triangle N_H(v)|\ge 2\log n$, we have $\mc R^* c'(u)\ne \mc R^* c'(v)$.
    \end{claim}
    \begin{claimproof}
        For any particular pair of vertices $u,v$, if any of the vertices in $N_H(u)\triangle N_H(v)$ are assigned a unique colour by $c'$ then $\mc R^* c'(u)\ne \mc R^* c'(v)$. So, if $|N_H(u)\triangle N_H(v)|\ge 2\log n$ then $\Pr[\mc R^* c'(u)= \mc R^* c'(v)]\le 0.3^{2\log n}=o(n^{-2})$. The desired result follows from a union bound over pairs $u,v$.
    \end{claimproof}
    Now, suppose that the conclusion of \cref{claim:distinguish-colour-cheap} holds (from now on we no longer view $c'$ as random), and suppose for the purpose of contradiction that some vertex $u$ has at least $4\log n$ neighbours in $D(H,\mc R^* c')$. Let $N_{D(H,\mc R^* c')}(u)$ be this set of neighbours, and let $C'$ be the colour class of $u$ in $D(H,\mc R^* c')$. Observe that we must have $|C'|\ge 2$, because colour classes of size 1 correspond to isolated vertices in $D(H,\mc R^* c')$. Now, by \cref{fact:D-equitable}, every vertex is adjacent to at most half the vertices in $C'$, so if $v$ is a uniformly \emph{random} vertex in $C'$  (this is now the only source of randomness, as we are viewing $c'$ as fixed), then for every vertex $x\in V(H)$, we have $\Pr[vx\text{ is an edge of }D(H,\mc R^* c')]\le 1/2$. This means that $\mb E[|N_{D(H,\mc R^* c')}(u)\cap N_{D(H,\mc R^* c')}(v)|]\le |N_{D(H,\mc R^* c')}(u)|/2$, so there must be some specific vertex $v\in C'$ with $|N_{D(H,\mc R^* c')}(u)\cap N_{D(H,\mc R^* c')}(v)|\le |N_{D(H,\mc R^* c')}(u)|/2$. That is to say, with respect to $D(H,\mc R^*c')$, the symmetric difference of neighbourhoods of $u$ and $v$ is at least $|N_{D(H,\mc R^* c')}(u)|/2\ge 2\log n$, and since $u,v$ lie in the same colour class $C'$, the same is true with respect to $H$. This contradicts the conclusion of \cref{claim:distinguish-colour-cheap}.
\end{proof}

\subsection{Splitting connected components via colour classes}
Our next key lemma is that in order to control the sizes of the connected components of the disparity graph (after sprinkling), it suffices to control the number of vertices of a single colour in a single component. We will prove this with an adaptive algorithm to explore a connected component, revealing randomness as we go. Along the way we also prove that sprinkling typically ensures that large components can be partitioned into full colour classes.

\begin{lemma}\label{lem:break}
Consider an $n$-vertex graph $H$ and an equitable colouring $c:V(H)\to \Omega$. Let $c'$ be a random refinement of $c$, obtained by assigning each vertex a new unique colour with probability at least $0.7$, independently. Then, whp $c'$ satisfies the following properties.
\begin{enumerate}
    \item Every connected component of $D(H,c)$ with at least $\log n$ vertices can be partitioned into colour classes of $\mc R^* c'$.
    \item For every colour class $C$ of $c$, and every connected component\footnote{Here, and in the rest of the section, we conflate components with their vertex sets (i.e., here $X$ is really a vertex set on which a connected component lies).} $X$ of $D(H,c)$, every connected component of $D(H,\mc R^* c')$ that intersects $C\cap X$ has at most $5|C\cap X|+5\log n$ vertices.
\end{enumerate}
\end{lemma}

\begin{proof}
Fix a colour class $C$ (of $c$) and a connected component $X$ (of $D(H,c)$), such that $C\cap X$ is nonempty (there are at most $n$ such choices). Suppose that $|X|\ge \log n$ (noting that (2) follows deterministically by \cref{lem:disparity-components} when $|X|< \log n$). We will show that with probability $1-o(1/n)$,
\begin{enumerate}
    \item $C\cap X$ can be partitioned into colour classes of $\mc R^* c'$, and
    \item every connected component of $D(H,\mc R^*c')$ that intersects $C\cap X$ has at most $5|C\cap X|+5\log n$ vertices.
\end{enumerate}
The desired result will follow by a union bound over all nonempty choices of $C\cap X$.

Say that a vertex $v\in U$ is \emph{special} if it is assigned a new unique colour by $c'$. We reveal which vertices are special in an adaptive fashion, according to the following 2-phase algorithm.
\begin{itemize}
    \item \textbf{Phase 1:} Consider an arbitrary subset $Y\subseteq X$ of size $\log n$, and reveal which of the vertices in $(C\cap X)\cup Y$ are special.
    \begin{itemize}
            \item If none of these vertices of $X$ are special, abort the whole algorithm.
            \end{itemize}
    \item \textbf{Phase 2:} Let $t=0$, and repeatedly do the following.
    \begin{enumerate}[label=(\alph*)]
        \item Let $E_t\supseteq C\cap X$ be the set of vertices whose specialness has been revealed so far, let $S_t\subseteq E_t$ be the set of vertices revealed to be special so far, and let $c_t$ be the colouring obtained from $c$ by assigning each vertex in $S_t$ a new unique colour (i.e., $c_t$ represents the information about $c'$ that we know so far).
        \item For each vertex $v\notin E_t$, let $c_t^v$ be the colouring obtained from $c_t$ by assigning $v$ a new unique colour (i.e., $c_t^v$ is the hypothetical colouring we would have after this step, if we revealed $v$ to be special).
        \item If there is some $v\notin E_t$ such that the number of colour classes of $\mc R^* c_t^v$ intersecting $E_t$ exceeds the number of colour classes of $\mc R^* c_t$ intersecting $E_t$, then choose such a $v$ arbitrarily, and reveal whether it is special.
        \begin{itemize}
            \item If there is no such $v$, then set $\tau=t$ and abort the algorithm.
        \end{itemize}
        \item Increment $t$ (i.e., set $t=t+1$).
    \end{enumerate}
\end{itemize}

To briefly summarise: in Phase 1, we ensure that there is a vertex in $X$ with a unique colour (this will imply (1), and is also an ingredient in the proof of (2)), then in Phase 2 we keep revealing whether vertices are special, as long as their specialness would create new colour classes among the previously revealed vertices.

First, note that
\[
    \Pr[\text{the algorithm aborts in Phase 1}]\le 0.3^{\log n}=o(1/n).\label{eq:easy-abort}
\]
For the rest of the proof, we assume that Phase 1 has completed without aborting (all probability calculations should be interpreted as being conditional on the information revealed by this phase).

For all $t\le \tau$ let $N_t$ be the number of vertices in $E_t$, minus the number of colour classes of $\mc R^* c_{t}$ intersecting $E_{t}$. Note that $N_0\le |C\cap X|+\log n$. Each time when step (c) of Phase 2 occurs (for a vertex $v$), either $v$ is special, in which case $N_t$ decreases by at least 1, or $v$ is not special, in which case $N_t$ increases by 1. Since vertices are more likely than not to be special (each vertex is special with probability 0.7), this process typically does not continue for long, as follows.

\begin{claim}\label{claim:size_explored}
We have
    $|E_\tau|\leq 5|C\cap X| + 5\log n$ with probability $1-o(1/n)$.
\end{claim}
\begin{claimproof}
    Artificially extend the definition of $N_t$ for $t> \tau$, by taking $N_{t+1}=N_{t+1}+1$ with probability 0.3 and $N_{t+1}=N_{t+1}-1$ with probability 0.7. For $s= |C\cap X| + \log n$ we have $\mb E [N_{4s}]\le s+(0.3-0.7)4s=-0.6s$. Also note that $E_{t}=t+s$. Thus $|E_\tau|>5s$ implies that $\tau>4s$. In particular, not every vertex in $E_{4s}$ lies in a distinct colour class of $\mc R^* c_{4s}$ (this event would had cause Phase 2 to terminate- see step (c)), meaning that $N_{4s}>0$. So
    \[\Pr\big[|E_\tau|> 5|C\cap X| + 5\log n\big]\le\Pr[N_{4s} > 0]\le \Pr[N_{4s}-\mb E N_s > 0.6s]=o(1/n)\]
    by a Chernoff bound (\cref{lem:chernoff}).
\end{claimproof}

Now, let $W$ be the union of all the colour classes of $\mc R^* c_\tau$ which are entirely contained in $E_\tau$. The following two claims essentially complete the proof of \cref{lem:break}.
\begin{claim}\label{claim:prepartition_colours}
    $C\cap X$ can be partitioned into colour classes of $\mc R^* c_\tau$, so $C\cap X\subseteq W$.
\end{claim}
\begin{claimproof}
This claim is where we use the fact that Phase 1 of the algorithm did not abort: let $v$ be a special vertex in $X$, so $v$ has a unique colour with respect to $\mc R^* c_\tau$.

Since $\mc R^* c_\tau$ is a refinement of $c$, we can interpret $D(G,c)$ as a generalised disparity graph of the form $D_L(G,\mc R^* c_\tau)$. Note that there is a path in $D(G,c)$ between a vertex of colour $\mc R^* c_\tau(v)$ and any vertex in $C\cap X$, but there is no such path between a vertex of colour $\mc R^* c_\tau(v)$ and any vertex in $C\setminus X$ (recalling that $v$ is the unique vertex with colour $\mc R^* c_\tau(v)$). So, by \cref{fact:detect-paths}, the vertices in $C\cap X$ and the vertices in $C\setminus X$ are assigned disjoint sets of colours by $\mc R^* c_\tau$.
\end{claimproof}

\begin{claim}\label{claim:partition_components}
    $W$ can be partitioned into connected components of $D(H,\mc R^* c')$.
\end{claim}

\begin{claimproof}
We need to prove that there is no edge between $W$ and $V(H)\setminus W$ in $D(H,\mc R^* c')$. Since $\mc R^* c'$ is a refinement of $\mc R^* c_\tau$, and $W$ can be partitioned into colour classes of $c_\tau$, it suffices to prove that there is no edge between $W$ and $V(H)\setminus W$ in $D(H,\mc R^* c_\tau)$. For the rest of this proof, all graph-theoretic language is with respect to $D(H,\mc R^* c_\tau)$, and all colour classes are with respect to $c_\tau$.

Suppose for the purpose of contradiction that there is an edge between a colour class $Z\subseteq E_\tau$ and a colour class $Y\not\subseteq E_\tau$, and consider any $v\in Y\setminus E_\tau$. By \cref{fact:D-equitable}, the number of neighbours of $v$ in $Z$ lies in the range $[1,|Z|/2]$; in particular, $v$ has both neighbours and non-neighbours in $Z$. Recall that $c_\tau^v$ is the colouring obtained from $c_\tau$ by assigning $v$ a new unique colour, and note that $\mc R^* c_\tau^v$ assigns disjoint sets of colours to the neighbours of $v$ and the non-neighbours of $v$. So, the number of colour classes of $\mc R^* c_\tau^v$ intersecting $E_\tau$ exceeds the number of colour classes of $\mc R^* c_\tau$ intersecting $E_\tau$, contradicting the definition of $\tau$ (see Phase~2(c)).
\end{claimproof}

\cref{claim:prepartition_colours} immediately implies (1). Then, \cref{claim:partition_components,claim:prepartition_colours} imply that every connected component of $D(H,\mc R^* c')$ intersecting $C\cap X$ is contained in $E_\tau$. By \cref{claim:size_explored}, with probability $1-o(1/n)$ every such component $X$ has at most $5|X\cap C| + 5\log n$ vertices, proving (2).
\end{proof}

\subsection{Splitting colour classes via distance information}

Given \cref{lem:break}, our goal is now to show that sprinkling tends to break up intersections of the form $C\cap X$ (where $C$ is a colour class and $X$ is a connected component of the disparity graph) into much smaller parts. We prove a general-purpose lemma that accomplishes this, which will be applied several times (each time with different parameters) in the proof of \cref{thm:smoothed-disparity}.
\begin{definition}
For a graph $H$ and a colouring $c:V(H)\to \Omega$, say that $D(H,c)$ is \emph{$s$-bounded} if for every colour class $C$ of $c$, and every connected component $X$ of $D(H,c)$, we have $|C\cap X|\le s$.
\end{definition}
\begin{lemma}\label{lem:graph-shrinks}
Consider integers $s,s',\Delta$ such that $s\ge s'\ge 100\log n$ and $s'/(8\Delta)\geq 10\log s$. Consider an $n$-vertex graph $H$ and an equitable colouring $c:V(H)\to \Omega$, such that $D(H,c)$ is $s$-bounded and has maximum degree at most $\Delta$. Let $c'$ be a random refinement of $c$, obtained by assigning each vertex a new unique colour with probability at least $0.7$, independently. Then whp $D(H,\mc V^*c')$ is $s'$-bounded.
\end{lemma}

Recalling \cref{lem:degrees-log} (which shows that the maximum degree of the disparity graph is at most $O(\log n)$), we can apply \cref{lem:graph-shrinks} with $s=n$, $s'=O((\log n)^2)$ and $\Delta=O(\log n)$ to establish $O((\log n)^2)$-boundedness. Actually, this would already be enough to prove \cref{thm:smoothed-disparity} if we were willing to use group-theoretic canonical labelling schemes on small components of the disparity graph, but for a combinatorial algorithm we need to go down to $O(\log n)$-boundedness. For this we will need some additional control over degrees in the disparity graph, via the following somewhat delicate lemma.
\begin{definition}
    For a graph $H$ and a colouring $c:V(H)\to \Omega$, say that $D(H,c)$ is \emph{$(r,d)$-degree-bounded} if for every colour class $C$ of $c$ which intersects some component of $D(H,c)$ in at least $r$ vertices, every vertex $y\in V(H)$ has at most $d$ neighbours in $C$, with respect to $D(H,c)$.
\end{definition}
\begin{lemma}\label{lem:degree-shrinks}
Consider integers $s,s'$ such that $s\ge s'\ge 100\log n$. Consider an $n$-vertex graph $H$ and a colouring $c:V(H)\to \Omega$, such that $D(H,c)$ is $s$-bounded and $D(H,c)$ has maximum degree less than $s'/8$. Let $c'$ be a random refinement of $c$, obtained by assigning each vertex a new unique colour with probability at least $0.7$, independently. Then whp $D(H,\mc V^*c')$ is $(s',10\log s)$-degree-bounded.
\end{lemma}

Both \cref{lem:graph-shrinks,lem:degree-shrinks} are proved with the same ``fingerprint'' idea (briefly sketched in \cref{subsec:2WL-outline}). In fact, we will be able to deduce both of \cref{lem:graph-shrinks,lem:degree-shrinks} from the following technical lemma.

\begin{lemma}\label{lem:graph_shrinks_certificate}
Consider integers $s,s'$ such that $s\ge s'\ge 100\log n$. Consider an $n$-vertex graph $H$ and a colouring $c:V(H)\to \Omega$, such that $D(H,c)$ is $s$-bounded. Let $c'$ be a random refinement of $c$, obtained by assigning each vertex a new unique colour with probability at least $0.7$, independently. Then, the following property holds whp.

For a set of vertices $S$, a vertex $y$ and a subset $\mc I\subseteq \mc N\cup \{\infty\}$, let $Q_{\mc I}(S,y)$ be the number of vertices in $S$ whose distance to $y$, with respect to $D(H,c)$, lies in $\mc I$. Then, for every vertex $y\in V(G)$, every $\mc I\subseteq \mb N\cup\{\infty\}$ and every colour class $C'$ of $\mc V^* c'$ which intersects some component of $D(H,\mc V^*c')$ in at least $s'$ vertices, we have $Q_{\mc I}(C',y)\notin [10\log s, s'/8)$.
\end{lemma}

Before proving \cref{lem:graph_shrinks_certificate}, we deduce \cref{lem:graph-shrinks,lem:degree-shrinks}.

\begin{proof}[Proof of \cref{lem:graph-shrinks}]
Suppose that the conclusions of \cref{lem:break}(1) and \cref{lem:graph_shrinks_certificate} both hold, and suppose for the purpose of contradiction that there is a colour class $C'$ of $\mc V^* c'$ and connected component $X'$ of $D(H,\mc V^* c')$ such that $|C'\cap X'|\ge s'$.

By \cref{lem:disparity-components}, there is a connected component $X$ of $D(H,c)$ which intersects $C'$ in at least $s'/2\ge 50\log n$ vertices. By the conclusion of \cref{lem:break}(1), we must have $X\supseteq C'$. By \cref{lem:findingavertex}, with $F=D(H,c)[X]$, there is a vertex $y\in X$ and a set $\mc I$ such that $Q_{\mc I}(C',y)\in [s'/(8\Delta), s'/8)$. Since we are assuming $s'/(8\Delta) \geq 10\log s$, this contradicts the statement of \cref{lem:graph_shrinks_certificate}.
\end{proof}

\begin{proof}[Proof of \cref{lem:degree-shrinks}]
Suppose that the conclusion of \cref{lem:graph_shrinks_certificate} holds, and suppose for the purpose of contradiction that there were some vertex $y$ and colour class $C'$ of $\mc V^*c'$ which intersects some component of $D(H,\mc V^*c')$ in at least $s'$ vertices, such that $y$ has more than $10\log s$ neighbours in $C'$ with respect to $D(H,\mc V^* c')$. Note that this number of neighbours is less than $s'/8$, by our assumption on the maximum degree of $D(H,c)$ and \cref{fact:degrees-decrease}.

But taking $\mc I=\{1\}$ or $\mc I=(\mb N\cup \{\infty\})\setminus\{1\}$, and recalling the definitions of the disparity graphs $D(H,c)$ and $D(H,\mc V^* c')$, we can write the above number of neighbours as $Q_{\mc I}(C',y)$, contradicting \cref{lem:graph_shrinks_certificate}.
\end{proof}
Now we prove \cref{lem:graph_shrinks_certificate}. The idea is that if there were $C',y$ satisfying $Q_{\mc I}(C',y)\in [10\log s, s'/4)$, then every other vertex $v$ in the colour class $B'$ of $y$ (with respect to $\mc V^* c' $) would also have $Q_{\mc I}(C',v)\in [10\log s, s'/4)$, due to the fact that the 2-dimensional Weisfeiler--Leman algorithm can detect distances in $D(H,c)$ (which can be interpreted as a generalised disparity graph $D_L(H,\mc V^* c)$). We would then be able to define an auxiliary bipartite graph and use \cref{lem:findset} to find a relatively small ``fingerprint'' set of vertices in $B'$ which have a wide variety of distances to the vertices in $C'$ (with respect to $D(H,c)$). But this situation can be ruled out whp with a union bound: if, in $D(H,c)$, a set of vertices $S$ has a wide variety of distances to many other vertices, then it is very unlikely that that the vertices in $S$ are all given the same colour by $\mc V^*c'$.
\begin{proof}[Proof of \cref{lem:graph_shrinks_certificate}]
First, we can use the union bound to show that whp certain sets of vertices (with a rich variety of distances to other vertices) are not monochromatic with respect to $\mc V^*c'$. For a set of vertices $S$ we denote by $\on{Diff}(S)$ the set of all vertices in $D(G,c)$ which do not see all vertices of $S$ at the same distance.
\begin{claim}\label{claim:diff-union-bound}
Whp the following holds. For every colour class $B$ of $c$ and connected component $X$ of $D(H,c)$, and every subset $S\subseteq B\cap X$ satisfying
\begin{equation}\label{eq:certificate}
\binom{|B\cap X|}{|S|}0.3^{|\on{Diff}(S)|}\le \frac1{n^2},
\end{equation}
$S$ is not contained inside a single colour class of $\mc V^*c'$.
\end{claim}
\begin{claimproof}
For a particular set $S$, if any of the vertices in $\on{Diff}(S)$ are assigned a unique colour by $c'$ then $S$ does not lie inside a single colour class of $\mc V^*c$. This follows from \cref{fact:detect-distances}, noting that we can interpret $D(H,c)$ as a generalised disparity graph $D_L(H,\mc V^*c')$. So, the probability that a particular set $S$ is contained inside a single colour class of $\mc V^*c'$ is at most $0.3^{|\on{Diff}(S)|}$. The desired result then follows from a union bound, noting that there are at most $n$ nonempty subsets of the form $|B\cap X|$.
\end{claimproof}
Suppose that the conclusions of \cref{claim:diff-union-bound} and \cref{lem:break}(1) both hold, and suppose for the purpose of contradiction that there is a colour class $C'$ of $\mc V^* c'$, connected component $X'$ of $D(H,\mc V^* c')$, vertex $y$ and set $\mc I$ such that $|C'\cap X'|\ge s'$ and $Q_{\mc I}(C',y)\in [10\log s, s'/8)$.

Let $B'$ be the colour class containing $y$ with respect to $\mc V^* c'$ (we may or may not have $B'=C'$), and let $F$ be the auxiliary bipartite graph with parts $C'$ and $B'$, with an edge between $u\in C'$ and $v\in B'$ if the distance between $u$ and $v$ in $D(H,c)$ lies in $\mc I$. We can interpret $D(H,c)$ as a generalised disparity graph $D_L(H,\mc V^* c')$, so by \cref{fact:detect-distances}, $F$ is a biregular bipartite graph. In this graph the degree of each vertex in $B'$ is exactly $Q_{\mc I}(C',y)$. By \cref{lem:findset}, there is a subset $S\subseteq B'$ of size at most $|C'|/(2Q_{\mc I}(C',y))+1\le |C'|/(19\log s)$ such that $|N_F(S)|\ge |C'|/4 \ge s'/4\ge 2Q_{\mc I}(C',y)$ (for both these inequalities we used our assumption $Q_{\mc I}(C',y)\in [10\log s, s'/8)$).

Now, fix a vertex $v\in S$ and note that $\on{Diff}(S)\supseteq N_F(S)\setminus N_F(v)$ (because every vertex in $N_F(S)\setminus N_F(v)$ sees some vertex in $S$ at a distance in $\mc I$, but does not see $v$ at a distance in $\mc I$). So, $|\on{Diff}(S)|\ge |N_F(S)|-Q_{\mc I}(C',y)\ge |C'|/8$.

Then (as in the deduction of \cref{lem:graph-shrinks}), \cref{lem:disparity-components} tells us that there is a connected component $X$ of $D(H,c)$ intersecting $C'$ in at least $s/2$ vertices, and  by the conclusion of \cref{lem:break}(1), we have $X\supseteq C'$. We must also have $B'\subseteq X$ (if there were a vertex in $B'$ lying in a different component to the vertices of $C'$, we would have $Q_{\mc I}(C',y)\in \{0,|C'|\}$, which contradicts our assumption $Q_{\mc I}(C',y)\in [10\log s, s'/8)$). Letting $B$ be the colour class of $c$ containing $B'$, recalling that $|C'|\ge s'\ge 100\log n$, and recalling that $D(H,c)$ is $s$-bounded, we compute
\[
\binom{|B\cap X|}{|S|}0.3^{|\on{Diff}(S)|}\leq|B\cap X|^{|S|}0.3^{|C'|/8}\le s^{|C'|/(19\log s)}0.3^{|C'|/8}\le \frac1{n^2}.
\]
But $S$ is a subset of the colour class $B'$ of $\mc V^*c'$, contradicting the conclusion of \cref{claim:diff-union-bound}.
\end{proof}

\subsection{Putting everything together}

We now prove \cref{thm:smoothed-disparity}, combining all the ingredients collected so far. By \cref{cor:small-components} and \cref{fact:2WL-efficient}, it suffices to prove the following theorem.
\begin{theorem}\label{thm:smoothed-disparity-cheap}
Consider any $p\in [0,1/2]$ satisfying $p\ge 100/n$, consider any graph $G_0$ on the vertex set $\{1,\dots,n\}$, and let $G_{\mr{rand}}\sim \mb G(n,p)$. Then whp every connected component of $D(G_0\triangle G_{\mr{rand}},\mc V^*\phi_G)$ has $O(\log n)$ vertices.
\end{theorem}
\begin{remark}
    Our proof approach shows that the implicit constant in ``$O(\log n)$'' can be made as small as desired, by increasing the ``100'' in the assumption on $p$.
\end{remark}

\begin{proof}
Let $G = G_{0} \triangle G_{\mr{rand}}$.
First note that we can assume $p\le 2\log n/n$, as otherwise the desired result follows from \cref{thm:smoothed-colour-refinement} (note that if $\mc R^*\sigma$ assigns all vertices distinct colours, then so does $\mc V^*\phi_G$, by \cref{fact:2WL-refines-CR}, and $D(G,\mc V^*\sigma)$ consists only of isolated vertices. As outlined, we view the random perturbation $G_{\mr{rand}}$ as the union of eight slightly sparser random perturbations $G_{\mr{rand}}^i\sim \mb G(n,p')$, for $i=1,...,8$, where $p'$ is chosen such that $1-p=(1-p')^8$. Since we are assuming $p\ge 100/n$, we have $p'\ge 10/n$.

Since $p'\ge 10/n$, standard results show that $|V_{3}(G_{\mr{rand}}^1)|\ge n/2$ whp (see e.g.\ \cite{PSW96}). So, whp we can fix a subset $V_{\mr{core}} \subseteq V_{3}(G_{\mr{rand}}^{\mr 1})$ with size $n/2$. Let $U=V(G)\setminus V_{\mr{core}}$, and for each $i\in \{2,\dots,8\}$, let $S_i$ be the set of vertices in $U$ which have at least three neighbours in $V_{\mr{core}}$ with respect to $G_{\mr{rand}}^i$. We use the sets $S_i$, together with the 2-dimensional Weisfeiler--Leman algorithm, to recursively define a sequence of vertex-colourings $c_1,\dots,c_8$ of $U$, as follows.

\begin{itemize}
    \item Let $c_1=\mc V^* \phi_{G[U]}$ (this is a colouring of vertices in $U$)
    \item For $i\in\{2,\dots,8\}$:
    \begin{itemize}
        \item 
    let $b_i$ be the vertex-colouring obtained from $c_{i-1}$ by giving all vertices in $S_i$ a unique colour, and
    \item let $c_i=\mc V^* \phi_{G[U],b_i}$ (recalling from \cref{def:2WL} that $\phi_{G[U],b_i}$ is the colouring of pairs of vertices of $G[U]$, obtained by augmenting the ``trivial'' colouring $\phi_{G[U]}$ with the vertex-colouring $b_i$).
    \end{itemize}
\end{itemize}

The plan is to first reveal $G[U]$ and $c_1$, and then reveal $c_2,\dots,c_8$ one-by-one, studying how the components, colour classes and degrees change along the way (using \cref{lem:degrees-log,lem:break,lem:graph-shrinks,lem:degree-shrinks}). The following claim justifies this plan.

\begin{claim}\label{claim:components-final}
Whp the largest connected component in $D(G,\mc V^* \phi_G)$ has at most twice as many vertices as the largest connected component in $D(G[U],c_8)$.
\end{claim}
\begin{claimproof}[Sketch proof of claim]
By \cref{prop:3core}, whp $\mc R^*\sigma$ assigns each vertex in the 3-core $V_{3}(G_{\mr{rand}})\supseteq V_{\mr{core}}\cup S_2\cup \dots\cup S_8$ a unique colour, and by \cref{fact:2WL-refines-CR} the same is true for $\mc V^*\phi_G$. We will see that the desired conclusion holds whenever this is the case.
  
Extend each $c_i$ to a colouring $c_i^{\mr{ext}}$ of $V(G)$, by assigning each vertex in $V_{\mr{core}}$ a unique colour. Note that $D(G,c_i^{\mr{ext}})$ has the same components as $D(G[U],c_i)$ (plus singleton components for each vertex of $V_{\mr{core}}$), so by \cref{lem:disparity-components} it suffices to show that $\mc V^* \phi_G$ is a refinement of each $c_i^{\mr{ext}}$. This can be proved inductively, with eight applications of \cref{fact:2WL-consistent}
(note that we are assuming $\mc V^*\phi_G$ assigns each vertex a unique colour, so $G[U]$ can be partitioned into full colour classes). 
\end{claimproof}

For the rest of the proof, we reveal all the edges in $G^1_{\mr{rand}}$ (which determines $U$), and reveal all the edges in $G[U]$. 
We assume that $|V_3(G^1_{\mr{rand}})|\ge n/2$ (so $V_{\mr{core}}$ and $U$ are well-defined).

Given \cref{claim:components-final}, our objective is to use the remaining randomness of $G^2_{\mr{rand}},\dots,G^8_{\mr{rand}}$ (via the colourings $c_2,\dots,c_8$) to prove that whp every component of $D(G[U],c_8)$ has at most $O(\log n)$ vertices.

First, note that the random sets $S_i$ are independent. For each $i\in \{2,\dots,8\}$, each vertex $v\in U$ is independently present in $S_i$ with probability \[\Pr[\on{Binomial}(n/2,10/n)\ge 3]=1-\frac{37}{2e^5}+o(1)\ge 0.7.\]
We now track how $D(G[U],c_i)$ evolves with $i$.

\medskip\noindent\textbf{Round 1.} Applying \cref{lem:degrees-log} with $H=G[U]$, $c=c_1$ and $c'=c_2$, we see that whp $D(G[U],c_2)$ has maximum degree at most $4\log n$. Reveal an outcome of $c_2$ such that this is the case (i.e., for the rest of the proof, all probabilities should be interpreted as being conditional on an outcome of $c_2$ with this property).

\medskip\noindent\textbf{Round 2.} Applying \cref{lem:graph-shrinks} with $H=G[U]$, $c=c_2$, $c'=c_3$, $\Delta=4\log n$, $s=n$ and $s'=10^4(\log n)^2$, we see that whp $D(G[U],c_3)$ is $10^4(\log n)^2$-bounded. Reveal an outcome of $c_3$ such that this is the case.

\medskip\noindent\textbf{Round 3.} 
By \cref{fact:degrees-decrease}, the maximum degree of $\Delta(D(H,c_3))$ is at most the maximum degree of $\Delta(D(H,c_2))$, which is at most $4\log n$. Applying \cref{lem:graph-shrinks} again, with $H=G[U]$, $c=c_3$, $c'=c_4$, $\Delta=4\log n$, $s=10^4(\log n)^2$ and $s'=10^4\log n\log \log n$, we see that whp $D(G[U],c_4)$ is $10^4\log n\log \log n$-bounded. Reveal an outcome of $c_4$ such that this is the case.

\medskip\noindent\textbf{Round 4.}
By \cref{lem:break}(2) with $H=G[U]$, $c=c_4$ and $c'=c_5$, using the $10^4\log n\log \log n$-boundedness we have just established in the previous round, whp every connected component of $D(G[U],c_5)$ has at most $10^5\log n\log \log n$ vertices. 
Also, applying \cref{lem:degree-shrinks} with $H=G[U]$, $c=c_4$, $c'=c_5$, $s=10^4\log n\log \log n$ and $s'=100\log n$ (noting that $s'$ is at least 8 times the maximum degree of $D(H,c_4)$, by \cref{fact:degrees-decrease}), we see that whp $D(G[U],c_5)$ is $(100\log n,11\log \log n)$-degree-bounded. Reveal an outcome of $c_5$ such that both of these properties hold.

Let $U^{\mr{big}}\subseteq U$ be the union of all sets of the form $C\cap X$ with size at least $100 \log n$, where $C$ is a colour class of $c_5$ and $X$ is a component of $D(G[U],c_5)$ (say such a set $C\cap X$ is a ``big intersection set''). Every connected component of $D(G[U],c_5)$ has at most $10^5\log n\log \log n$ vertices, so it contains at most $10^3 \log \log n$ big intersection sets. By degree-boundedness, every vertex has at most $11\log \log n$ neighbours (with respect to $D(G[U],c_5)$) in each big intersection set, so the maximum degree of $D(G[U],c_5)[U^{\mr{big}}]$ is at most $11\log \log n\cdot 10^3 \log \log n\le 10^5(\log \log n)^2$.

\medskip\noindent\textbf{Round 5.} By \cref{lem:break}(1) with $H=G[U]$, $c=c_5$ and $c'=c_6$, whp $U^{\mr{big}}$ can be partitioned into colour classes of $c_6$. Also, by \cref{lem:break}(2) (with the same $H,c,c'$), whp $U\setminus U^{\mr{big}}$ can be covered by connected components of $D(H,c_6)$, each with at most $10^6\log n$ vertices. Reveal an outcome of $c_6$ such that both these properties hold.

For $i\in \{5,6,7\}$, let $c_i^{\mr{big}}$ be the restriction of $c_i$ to $U^{\mr{big}}$, so $D(G[U],c_5)[U^{\mr{big}}]$ can be interpreted as a generalised disparity graph $D_L(G[U],c_5^{\mr{big}})$. Then $D(G[U^{\mr{big}}],c_6^{\mr{big}})$ has maximum degree at most $10^5(\log \log n)^2$ by \cref{fact:degrees-decrease}, and its components have at most $2\cdot 10^5\log n\log \log n$ vertices, by \cref{lem:disparity-components}.

\medskip\noindent\textbf{Round 6.} We apply \cref{lem:graph-shrinks} yet again, this time with $H=G[U^{\mr{big}}]$, $c=c_6^\mr{big}$, $c'=c_7^{\mr{big}}$, $\Delta=10^5(\log \log n)^2$, $s=2\cdot 10^5\log n\log \log n$ and $s'=100\log n$, to see that whp $D(G[U^{\mr{big}}],c_7^{\mr{big}})$ is $100\log n$-bounded. Reveal an outcome of $c_7$ such that this is the case.

Recalling that $U\setminus U^{\mr{big}}$ can be covered by connected components of $D(H,c_6)$ with at most $100\log n$ vertices, \cref{lem:disparity-components} tells us that in fact the whole of $D(G[U],c_7)$ is $2\cdot 10^5\log n$-bounded.

\medskip\noindent\textbf{Round 7.} Finally, apply \cref{lem:break} with $H$, $c=c_7$ and $c'=c_8$, to see that whp all connected components of $D(H,c_8)$ have at most $O(\log n)$ vertices, as desired.
\end{proof}

\section{Near-critical random graphs}\label{sec:random-graphs}

In this section we prove that for random graphs of any density, colour refinement whp assigns distinct colours to each vertex of degree at least 3 in the 2-core. By \cref{cor:kernel-canonical} and \cref{rem:automorphisms}, this implies \cref{thm:sparse-random-informal,thm:automorphisms}.

\begin{theorem}\label{thm:2-core}
For any sequence $(p_n)_{n\in \mb N}\in [0,1]^{\mb N}$, and $G\sim \mb G(n,p_n)$, whp every two distinct vertices $u,v\in V_{2,3}(G)$ are assigned different colours in the stable colouring $\mc R^*\sigma$.
\end{theorem} 
\cref{thm:2-core} follows directly from \cref{prop:2core} (with assumption \cref{item:2core-for-random}) in the case $p\ge (1+(\log n)^{-40})/n$, and it is vacuous for $p\le (1-\varepsilon)/n$, for any constant $\varepsilon>0$, as in this regime $V_{2,3}(G)$ is empty whp (see for example \cite[Lemma~2.10]{FK16}). So, we focus on the near-critical regime $0.9/n\le p\leq (1+(\log n)^{-2})/n$.

In this near-critical regime we cannot apply the machinery developed in \cref{prop:2core,prop:expansion-colour-refinement} Indeed, this machinery fundamentally relies on expansion properties only available for supercritical random graphs, and critical random graphs typically have poor expansion properties (in particular, most pairs of vertices are quite far from each other). But this gives us another way to proceed: if two vertices $u,v$ are far from each other (or, more generally, if for reasonably large $i$ we can find vertices which have distance $i$ from $u$ while being at distance greater than $i$ from $v$), then if we consider exploration processes starting from $u$ and $v$, we can run these processes for quite a long time while continuing to maintain some independent randomness between the two processes. Although this a rather small amount of randomness per step (compared to \cref{prop:expansion-colour-refinement}), we can accumulate this for many steps, to see that there is very likely to be some deviation between the degree statistics of these processes. This will allow us to  conclude with \cref{lem:receive_same_colour}.

The following lemma can be viewed as an analogue of \cref{prop:expansion-colour-refinement}, which will be used to execute the above plan. (It is much easier to prove than \cref{prop:expansion-colour-refinement}).

\begin{lemma}\label{lem:longpaths}
  Let $G\sim \mb G(n,p)$ for some $p$ satisfying $0.9/n\le p\le (1+(\log n)^{-40})/n$. Then whp $G$ satisfies the following property. For every pair of vertices $u,v$ such that $\mc R^*\sigma(u)= \mc R^*\sigma(v)$, we have $\mc S^{i}(\{u,v\})=\emptyset$ for some $i\le 10\log n$.
\end{lemma}
(Recall that the sets $\mc S^{i}(\{u,v\})$, defined in \cref{def:multisets}, describe the ``vertices that appear differently in the $i$-th step of the view exploration process'').

For our proof of \cref{lem:longpaths} we will need a simple lemma showing that in very sparse random graphs, every vertex has distance greater than $10\log n$ from most other vertices (therefore throughout the first $10\log n$ steps of an exploration process, there are always many vertices that have not yet been explored).

\begin{lemma}\label{lem:few-close}
    Let $G\sim \mb G(n,p)$ for some $p$ satisfying $p\le (1+1/\log n)/n$. Whp, $G$ does not have any vertex $v$ which is within distance $10\log n$ of at least $n/4$ different vertices.
\end{lemma}
\begin{proof}
For a vertex $v$, let $N^i(v)$ be the set of vertices at distance $i$ from $v$.
We claim that, whp, for all vertices $v$, and all $i\le 10\log n$, we have
\begin{equation}
    |N^i(v)|\le \max\Big\{n^{0.1},\;(1+5/\log n)|N^{i-1}(v)|\Big\}.\label{lem:neighbourhood-growth}
\end{equation}
Indeed, for any fixed $v$, imagine an exploration process where we iteratively reveal the sets $N^{i}(v)$ (at step $i$ we reveal all edges incident to vertices in $N^{i-1}(v)$). For fixed $i$, if we condition on any outcome of $N^{i-1}(v)$, then $|N^{i}(v)|$ is at most the number of edges between $N^{i-1}(v)$ and previously unrevealed vertices, which is stochastically dominated by the binomial distribution $\on{Binomial}\big(n|N^{i-1}(v)|,p\big)$. By a Chernoff bound (\cref{lem:chernoff}), \cref{lem:neighbourhood-growth} holds with probability at least $1-o(1/n^2)$, so the result holds by a union bound over all vertices $v$, and positive integers $i$.

Then, note that if \cref{lem:neighbourhood-growth} holds for each $i\le 10\log n$, the total number of vertices within distance $10\log n$ of $v$ is at most $\sum_{i=0}^{10\log n} (1+5/\log n)^{i} n^{0.1} \leq  (10\log n )(1+5/\log n)^{10\log n} n^{0.1} \le n/4$.
\end{proof}
Now we prove \cref{lem:longpaths}.
\begin{proof}[Proof of \cref{lem:longpaths}]
Fix $v,u\in [n]$. We consider the exploration process defined in the proof of \cref{prop:expansion-colour-refinement} (also used in the proofs of \cref{prop:2core,lem:expansion_SLi}), to iteratively reveal the sets $\mc S^{i-1}(\{u,v\})$.

Fix vertices $u,v$ and let $i\leq 10\log n$. Suppose we have so far revealed $\mc S^{1}(\{u,v\}),\dots,\mc S^{i-1}(\{u,v\})$, suppose that $\mc S^{i-1}(\{u,v\})\ne \emptyset$ and suppose that $|\mc S^{\le i-1}(\{u,v\})|\le n/2$. For vertices $w,x$, let $\xi_{wx}$ be the indicator random variable for the event that $wx$ is an edge of $G$, and let $a_x$ be the number of times $x$ appears in $\mc L^{i-1}(u,v)$, minus the number of times it appears in $\mc L^{i-1}(v,u)$. So,
\[|\mc L^i(u,v)|-|\mc L^i(v,u)|=\sum_{w,x\in V(G)} a_x\xi_{wx}.\]
Some of the $\xi_{wx}$ may have already been revealed. In particular, there are \[\big(n-|\mc S^{\le i-1}(\{u,v\})|\big)\cdot |\mc S^{i-1}(\{u,v\})|\ge n/2\] pairs of vertices $w\notin \mc S^{\le i-1}(\{u,v\})$ and $x\in \mc S^{i-1}(\{u,v\})$ for which $a_x\ne 0$ and $\xi_{wx}$ has not yet been revealed.
Reveal $\xi_{wx}$ for all but $n/2$ of these pairs,
putting us in a position to apply \cref{thm:LO}. Recall the definition of $M(\lfloor n/4\rfloor,p)$ from \cref{def:modalprobability}; conditional on information revealed so far, we see that the probability of the event $|\mc L^i(u,v)|-|\mc L^i(v,u)|=0$ is at most
\[M(\lfloor n/4\rfloor,p) = (1-p)^{\lfloor n/4\rfloor} \leq e^{-0.9/4}+o(1)\le 0.8\]
(recalling that $p\ge 0.9/n$). 
Recall from \cref{lem:receive_same_colour} that if $|\mc L^i(u,v)|\ne |\mc L^i(v,u)|$ then $\mc R^*\sigma (u)\ne \mc R^* \sigma(v)$. So, for our fixed $u,v$, the joint event that
\[\mc S^{i-1}(\{u,v\})\ne \emptyset\text{ and }|\mc S^{\le i-1}(\{u,v\})|\le n/2\text{ for all }i\le 10\log n,\quad \text{and }\mc R^*\sigma (u)=\mc R^* \sigma(v)\]
holds with probability at most $0.8^{10\log n}=o(n^{-2})$. Taking a union bound over choices of $u,v$, we see that whp the above event does not hold for any $u,v$. To finish the proof, observe that the property in \cref{lem:few-close} implies that $|\mc S^{\le i-1}(\{u,v\})|\le n/2$ for all $u,v$ and all $i\le 10\log n$.
\end{proof}

We need two final ingredients before completing the proof of \cref{thm:2-core}: in very sparse random graphs, whp there are no short cycles in close proximity to each other, and this implies a certain structural fact about vertices in $V_{2,3}(G)$.

\begin{lemma}\label{lem:no-close-cycles}
    Let $G\sim G(n,p)$ for some $p$ satisfying $p\le (1+(\log n)^{-2})/n$. Then whp $G$ does not have any connected subgraph with at most $200\log n$ vertices and at least two cycles.
\end{lemma}
\begin{proof}
An edge-minimal connected subgraph containing two cycles is always either:
\begin{itemize}
    \item a pair of vertex-disjoint cycles joined by a (possibly trivial) path, or
    \item two vertices with three internally disjoint paths between them.
\end{itemize}
The probability that $G$ contains two vertex-disjoint cycles of lengths $k_1,k_2$, joined by a path of length $k_3\ge 0$ (i.e., a path with $k_3$ edges) is at most
\[n^{k_1+k_2+k_3-1}p^{k_1+k_2+k_3}=n^{-1}(np)^{k_1+k_2+k_3}\le \frac 1{n}\bigg(1+\frac1{(\log n)^2}\bigg)^{k_1+k_2+k_3}.\]
The number of vertices in such a configuration is $k_1+k_2+k_3-1$, so the probability that such a configuration exists on at most $200\log n$ vertices is at most
\[\frac{(200\log n)^3}n \bigg(1+\frac1{(\log n)^2}\bigg)^{200\log n+1}=o(1).\]
Similarly, the probability that $G$ contains two vertices joined by three internally disjoint paths of lengths $k_1,k_2,k_3\ge 1$ is at most
\[n^{k_1+k_2+k_3-1}p^{k_1+k_2+k_3},\]
so essentially the same calculation shows that whp there is no such configuration on at most $200\log n$ vertices.
\end{proof}

\begin{lemma}\label{lem:strong-cycles}
    For an integer $k\ge 1$, let $G$ be a connected graph containing no connected subgraph on at most $20k$ vertices with at least two cycles. Consider two distinct vertices $u,v\in V_{2,3}(G)$. Then for $i\le k$ there is a vertex $w$ with $d(u,w)=i$ and $d(v,w)>i$, or with $d(v,w)=i$ and $d(u,w)>i$.
\end{lemma}

\begin{proof}
    
Let $w^1=u$ and $w^2=v$, and for $i\in \{1,2\}$ let $w_1^i,w_2^i,\ldots, w_{d_{i}}^i$ be the neighbours of $w^i$. For $i\in \{1,2\}$ and $j\in\{1,\dots,d_{i}\}$, let $W_j^i$ be the set of vertices within distance $k$ of $w_j^i$, with respect to the graph $G-\{u,v\}$ obtained by deleting $u$ and $v$.

It may be helpful to imagine an exploration process where we start at a vertex $w^i$ and ``explore outwards in the direction of $w_j^i$'', at each step discovering a new ``layer'' of vertices which are adjacent to previously discovered vertices. Then $W_j^i$ can be interpreted as the set of vertices discovered in the first $k$  steps of this process.

\begin{claim}\label{claim:cycle-in-Wji}
    If for any $i,j$ we have that $G[W_j^i]$ contains a cycle, then in fact $W_j^i$ contains a cycle $C$ with the property that $C$, together with a shortest path between $w_j^i$ and $C$, collectively comprise a subgraph $G_j^i\subseteq G[W_j^i]$ with at most $2k$ vertices.
\end{claim}

\begin{claimproof}
    Consider a breadth-first search tree $T$ (rooted at $w_j^i$) spanning $W_j^i$; if $W_j^i$ contains a cycle then $G[W_j^i]$ has an edge $xy$ which does not appear in $T$. Let $C$ be the cycle consisting of the edge $xy$, and the unique path in $T$ between $x$ and $y$. 
\end{claimproof} 

\begin{claim}\label{claim:path-between-Wji}
For any $i,j,i',j'$, if the sets $W_j^i$ and $W_{j'}^{i'}$ intersect, then there is a connected subgraph $G_{j,j'}^{i,i'}\subseteq G[W_j^i\cup W_{j'}^{i'}]$ containing $w_j^i$ and $w_{j'}^{i'}$, and fewer than $2k$ vertices in total.
\end{claim}
\begin{claimproof}
    If there is $x\in W_j^i\cap W_{j'}^{i'}$, we can simply take $G_{j,j'}^{i,i'}$ to be the union of a shortest path between $w_j^i$ and $x$ followed by a shortest path between $x$ and $w_{j'}^{i'}$.
\end{claimproof}
Now, consider the auxiliary multigraph $G^*$ constructed as follows.
\begin{itemize}
\item The vertex set of $G^*$ consists of the two vertices $u,v$ and all their neighbours (this is at most $2+d_{1}+d_{2}$ vertices, but it may be fewer if some neighbours of $u$ and $v$ coincide, or if $u$ and $v$ are adjacent). For each $i,j$, put an edge\footnote{If $u$ and $v$ are adjacent then $v$ is one of the $w^1_j$s and $u$ is one of the $w^2_j$s, but we put only one edge between $u$ and $v$.}between $w^i,w^i_j$ in $G^*$ (these edges are all also in $G$). Call these \emph{basic edges}.
\item For each $(i,j)\ne (i',j')$, if $W_j^i$ and $W_{j'}^{i'}$ intersect, then put an edge between $w_j^i$ and $w_{j'}^{i'}$ in $G^*$ if $w_{j}^{i} \neq w_{j'}^{i'}$. Call this a \emph{type-1 special edge}.
\item For each $i,j$, if $w_j^i$ is not incident to a type-1 special edge and $W_j^i$ contains a cycle with respect to $G$, then put a loop on $w_j^i$ in $G^*$. Call this a \emph{type-2 special edge}.
\end{itemize}
\begin{claim}\label{claim:reduced-cycles}
    $G^*$ does not have a connected component with at least two cycles.
\end{claim}
\begin{claimproof}
    For the purpose of contradiction, suppose such a connected component were to exist.
    Considering all possibilities for the structure of this component, we see that there must be a connected subgraph $G'\subseteq G^*$ with at least two cycles, at most three type-1 special edges, and at most two type-2 special edges (specifically, consider an edge-minimal connected subgraph with at least two cycles). Via \cref{claim:path-between-Wji,claim:cycle-in-Wji}, the special edges in $G'$ correspond to subgraphs of $G$ which each have at most $2k$ vertices, and the union of these subgraphs, together with the basic edges of $G'$, yields a connected subgraph of $G$ with at least two cycles and at most $(2+3)(2k+2)\le 20k$ vertices. 
    This contradicts our assumption on $G$.
\end{claimproof}

If every vertex of $G^*$ had degree at least 2, then $G^*$ would have a connected component with at least two cycles (since $u,v$ have degree at least 3). So, \cref{claim:reduced-cycles} shows that there is a vertex of $G^*$ which has degree 1. Suppose without loss of generality that this vertex is $w^1_1$, so $W_1^1$ does not contain any cycles (of $G$) and does not intersect any other $W_{j}^{i}$.

Recall that $w^1=u$ lies in the 2-core of $G$, so if we ``explore outwards from $u$ in the direction of $w_1^1$'' for long enough, we will eventually see a cycle or some other $w^i_j$. This does not happen within $k$ steps, so for every $i\le k$ there is some vertex $w\in W_1^1$ with $d(u,w)=i$. Since $W_1^1$ does not intersect any $W_{j}^{2}$, we have $d(v,w)>i$.
\end{proof}

Now we are ready to complete the proof of \cref{thm:2-core}.

\begin{proof}[Proof of \cref{thm:2-core} in the regime $0.9/n\le p\leq (1+(\log n)^{-2})/n$]Recall the definitions of $\mc L^{i}(u,v)$ and $\mc S^{i}(\{u,v\})$ (in \cref{def:multisets}) as the multisets of vertices that ``appear differently'' in the view exploration processes starting from $u$ and $v$. For any vertices $u,v,w$ with $d(v,w)=i$ and $d(u,w)>i$ (or vice versa), we have $w\in \mc S^{i}(\{u,v\})$. So, by \cref{lem:no-close-cycles} and \cref{lem:strong-cycles}, whp $G$ has the property that for every pair of distinct vertices $u,v\in V_{2,3}$, and every $i\le 10\log n$, we have $S^{i}(\{u,v\})\ne 0$. By \cref{lem:longpaths}, it follows that the vertices in $V_{2,3}(G)$ have distinct colours with respect to $\mc R^*\sigma$.
\end{proof}

\bibliographystyle{amsplain_initials_nobysame_nomr}
\bibliography{main.bib}

\appendix
\section{Canonical labelling proofs}\label{app:labelling}
In this section we provide some details for the (routine) proofs of \cref{cor:kernel-canonical,prop:kernel-canonical,cor:small-components}.

First, we make the basic (well-known) observation that for the purposes of efficient canonical labelling, it suffices to label each connected component separately. 

Recall that $\mc G_n$ is the set of $n$-vertex graphs (on the vertex set $\{1,\dots,n\}$); we write $\mc G_{\le n}\subseteq \bigcup_{k=0}^n \mc G_{k}$ for the set of all graphs on \emph{at most} $n$ vertices.

\begin{fact}\label{fact:components}
    Suppose $\Phi_{\mc H}$ is a canonical labelling scheme for a graph family $\mc H\subseteq \mc G_{\le n}$, such that $\Phi_{\mc H}(G)$ can be computed in time $T$ for each $G\in \mc H$. Let $\mc F$ be the family of graphs whose connected components are all isomorphic to a graph in $\mc H$. Then there is a canonical labelling scheme $\Phi$ for $\mc G_n$, such that for every $G\in \mc F$, we can compute $\Phi(G)$ in time $O(n^3\log n+nT +m)$.
\end{fact}
\begin{proof}
     For a graph $G\in \mc F$, our canonical labelling $\Phi(G)\in \mc S_n$ (interpreted as an ordering of the vertices of $G$) is computed as follows. First, identify the connected components of $G$ (which takes time $O(n+m)$). For each component $C$ (with $n_C$ vertices, say), we interpret $C$ as a graph on the vertex set $\{1,\dots,n_C\}$, and compute the ordering $\Phi_{\mc H}(C)$. Since there are at most $n$ components, we can do this in time $O(nT)$.

Now that we have a labelling on each component, we simply need to decide how the components are ordered relative to each other. To do this, we consider the adjacency matrix of each connected component. To unambiguously specify these adjacency matrices we need an ordering of the vertices of each component; we simply use the canonical labelling $\Phi_{\mc H}(C)$ that we computed above. Then, we sort the components lexicographically by their adjacency matrices\footnote{Here, we are viewing a $k\times k$ matrix as a string of length $k^2$, so lexicographic ordering makes sense.}.
This can be done in time $O(n^{3}\log n)$ (the string corresponding to each adjacency matrix has length at most $n^{2}$, so comparing two sequences takes time $O(n^{2})$, and we are sorting at most $n$ components)\footnote{We are being very crude here, it is straightforward to bring the time down to $O(n^2)$.}.

Note that we did make some ``arbitrary choices'' in this description. Indeed, if there were two components which have the same adjacency matrix, then we had to arbitrarily break a tie when sorting these components. These arbitrary choices are not a problem (i.e., we have indeed described a canonical labelling scheme), because any two outcomes of these arbitrary choices correspond to an automorphism of $G$.
\end{proof}
Taking $\mc H$ to be the set of all graphs on at most $C\log n$ vertices, for a constant $C$, \cref{cor:small-components} follows immediately from \cref{prop:D-to-G}, \cref{thm:CG} and \cref{fact:components}.

Next, we provide some explanation for \cref{prop:kernel-canonical}. First, we observe that the vertices in the kernel, the non-kernel vertices in the 2-core, and the vertices outside the 2-core can all be distinguished from each other via the colour information generated by the colour refinement algorithm. Let $V_2(G)$ be the vertex set of the 2-core of $G$.

\begin{lemma}
\label{lem:differentiatecoreandkernel}
Let $G$ be a graph. The colour $\mc R^{*} \sigma(v)$ is enough information to distinguish between the three possibilities $v\in V(G)\setminus V_2(G)$, $v \in V_2(G) \setminus V_{2,3}(G)$ and $v\in V_{2,3}(G)$.
\end{lemma}
\begin{proof}
Consider the process of ``peeling off'' vertices of degree less than 2, to generate the 2-core: at each step, we identify all vertices of degree less than 2, and remove them. It is easy to prove by induction that the colour of a vertex with respect to $\mc R^t\sigma$ is enough information to determine whether that vertex will be peeled off by the $t$-th peeling step (indeed, recall that $\mc R^t\sigma(v)$ tells us the number of neighbours that $v$ has in each colour, with respect to $\mc R^{t-1}\sigma(v)$).

Recall that $\mc R^*\sigma(v)$ contains all the information about $\mc R^t\sigma$ for each $t$ until stabilisation is reached. So, $\mc R^*\sigma(v)$ tells us whether $v\in V_2(G)$. Since $\mc R^*\sigma(v)$ has the property that every vertex of a given colour has the same number of neighbours in every other colour (\cref{fact:equitable}), it follows that, in the 2-core, the vertices of degree at least 3 have different colours than the vertices of degree 2.
\end{proof}

Now we explain how to reconstruct a connected graph $G$ given that the vertices in $V_{2,3}(G)\ne \emptyset$ are assigned distinct colours by $\mc R^{*} \sigma$.

\begin{proof}[Proof of \cref{prop:kernel-canonical}]
By \cref{lem:differentiatecoreandkernel}, if we know the multiset of colours in $\mc R^{*} \sigma$ then we know the multiset of colours assigned to vertices in $V_{2,3}(G)$. Start with a generic set of vertices with these colours. We want to show that the colours specify a unique way to put bare paths between these vertices to reconstruct the 2-core, and they also specify a unique way to attach trees to vertices of the 2-core to reconstruct $G$.

It would be possible to show this by induction (in a similar fashion to \cref{lem:differentiatecoreandkernel}), but it is perhaps most convenient/intuitive to proceed via universal covers (cf. \cref{rem:view-vs-UC}).

Indeed, to reconstruct the bare paths between vertices in $V_{2,3}(G)$, consider a vertex $u\in V_{2,3}(G)$, and consider the universal cover $\mathcal{T}_G(u)$ (recall from \cref{Universialtreecoverstep} that this is equivalent information to $\mc R^*\sigma(u)$). The vertices in this universal cover can be viewed as copies of vertices in $G$; we do not necessarily know the identities of these vertices, but we do know their colours (with respect to $\mc R^*\sigma$). Since the vertices in $V_{2,3}(G)$ receive unique, distinguishable colours (by the assumption of this proposition, and \cref{lem:differentiatecoreandkernel}), we can directly read off the lengths of the bare paths between $u$ and other vertices in $V_{2,3}(G)$, and we can read off what the other endpoints of these bare paths are.
    
To reconstruct the trees attached to vertices of $V_2(G)$, similarly consider a vertex $u\in V_{2}(G)$ and consider the universal cover $\mathcal{T}_G(u)$. Some of the children of the root are in $V_2(G)$ (we can see this from their colours, by \cref{lem:differentiatecoreandkernel}); after deleting these children (and their subtrees) we see the exact structure of the tree outside the 2-core attached to $u$.
\end{proof}

Finally, to prove \cref{cor:kernel-canonical}, we apply \cref{fact:components} with $\mc H$ being the set of graphs on at most $n$ vertices which are either CR-determined or outerplanar (recalling \cref{thm:outerplanarcl,thm:CR-determined-canonical,prop:kernel-canonical}).

\section{A simpler proof of a weaker result}\label{app:loglog}
Recall the definition of the disparity graph from \cref{def:disparity}. In this appendix we give the details of the argument sketched in \cref{subsec:sprinkling-outline}, proving the following simpler version of \cref{thm:smoothed-disparity}.

\begin{theorem}\label{thm:smoothed-disparity-cheap2}
Consider any $p\in [0,1/2]$ satisfying $p\ge 100\log \log n/n$, consider any graph $G_0$, and let $G_{\mr{rand}}\sim \mb G(n,p)$. Let $\sigma$ be the trivial colouring, assigning each vertex the same colour. Then whp the connected components of $D(G_0\triangle G_{\mr{rand}},\mc R^*\sigma)$ each have at most $\log n/\log \log n$ vertices.
In particular, by \cref{cor:small-components}, whp $G_0\triangle G_{\mr{rand}}$ can be tested for isomorphism with any other graph in polynomial time\footnote{\cref{cor:small-components} uses the Corneil--Goldberg exponential-time canonical labelling scheme on each connected component of the disparity graph, meaning that the components can have up to $O(\log n)$ vertices. However, here we can ensure that the components have at most $\log n/\log \log n$ vertices, so the Corneil--Goldberg labelling scheme is not necessary: we can afford to use a trivial factorial-time canonical labelling scheme on each connected component.}.
\end{theorem}

At a very high level the strategy to prove \cref{thm:smoothed-disparity-cheap2} is as follows. First, we view the random perturbation $G_{\mr{rand}}$ as the union of three slightly sparser random perturbations $G_{\mr{rand}}^{1},G_{\mr{rand}}^{2},G_{\mr{rand}}^{3}\sim \mb G(n,p')$. Assuming $p=\Omega(\log \log n/n)$, standard results show that the 3-core of $G_{\mr{rand}}^{1}$ (which is a subgraph of the 3-core of $G_{\mr{rand}}$) comprises \emph{almost all} the vertices of $G$, and by \cref{prop:2core}, whp almost all these vertices are assigned unique colours by $\mc R_G^*\sigma$.

So, we first reveal the 3-core of $G_{\mr{rand}}^1$, and fix a set $V_{\mr{core}}$ of $n/2$ vertices in this 3-core. We then define the complementary set $U=V(G)\setminus V_{\mr{core}}$, and reveal all edges of $G_{\mr{rand}}$ except those between $U$ and $V_{\mr{core}}$. For the rest of the proof we work with the remaining random edges in $G_\mr{rand}^2$ between $U$ and $V_{\mr{core}}$.

Specifically, whenever a vertex in $U$ has at least three neighbours in $V_{\mr{core}}$, that vertex is guaranteed to be in the 3-core of $G_{\mr{rand}}$, and therefore (whp) it is assigned a unique colour. Such vertices are removed from consideration in the disparity graph. To show that such deletions typically shatter the disparity graph into small components, we will use the following lemma.
\begin{lemma}\label{lem:graph random vertex removal}
     Let $G$ be an $n$-vertex graph with maximum degree at most $4\log n$. Let $G(q)$ be the graph obtained by removing each vertex with probability $q\ge 1-1/(\log n)^6$ independently. Then whp $G(q)$ does not have a component of size larger than $\log n/(4\log\log n)$.
\end{lemma}

We prove \cref{lem:graph random vertex removal} at the end of this section. First, we deduce the statement of \cref{thm:smoothed-disparity-cheap2}.

\begin{proof}[Proof of \cref{thm:smoothed-disparity-cheap2}]Let $G =G_{0} \triangle G_{\mr{rand}} $.
 First note that we can assume $p\le 2\log n/n$, as otherwise the desired result follows from \cref{thm:smoothed-colour-refinement} (note that if $\mc R^*\sigma$ assigns all vertices distinct colours, then $D(G,\mc R^*\sigma)$ consists only of isolated vertices.

Write $G_{\mr{rand}}=G_{\mr{rand}}^1\cup G_{\mr{rand}}^2\cup G_{\mr{rand}}^3$, where $G_{\mr{rand}}^1,G_{\mr{rand}}^2,G_{\mr{rand}}^3\sim \mb G(n,p')$ are independent random graphs with edge probability $p'$ satisfying $1-p=(1-p')^3$. Note that $p'\ge (100/3)\log \log n/n\ge 20\log \log n/n$.

Since $np'\to \infty$, the 3-core of $G_{\mr{rand}}^1$ has $n-o(n)$ vertices whp (see e.g.\ \cite[Exercise~2.4.14]{FK16}). So, whp we can fix a subset $V_{\mr{core}}\subseteq V_3(G_{\mr{rand}}^1)$ with size $n/2$. Let $U=V(G)\setminus V_{\mr{core}}$, and for $i\in \{2,3\}$ let $S_i$ be the set of vertices in $U$ which have at least three neighbours in $V_{\mr{core}}$ with respect to $G_{\mr{rand}}^i$. We use the sets $S_i$, together with the colour refinement algorithm, to recursively define a sequence of vertex-colourings $c_1,c_2,c_3$ of $U$, as follows.

\begin{itemize}
    \item Recall that $\sigma_{G[U]}$ is the trivial vertex-colouring of $G[U]$, and let $c_1=\mc R_{G[U]}^* \sigma_{G[U]}$.
    \item For $i\in\{2,3\}$: let $b_i$ be the vertex-colouring obtained from $c_{i-1}$ by giving all vertices in $S_i$ a unique colour, and let $c_i=\mc R^* b_i$.
\end{itemize}

The plan is to first reveal $G[U]$ and $c_1$ and $c_2$, using \cref{lem:degrees-log} to bound the maximum degree of $D(G[U],c_2)$, and to then reveal $c_3$, using \cref{lem:graph random vertex removal} to bound the sizes of the components of $D(G[U],c_3)$. The following claim justifies this plan.

\begin{claim}\label{claim:components-final2}
Whp the largest connected component in $D(G,\mc R^* \sigma_G)$ has at most twice as many vertices as the largest connected component in $D(G[U],c_3)$.
\end{claim}
\begin{claimproof}[Sketch proof of claim]
By \cref{prop:3core}, whp $\mc R^*\sigma$ assigns each vertex in the 3-core $V_{3}(G_{\mr{rand}})\supseteq V_{\mr{core}}\cup S_2\cup S_3$ a unique colour. We will see that the desired conclusion holds whenever this is the case.

Extend each $c_i$ to a colouring $c_i^{\mr{ext}}$ of $V(G)$, by assigning each vertex in $V_{\mr{core}}$ a unique colour. Note that $D(G,c_i^{\mr{ext}})$ has the same components as $D(G[U],c_i)$ (plus singleton components for each vertex of $V_{\mr{core}}$), so by \cref{lem:disparity-components} it suffices to show that $\mc R^* \sigma_G$ is a refinement of each $c_i^{\mr{ext}}$. This can be proved inductively, with three applications of \cref{fact:consistent}
(note that we are assuming $\mc R^*\sigma_G$ assigns each vertex a unique colour, so $G[U]$ can be partitioned into full colour classes).
\end{claimproof}

For the rest of the proof, we reveal all the edges in $G^1_{\mr{rand}}$ (which determines $U$), and reveal all the edges in $G[U]$. 
We assume that $|V_3(G^1_{\mr{rand}})|\ge n/2$ (so $V_{\mr{core}}$ and $U$ are well-defined).

Given \cref{claim:components-final2}, our objective is to use the remaining randomness of $G^2_{\mr{rand}}$ and $G^3_{\mr{rand}}$ (via the colourings $c_2$ and $c_3$) to prove that whp every component of $D(G[U],c_3)$ has at most $\log n/(2\log \log n)$ vertices.

First, note that the random sets $S_i$ are independent. For each $i\in \{2,3\}$, each vertex $v\in U$ is independently present in $S_i$ with probability
\begin{align*}
  1-\Pr(\on{Binomial}(n/2,p')\le 2)&\geq 1-3\binom{n/2}{2}(p')^2(1-p')^{n/2-2}
  \\& \geq 1-(np')^2e^{-np'/2}\geq 1-(20\log \log n)^2 e^{-10\log\log n}\ge 1-1/(\log n)^6.
\end{align*}

We now track how $D(G[U],c_i)$ evolves with $i$.

\medskip\noindent\textbf{Round 1.} Applying \cref{lem:degrees-log} with $H=G[U]$, $c=c_1$ and $c'=c_2$, we see that whp $D(G[U],c_2)$ has maximum degree at most $4\log n$. Reveal an outcome of $c_2$ such that this is the case (i.e., for the rest of the proof, all probabilities should be interpreted as being conditional on an outcome of $c_2$ with this property).

\medskip\noindent\textbf{Round 2.} Recall that each vertex is present in $S_3$ with some probability $q\ge 1-1/(\log n)^6$. So, $D(G[U],c_2)-S_3$ is obtained from $D(G[U],c_2)$ by deleting each vertex with probability $q$, independently. By \cref{lem:graph random vertex removal}, whp the connected components of $D(G[U],c_2)-S_3$ all have size at most $\log n/(4\log \log n)$. Reveal an outcome of $c_3$ such that this is the case.

Let $H'$ be the graph obtained from $G[U]$ by deleting all edges incident to vertices of $S_3$ (leaving those vertices isolated). So, apart from isolated vertices, $H'$ has the same components as $G[U]-S_3$, and these components therefore all have size at most $\log n/(4\log \log n)$. But note that $H'$ can be interpreted as a generalised disparity graph $G_L(G[U],c_3)$, so by \cref{lem:disparity-components}, the connected components of $D(G,c_3)$ all have size at most $\log n/(2\log \log n)$, as desired.
\end{proof}

Now we prove \cref{lem:graph random vertex removal}.
\begin{proof}[Proof of \cref{lem:graph random vertex removal}]
Say that a vertex of $G$ is \emph{dead} if it is not in $G(q)$; otherwise it is \emph{alive}.

Fix $v\in V(G)$. We wish to show that the component of $v$ in $G(q)$ (if $v$ is alive) has at most $\log n/(4\log \log n)$ vertices, with probability at least $1-o(n^{-1})$ (the desired result will then follow from a union bound over $v$).

We explore the component of $v$ in $G(q)$ using breadth-first search, revealing whether vertices are dead or alive as we go. Specifically, let $S^0=\{v\}$ (we can assume $v$ is alive, or there is nothing to prove) and mark $v$ as ``explored''. For each $t\ge 1$ (in increasing order), let $T^t$ be the set of neighbours of vertices in $S^{t-1}$ which have not yet been marked as ``explored'', let $S^t$ be the set of alive vertices in $T^t$, and mark all vertices in $T^t$ as ``explored''. This process terminates when $S^t=\emptyset$.

Note that $S=\bigcup_{t=0}^\infty S^t$ is precisely the connected component of $v$ in $G(q)$. Since each vertex in $G$ has degree at most $4\log n$, we have $|T^t|\le (4\log n)|S^{t-1}|$ for all $t\ge 1$, so writing $T=\bigcup_{t=0}^\infty T^t$ we have $|T|\le (4\log n) |S|$.

Just like in the proof of \cref{lem:expansion_SLi}, we can couple our exploration process with an infinite sequence of $\on{Bernoulli}(1-q)$ random variables $(\alpha_i)_{i\in \mb N}$: whenever we need to know the alive/dead status of a vertex $v$, we look at the next $\alpha_i$ in our sequence. In particular, $|S|-1$ is the number of times we see $\alpha_i=1$ in the first $|T|$ coin flips, so the probability that $|S|\ge \log n/(4\log \log n)$ (i.e., the component of $G$ has at most $\log n/(4\log \log n)$ vertices) is at most the probability that for some $s\ge \log n/(4\log \log n)$, among the first $(4\log n)s+1$ of the $\alpha_i$ we see $\alpha_i=1$ at least $s$ times. This probability is at most
\begin{align*}
         \sum_{s=\frac{\log n}{4\log\log n}} \binom{4s\log n+1}{s}(1-q)^s
    \leq \sum_{s=\frac{\log n}{4\log\log n}} (5\log n)^s (1-q)^s
    &= \sum_{s=\frac{\log n}{4\log\log n}} \left(\frac{5}{(\log n)^5}\right)^s
    \\&=  e^{-(5/4+o(1)) \log n}=o(n^{-1}).
\end{align*}
The desired result follows.
\end{proof}

\end{document}